\documentclass[12pt]{amsart}

\usepackage[left=1in, right=1in, top=.75in, bottom=.75in]{geometry} 
\newcommand{\Z}{\mathbb{Z}}
\newcommand{\R}{\mathbb{R}}

\newcommand{\CPbar}[1]{ \overline{\mathbb{CP}^2}}
\newcommand{\CPsum}[2]{#1 \mathbb{CP}^2 \# #2 \overline{\mathbb{CP}^2}}
\newcommand{\Fibersumd}[3]{#1 \#_{#3} #2} 
\newcommand{\set}[1]{\{#1\}} 

\newcommand{\Mod}{\operatorname{Mod}}

\usepackage{amssymb, enumerate}
\usepackage{amsmath, eqnarray, amsthm}
\usepackage{graphicx} 
\usepackage{wrapfig} 
\graphicspath{ {./images/} }
\usepackage{rotating}
\usepackage{bm}
\usepackage{verbatim}
\usepackage{xcolor}
\usepackage{hyperref}  
\usepackage{indentfirst}
\usepackage{float}
\usepackage{tikz-cd}

\usepackage{soul}\usepackage[T1]{fontenc}
\usepackage{mathtools}
\DeclareMathAlphabet{\mathbbmsl}{U}{bbm}{m}{sl}

\newtheorem{theorem}{Theorem}
\newtheorem{corollary}{Corollary}
\newtheorem{lemma}{Lemma}
\newtheorem{proposition}{Proposition}

\theoremstyle{definition}
\newtheorem{rmk}{Remark}
\newtheorem{definition}{Definition}[section]

\usepackage{pgfplots} 
\pgfplotsset{compat=1.18} 
\usetikzlibrary{patterns}
\usetikzlibrary{intersections}
\usetikzlibrary{fillbetween}
\usetikzlibrary{arrows.meta}

\title[Irreducible 4-manifolds with $\mathbf{\pi_1=\Z_2}$]
{Geography of irreducible 4-manifolds with order two fundamental group}

\author[Mihail Arabadji]{Mihail Arabadji}
\address{Department of Mathematics and Statistics, University of Massachusetts, Amherst, MA 01003-9305, USA}
\email{marabadji@umass.edu}

\author[Porter Morgan]{Porter Morgan}
\address{Department of Mathematics and Statistics, University of Massachusetts, Amherst, MA 01003-9305, USA}
\email{pamorgan@umass.edu}

\begin{document}

\begin{abstract}
	Let $R$ be a closed, oriented topological $4$--manifold whose Euler characteristic and signature are denoted by $e$ and $\sigma$. We show that if $R$ has order two $\pi_1$, odd intersection form, and $2e + 3\sigma\geq 0$, then for all but seven $(e,\sigma)$ coordinates, $R$ admits an \textit{irreducible} smooth structure. We accomplish this by performing a variety of operations on irreducible simply--connected $4$--manifolds to build $4$--manifolds with order two $\pi_1$. These techniques include torus surgeries, symplectic fiber sums, rational blow-downs, and numerous constructions of Lefschetz fibrations, including a new approach to equivariant fiber summing.
\end{abstract}
\maketitle

\section{Introduction}

Broadly speaking, this paper is concerned with four-dimensional exotica. Like many papers on the topic, our motivating question is the following: Given a smooth $4$--manifold $M$, does there exist a smooth manifold $N$ which is homeomorphic to $M$, but not diffeomorphic to it? Such a manifold $N$ would be an {\it exotic copy} of $M$. This question has been central to four--dimensional topology in recent decades. In our efforts, we seek to answer it with the following constraints; first, we start with a closed, oriented, smooth $4$--manifold $M$ having order two fundamental group and an odd intersection form. Second, we ask that $N$ be {\it irreducible}, meaning that if $N=A\#B$, then one of $A$ or $B$ is a homotopy $4$--sphere. In this case, we call $N$ an {\it irreducible copy} of $M$. If $M$ and $N$ are exotic copies of each other with distinct Seiberg-Witten invariants, one can argue that $N\# \CPbar{}$ is an exotic copy of $M\# \CPbar{}$. This gives a fairly straightforward approach to finding reducible exotic smooth structures. An irreducible copy, which cannot be achieved by blow-up, is harder to come by, and all the more exciting.

Research on the topic of irreducible exotic structures took off in the 90's and early 2000's, where much progress was made in finding irreducible copies of simply--connected $4$--manifolds. One culminating result in this search was \cite[Theorem A]{ABBKP}, which found a symplectic irreducible copy of $\CPsum{a}{b}$ for a large region of the plane $\{(a,b)\in \Z_{>0}\times \Z_{>0} : 4+5a\geq b, a \text{ is odd}\}$. In the simply--connected case we use the following symplectic results to prove irreducibility; under mild $\pi_1$ conditions, minimal symplectic $4$--manifolds are irreducible \cite{Hamilton_2006}. By the work of Taubes, all minimal, symplectic, non-ruled $4$--manifolds have a non-negative Chern number $c_1^2$ \cite{Taubes_c_positive}. This and the fact that simply--connected symplectic manifolds have odd $b_2^+$ explain the constraints on $(a,b)$ in \cite[Therem A]{ABBKP}. Four-dimensional topologists have made great strides finding minimal symplectic simply--connected $4$--manifolds, from which we often obtain other irreducible copies using torus surgery, see e.g. \cite{FS_knots}, \cite{FPS}. By contrast, we know very little about irreducible structures on simply--connected $4$--manifolds which are not symplectic.

More recently, there have been notable results on non-simply-connected $4$--manifolds admitting exotic structures. The work of Baykur, Stipsicz, and Szabó \cite{baykur2024smooth} shows that if $M$ is oriented, closed, smooth, has odd intersection form and order two $\pi_1$, unless $b_2^+(M)=b_2^-(M)\leq 7$, there are infinitely many exotic copies of $M$. To show this, they find infinite families of irreducible, distinct smooth structures for each such $M$ with signature $-1$ and $0$, then extend the results to arbitrary negative signature using the blow-up method described above. By reversing orientation, the results immediately apply to manifolds with positive signature as well. This leaves open the question of whether irreducible structures exist when $|\sigma(M)|>1$. We answer this question affirmatively when $c_1^2(M)\geq 0$, and emphasize that there are no examples of irreducible $4$--manifolds $M$ with $b_1=0$ and $c_1^2(M) <0$ known to date. 

\begin{theorem}
	Let $R_{m,n}$ be the closed, oriented, smoothable $4$--manifold with $b_2^+=m$, $b_2^-=n$, odd intersection form, and order two fundamental group.
	Let $\mathcal R$ be the region \[
	\mathcal R =\{(m,n) \in \Z_{> 0}\times \Z_{> 0}: 4+5m\geq n \text{ and } 4+5n\geq m\}.
	\] Unless $m=n\leq 7$, $R_{m,n}$ admits an irreducible smooth structure.
\label{thm:main_thm}
\end{theorem}

It is often helpful to think of a smoothable $R_{m,n}$ more specifically as $L_2\#\CPsum{m}{n}$, where $L_2$ is a rational homology $4$-sphere with order two $\pi_1$. Our approach to proving Theorem \ref{thm:main_thm} is to employ an array of methods that realize irreducible copies of $R_{m,n}$ for large families of $(m,n)$--coordinates in $\mathcal R$. The process of populating a large plane of integral points with exotic copies has historically been called the \textit{geography problem}. With this terminology, our introduction so far can be summarized as follows: \cite{ABBKP} and \cite{baykur2024smooth} addressed the irreducible $\pi_1=1$ geography problem and the $\pi_1 =\Z_2$ geography problem, respectively. Here we address the \textit{irreducible} $\mathit{\pi_1=\mathbbmsl{Z}}_2$ geography problem.

\subsubsection*{Summary of techniques} To populate $\mathcal R$, we generally start with a simply--connected irreducible $4$--manifold $M$, then perform some operation to get another manifold which is closely related to $M$, but has order two $\pi_1$. The simplest of these operations is torus surgery, whereby we remove an embedded torus in $M$ and reglue it in such a way that the fundamental group changes. Torus surgery preserves Euler characteristic $e$ and signature $\sigma$; in our case, the resulting manifold will have the same $b_2^{\pm}$ numbers as $M$. Using this method, we produce symplectic manifolds, and we get irreducible copies $R_{m,n}$ for $m$ odd only. When $m$ is even, we often proceed by realizing $R_{m,n}$ as the quotient of an irreducible simply-connected manifold $M$ by an orientation--preserving free involution. This process is called the {\it $\mathbbmsl{Z}_2$--construction}, and was introduced in \cite{baykur2024smooth}. The $\Z_2$--construction will often increase $b_2^+$ and $b_2^-$ by one, and so realizes many of our $(m,n)$--coordinates with $m$ even. Although it is often straightforward to cook up a manifold $M$ admitting such an involution, showing $\pi_1(M)=1$ can be a challenging computation. This verification posed the greatest obstacle when using the $\Z_2$--construction, and many of our technical lemmas are devoted to it. Both torus surgery and the $\Z_2$--construction require an abundance of minimal, symplectic, simply--connected manifolds to operate and build on, for which we are indebted to \cite{ABBKP}. Since such simply--connected manifolds must have $c_1^2\geq 0$, our region $\mathcal R$ inherits this constraint. We are uncertain about the existence of irreducible copies of $R_{m,n}$ with negative $c_1^2$.

Torus surgery and the $\Z_2$--construction proved very effective when we performed them on minimal symplectic manifolds. However, realizing the coordinates close to the line $c_1^2 = 4+5b_2^+-b_2^-=0$ required starting with {\it non-minimal} manifolds. To realize each of these points, we take the double of a relatively minimal Lefschetz fibration given in \cite{BKS} to obtain an irreducible simply--connected Lefschetz fibration $M\to S^2$, where $c_1^2(M)=0$ and $M$ admits a free involution which reverses the fiber orientation. Then we perform a small number of rational blow--downs (or, equivalently, lantern substitutions on the corresponding positive factorization) in such a way that after the blow-downs, the resulting manifold still admits a free involution. Quotienting out by this involution gives a manifold with order two $\pi_1$ and small $c_1^2$ (in our case, $c_1^2\leq 4$). Unlike previous constructions in the literature, this method reaches points with even $b_2^+$ in addition to odd. We dub such a Lefschetz fibration the {\it fiber-reversing double}. This operation may be of interest in its own right to some readers.

We are by no means the first people to think about the irreducible $\pi_1=\Z_2$ geography problem. For instance, irreducible geography of $R_{a,b}$ was studied in \cite{Torres_2014} for $b_2^+$ odd, which was one of the culminating results for finding minimal symplectic manifolds with $\pi_1=\Z_2$. More recently, \cite{beke} and \cite{StipSzab2} provided infinite examples of irreducible, indefinite copies of $R_{a,b}$ with $b_2^+$ even. There is also progress in the definite territory, for instance \cite{levine} and \cite{StipSzab2}. Our approach to the irreducible $\Z_2$--geography offers a comprehensive treatment, in the sense that we give constructions for the entire geography plane defined in Theorem \ref{thm:main_thm} (excluding the work completed in \cite{baykur2024smooth}, which covers the lines $\sigma \in \{-1,0,1\}$ except for seven points). Our methods for low $c_1^2$ may be of particular interest, since these techniques have not appeared in the literature before. Our use of $\Z_2$--construction is also noteworthy, since this is a relatively new concept, and we demonstrate that it can be used on a variety of different $4$--manifolds.

\subsubsection*{Outline of paper} In Section 2, we provide the necessary background information about the topological classification of $4$--manifolds with cyclic $\pi_1$, Lefschetz fibrations, and operations on symplectic manifolds such as torus surgery and $\Z_2$--constructions. Section 3 develops some of the technical machinery that we need to construct irreducible copies of $R_{m,n}$. This includes careful $\pi_1$ calculations as well as a novel operation that we call a {\it fiber-reversing double} of a Lefschetz fibration. At that point, we are equipped to populate $\mathcal R$, and we do this in two chunks. Section 4 finds irreducible copies of $R_{m,n}$ when $m$ is even, and Section 5 finds irreducible copies when $m$ is odd. We note to the reader that our approach in \ref{sec:small_c} actually obtains irreducible copies of $R_{m,n}$ for even and odd $m$ simultaneously. In Section 6, we combine all points realized in Sections 4 and 5, then give some final constructions to complete the proof of Theorem 1.

\subsubsection*{Acknowledgements} The authors are grateful to their advisor, R. \.{I}nan\c{c} Baykur, for many helpful discussions. The authors were partially supported by the NSF grant DMS-2005327.

\section{Background}

\subsection{Topological classification of $4$-manifolds with $\mathbf{\pi_1\cong\Z_2}$.}

Let $L_2$ denote the spun of $\R P^3$, meaning $L_2$ is obtained by attaching $D^2\times S^2$ to $S^1\times (\R P^3-D^3)$ along $S^1\times \{pt\}\subset S^1\times \partial D^3$. $L_2$ is a rational homology $4$-sphere with fundamental group isomorphic to $\Z_2$, and is a quotient of $S^2 \times S^2$. A Kirby diagram of $L_2$ can be found in \cite[Figure 1]{baykur2024smooth}. Let $R_{a,b}=L_2\#\CPsum{a}{b}$. The aim of this paper is to construct a smooth, irreducible manifold $X_{a,b}$ homeomorphic to $R_{a,b}$. Because $X_{a,b}$ is irreducible and $R_{a,b}$ is not, such as a manifold is necessarily an exotic copy of $R_{a,b}$.

The works of Freedman and Hambleton-Kreck give a complete topological classification of $4$-manifolds with finite cyclic fundamental group. Given a closed, oriented $4$-manifold $M$ with $\pi_1=\Z_n$, the homeomorphism type of $M$ is determined by its Kirby-Siebenmann invariant, its intersection form, and its $w_2$-type \cite[Theorem C]{HambletonKreck}. We are only interested in $4$-manifolds admitting smooth structures, and those have trivial Kirby-Siebenmann invariants. On the other hand, the intersection form $Q_M$ of $M$ is determined by the rank, signature, and parity of $Q_M$. As for $w_2$-type, this may be defined in terms of the universal cover $\widetilde M$. $M$ has $w_2${\it -type (i)} if $M$ and $\widetilde M$ are both non-spin, $w_2${\it -type (ii)} if $M$ and $\widetilde M$ are both spin, and $w_2${\it -type (iii)} if $M$ is non-spin, and $\widetilde M$ is spin. When $\pi_1(M)\cong \Z_2$, there is a simple criterion for $M$ to have $w_2${\it -type (i)}, as the following proposition shows.

\begin{proposition}
	If $Q_M$ is odd and $\pi_1(M) \cong \Z_2$, then $M$ has the $w_2$-type (i).
\end{proposition}

\begin{proof}
	By \cite[p.~57]{HambletonKreck1988}, $M$ is stably equivalent to $\Sigma \# Z'$, where $\Sigma$ is a rational homology sphere, and $Z'$ is simply--connected. So there is a homeomorphism $M\# n (S^2\times S^2) \approx \Sigma \# Z$ for $Z=Z'\# n(S^2 \times S^2)$. Because $M$ is odd, $Z$ also must be odd. Take the double cover of both sides to obtain $\widetilde{M} \# 2n (S^2 \times S^2) \approx Z \# Z \# \widetilde{\Sigma}$. Then $\widetilde{M}$ is odd, and hence non-spin. So $M$ must be of $w_2$-type (i).
\end{proof}

This implies the following neat classification result which will be used in the proof of our main theorem after each construction.

\begin{proposition}
	Let $M$ be a closed, smooth, oriented $4$--manifold. If $\pi_1(M)\cong \Z_2$, $b_2^+(M)=a$, $b_2^-(M)=b$, and $Q_M$ is odd, then $M$ is homeomorphic to $R_{a,b}$.
	\label{prop:Rab_topological_type}
\end{proposition}

We prefer to fill the geography using $(b_2^+, b_2^-)$ coordinates. That said, one can also express the geography problem in terms of the algebraic invariants $(\sigma, e)$ or $(c_1^2,\chi_h)$. Due to the nature of the constructions, and the fact that we rely on other sources that use different coordinates, we will often need to go back and forth between coordinates. Hence, the following conversion formulas will be handy throughout this paper.

\begin{proposition}
If $X$ is a closed 4-manifold having Betti numbers $b_2^+, b_2^-$, and $b_1=0$, Euler characteristic $e$, signature $\sigma$, Chern number $c_1^2$, and holomorphic Euler characteristic $\chi_h$, then 
\begin{align*}
	b_2^+&=\frac{e-2+ \sigma}{2} = 2\chi_h-1 \\
	b_2^-&=\frac{e-2-\sigma}{2} = 10\chi_h -c_1^2-1 \\
	c_1^2&=2e+3\sigma = 4 + 5b_2^+ - b_2^- \\
	\chi_h&=\frac{1}{4}(e+\sigma) = \frac{1}{2} \left( 1+b_2^+ \right) \\
	e&=12\chi_h-c_1^2 = 2 + b_2^+ + b_2^-\\
	\sigma&=c_1^2-8\chi_h = b_2^+ - b_2^-.\\
\end{align*}

\end{proposition}

\subsection{Lefschetz fibrations and surgeries}

We recall the definition of a Lefschetz fibration.

\begin{definition}
	Let $X$ and $\Sigma$ be an oriented 4-manifold and surface, respectively. A smooth surjection $f\colon X\to \Sigma$ is called a Lefschetz fibration if it has finitely many critical points, and around each critical point there are orientation--preserving charts where $f$ is locally modelled by the map $g\colon\mathbb{C}^2\to \mathbb{C}$ defined as $g(z_1,z_2)=z_1z_2$. The genus of $f$ refers to the genus of a regular fiber.
\end{definition}

Given a Lefschetz fibration $f\colon X\to \Sigma$, there's a natural way to equip $X$ with a symplectic form such that fibers and sections are symplectic. There are a few notions of minimality relating to symplectic $4$--manifolds and Lefschetz fibrations that are important to distinguish. A smooth $4$--manifold $X$ is \textit{minimal} if it does not contain an smooth $(-1)$--spheres, i.e. an embedded sphere with self-intersection $(-1)$. If $S\subset X$ is an embedded surface, we say that the pair $(M,S)$ is {\it relatively minimal}  if $M\backslash \nu(S)$ is minimal, where $\nu(S)$ is a tubular neighborhood of $S$. On the other hand, a Lefschetz fibration is said to be {\it relatively minimal} if none of the fibers contain an embedded $(-1)$--sphere. The notion of relative minimality may seem equivocal, but the intended definition will be clear from context. Note that irreducibility implies minimality, since if $M$ is not minimal, then $M\cong N \# \CPbar{2}$ for some $N$ with the obvious exception when $M\cong \CPbar{2}$. 


Let $(M_1,\omega_1)$ and $(M_2,\omega_2)$ be symplectic $4$--manifolds where each $M_i$ admits an embedded genus $g$ symplectic surface $\Sigma_i$ such that $[\Sigma_i]^2 =0$. The normal bundles of $\Sigma_i$ induce $S^1$-bundles on $\partial \nu(\Sigma_i)$ over $\Sigma_i$. We can obtain a new symplectic manifold by gluing $M_1 \setminus \nu(\Sigma_1)$ and $M_2\setminus \nu(\Sigma_2)$ along the boundaries $\partial \nu(\Sigma_i)$ via an orientation--reversing bundle diffeomorphism. This new manifold is a {\it symplectic fiber sum} of $M_1$ and $M_2$, and its symplectic form resticts to $\omega_i$ on each $M_i\setminus \nu(\Sigma_i)$. We denote it by $M_1\#_{\Sigma_1=\Sigma_2} M_2$, or sometimes $M_1\#_\Sigma M_2$ when the gluing surfaces are clear from context. This notation suppresses the choice of gluing map, which will be specified separately when needed. The symplectic fiber sum of two manifolds along a genus $g$ surface has algebraic invariants given by $e(M_1\#_{\Sigma_g} M_2) = e(M_1)+e(M_2) +2-2g$ and $\sigma(M_1\#_{\Sigma_g} M_2) = \sigma(M_1)+\sigma(M_2)$. The following theorems provide criteria for when a symplectic fiber sum is irreducible. Symplectic fiber sums were introduced in \cite{GompfNew}, where the definition extends to the symplectic submanifolds with non-zero self-intersection, but for our purposes it suffices to take symplectic sums along square zero surfaces. The following theorems demonstrate that we can use symplectic fiber sums to build new irreducible manifolds.

\begin{theorem}[Usher \cite{usher2007minimality}]\label{thm:usher}
	Let\/ $X=Y\#_{\Sigma=\Sigma'}Y'$ be the symplectic sum, where the genus $g$ of $\Sigma$ and\/ $\Sigma'$ is positive.
	\begin{itemize}
		\item[(i)]  If either\/ $Y\setminus\Sigma$ or $Y'\setminus\Sigma'$ contains an embedded symplectic sphere of square $-1$, then $X$ is not minimal.
		\item[(ii)]  If one of the summands, say $Y$ for definiteness, admits the structure of an $S^2$-bundle over a surface of genus $g$ such that $\Sigma$ is a section of this $S^2$-bundle, then $X$ is minimal if and only if $Y'$ is minimal.
		\item[(iii)] In all other cases, $X$ is minimal.
	\end{itemize}
\end{theorem}

\begin{theorem}[Hamilton and Kotschick] \cite{Hamilton_2006}\label{thm:HK_minimal}
	Minimal symplectic\/ $4$-manifolds with residually finite fundamental groups are irreducible.
\end{theorem}

Finite groups are residually finite, so the fundamental group condition in Theorem \ref{thm:HK_minimal} is always met in this paper. Hence all minimal symplectic manifolds discussed here are irreducible. Next we describe a type of $4$--manifold particularly well-suited for symplectic fiber summing. The following family of symplectic manifolds were used extensively in \cite{ABBKP}.

\begin{definition}
	An ordered triple $(X,T_1,T_2)$ is a {\it telescoping triple} if $X$ is a symplectic $4$--manifold with disjoint embedded Lagrangian tori $T_1$ and $T_2$ such that the following conditions are met:\begin{enumerate}[(1)]
		\item $T_1$ and $T_2$ span a $2$-dimensional subspace of $H_2(X;\R)$.
		
		\item $\pi_1(X)\cong \Z^2$ and the meridian of each $T_i$ is trivial in $\pi_1(X - (T_1\cup T_2))$.
		
		\item The image of the homomorphism induced by inclusion $\pi_1(T_1)\to \pi_1(X)$ is a summand $\Z\subset \pi_1(X)$.
		
		\item The homomorphism induced by inclusion $\pi_1(T_2)\to \pi_1(X)$ is an isomorphism.
	\end{enumerate}
\end{definition}

The telescoping triple above is called minimal whenever $X$ is minimal. Because $T_1$ and $T_2$ above are homologically essential, they become symplectic after perturbing the symplectic form on $X$. If $(X,T_1,T_2)$ and $(X',T_1',T_2')$ are two telescoping triples, then for an appropriate gluing map the symplectic fiber sum $X\#_{T_2=T_1'} X'$ is again a telescoping triple by \cite[Proposition 3]{ABBKP}. Moreover, Usher's theorem shows that if $X$ and $X'$ are both minimal, so is the new telescoping triple $(X\#_{T_2=T_1'} X', T_1, T_2')$. In this way, telescoping triples are building blocks for constructing irreducible $4$--manifolds.

While symplectic fiber sums are effective for obtaining irreducible $4$--manifolds with various Euler characteristics and signatures, we introduce another operation which preserves these invariants but changes $\pi_1$. Let $X$ be a smooth $4$--manifold and $T$ an embedded torus with framing $\eta\colon \nu T\to T^2\times D^2$. Let $\lambda$ be the image of a push-off of a primitive curve on $T$ under $\eta$, and let $\mu_T$ be a meridian of $T$, i.e. a fiber of $\partial \nu(T)$. Then let $\phi\colon T^2\times S^1 \to \partial \nu(T)$ be a diffeomorphism which induces $[\{pt\}\times S^1]\mapsto p[\mu_T]+q[\lambda]$ on the first homology. The manifold $X'$ given by \[
	X' = X\setminus \nu(T) \cup_\phi T^2\times D^2,
	\] 
is the result of {\it $p/q$-torus surgery} performed on $X$ along $T$ with respect to $\lambda$. From Seifert-Van Kampen, we see that $\pi_1(X')$ is obtained from $\pi_1(X-T)$ by quotienting out the subgroup normally generated by $\mu_T^p\lambda_\eta^q$. The 4-manifold obtained from this construction admits symplectic structure whenever $X$ is symplectic, $T$ is Lagrangian, and $p=1$. In this case, the surgery is called {\it Luttinger surgery with coefficient $1/q$.}  When performing Luttinger surgery to change the fundamental group, we will make use of the following fact.

\begin{proposition}[\cite{Tian}, Proposition 3.1]
	Luttinger surgery preserves minimality.
	\label{prop:Luttinger_minimal}
\end{proposition}

As a simple motivating example, Proposition \ref{prop:Luttinger_minimal} shows that if $(X,T_1,T_2)$ is a minimal telescoping triple, then an irreducible simply-connected manifold is obtained through $\pm 1$--Luttinger surgeries along $T_1$ and $T_2$. Throughout our constructions, we will use a combination of symplectic fiber sums and Luttinger surgeries. Proposition \ref{prop:Luttinger_minimal} and Usher's theorem will show that the resulting $4$--manifold is irreducible.

\subsection{The $\mathbf{\Z_2}$-construction}\label{sec:Z2}

The following construction is introduced in \cite{baykur2024smooth}, which we review here. Let $Z$ be a closed, oriented, smooth 4-manifold and let $\Sigma$ be an embedded closed genus $g$ surface in $Z$ with a trivial normal bundle. Let $\eta \colon \nu\Sigma\to D^2\times \Sigma_g$ be a framing of the normal bundle for $\Sigma$. Then let $\phi$ be the involutive deck transform of the double cover $\Sigma_g\to \#_{g+1}\mathbb{RP}^2$. The map \[\Phi:= \eta|_{\partial(\nu\Sigma)}^{-1}\circ (\text{id}_{S^1}\times \phi)\circ \eta|_{\partial(\nu\Sigma)}\] is an orientation-reversing free involution on $\nu\Sigma$. The manifold \[\widetilde{X}:=(Z\setminus \nu\Sigma)\cup_{\Phi} (Z\setminus \nu\Sigma)\] is the {\it double} of $Z$ along $\Sigma$. Let $\tau$ be the orientation-preserving free involution on $\widetilde{X}$ that swaps the two copies of $Z\setminus \nu\Sigma$ and restricts to $\Phi$ on $\partial(\nu\Sigma)$. Then $X:=\widetilde{X}/\tau$ is a smooth, oriented, closed 4--manifold which we call the $\Z_2${\it -construction of $Z$ along $\Sigma$.}

\begin{proposition}\label{cor:Z2_minimal}
	Let $g\geq 1$. If $(M,\Sigma_g)$ is a symplectic, relatively minimal pair, then the $\Z_2$-construction of $M$ along $\Sigma_g$ is irreducible.
\end{proposition}

\begin{proof}
	Since the $\Z_2$--construction of $M$ along $\Sigma_g$ is double covered by $\widetilde X$, the double of $M$ along $\Sigma_g$, it suffices to show that $\widetilde X$ is minimal. For this, we appeal to Usher's theorem. Case (i) is ruled out because $(M,\Sigma_g)$ is relatively minimal.
	
	Case (ii) is ruled out by the following well-known fact: if  $M$ is an $S^2$-bundle over $\Sigma_g$, then it is minimal. For completeness, we sketch a short  argument. Let $M \cong N \# \CPbar{2}$ and let $\widetilde{M}$ denote the universal cover of $M$. On the one hand, $\pi_2(\widetilde{M})\cong\pi_2(M) \cong \Z$, with the latter isomorphism following from the long exact sequence for the fibration on $M$. On the other hand, $\pi_2(\widetilde{N \# \CPbar{2}}) \cong H_2(\widetilde{N \# \CPbar{2}}) $ by Hurewicz. Since  $\pi_1(N)\cong \pi_1(M) \cong \pi_1(\Sigma_g)$ is infinite, $\widetilde{N \# \CPbar{2}} \cong \widetilde{N} \# \left(\#_{n=1}^{\infty}\CPbar{2}\right)$. Consequently, $H_2(\widetilde{N \# \CPbar{2}})$ is infinitely generated. So  $\pi_2(\widetilde{N \# \CPbar{2}}) $ is also infinitely generated, whereas $\pi_2(\widetilde{M})$ is cyclic.
	
\end{proof}

The following lemma shows how algebraic invariants change when the $\Z_2$--construction is performed. The proof is a matter of simple algebraic computation and is omitted.

\begin{lemma}\label{lemma:inv_computations}
	Let $M$ be a $4$--manifold with $b_1(M)=0$ and let $\Sigma$ an embedded genus $g$ surface. Then taking a double/quotient/both along $\Sigma\subset M$ transforms the invariants of $M$ as follows, respectively:
	\begin{align*}
		\textit{Double} && \textit{Quotient} && \textit{$\Z_2$-construction}\\  
		\sigma \mapsto 2\sigma && \sigma \mapsto \sigma/2 && \sigma \mapsto \sigma \\
		e \mapsto 2e + 4g- 4 && e \mapsto e/2 && e \mapsto e+2g-2\\
		c_1^2 \mapsto 2c_1^2 +8g-8 && c_1^2\mapsto c_1^2/2 &&c_1^2 \mapsto  c_1^2 +4g-4 \\
		\chi_h \mapsto 2\chi_h +g-1 && \chi_h \mapsto \chi_h /2 && \chi_h \mapsto \chi_h + \frac{g-1}{2}\\
		b_{\pm} \mapsto 2b_{\pm} + 2g-1 && b_{\pm} \mapsto \frac{b_{\pm}-1}{2} && b_{\pm} \mapsto b_{\pm} + g-1\\ 
	\end{align*}
\end{lemma}

We often apply the $\Z_2$--construction in the hopes of getting a 4-manifold with a fundamental group of order two and an odd intersection form. The following proposition gives easy criteria for when we have succeeded.

\begin{proposition}\label{prop:z2_pi1_odd}
	Let $X$ be the $\Z_2$-construction of $M$ along $\Sigma$, where $M$ is simply-connected. If the meridian of $\Sigma$ is trivial in $\pi_1(M-\Sigma)$, then $\pi_1(X)\cong \Z_2$. If $M$ contains an immersed surface $F$ such that $[F]\cdot [\Sigma]=0$ and $[F]\cdot [F]$ is odd, then $X$ has odd intersection form.
\end{proposition}

\begin{proof}
	The statement about $\pi_1(X)$ is immediate. If $F$ is disjoint from $\Sigma$, then $F$ survives as an immersed surface in $X$, making $Q_X$ odd. If not, using the condition $[F]\cdot [\Sigma]=0$, we can construct $F'$ by pairing up positive and negative geometric intersections of $F$ and $\Sigma$, and resolving each pair by removing the neighborhoods of $F$ around each point and gluing cylinders $S^1 \times [0,1]\subset S^1\times \Sigma \cong \partial \nu \Sigma, $ where the image of $[0,1]$ connects intersection points. Then $F'$ will be disjoint from $\Sigma$ and have the same self--intersection as $F$.
\end{proof}

\section{Preliminary constructions and computations}

\subsection{Fiber-reversing double of Lefschetz fibrations} \label{section:fiber_reversing_double}

Some of our constructions involve performing the $\Z_2$-construction to $4$--manifolds admitting Lefschetz fibrations. As such, we assume basic familiarity with the notions of Lefschetz fibrations and their associated positive factorizations in $\text{Mod}(\Sigma_g)$. For a more thorough exposition, see \cite[Chapter 8]{GompfStip} or \cite{Bay_Kork}. Given the Lefschetz fibration $\pi\colon M\to S^2$, we may wonder to what extent the $\Z_2$--construction is compatible with the structure coming from $\pi$. That is, if we take the double of $M$ along a fiber $F$ of $\pi$, as described in \ref{sec:Z2}, does this double admit a Lefschetz fibration? If the gluing map were to preserve the orientation of $F$ and reverse the orientation of the $S^1$ factor, this would be an example of an ordinary fiber sum between Lefschetz fibrations. In this case, it is easy to see that the resulting manifold admits a Lefschetz fibration whose positive factorization is a concatenation of the factorizations from each summand. In the proposition below, we investigate when the fiber gluing is orientation--reversing, as is the case when we perform the $\Z_2$--construction. We make a quick note about notation, that for mapping classes $\phi,\psi\in\Mod(\Sigma)$, we let $\psi^\phi$ denote the conjugation $\phi\psi\phi^{-1}$. We denote a right--handed Dehn twist about a curve $\gamma$ by $t_\gamma$.

\begin{proposition}
	Let $f_1\colon M_1\to D^2$ and $f_2\colon M_2 \to D^2$ be Lefschetz fibrations with positive factorizations of $W_1 := t_{x_1}t_{x_2}\cdots t_{x_k}$ and $W_2 := t_{y_1} t_{y_2}\cdots t_{y_{\ell}}$ in $\Mod(\Sigma_g)$. Let $r$ be an orientation-reversing diffeomorphism of $\Sigma_g$. There exists a Lefschetz fibration $f_1\cup f_2\colon M_1 \cup M_2\to D^2$ whose total space is diffeomorphic to the union of $M_i$ glued, restricting to $f_i$ on $M_i$, and which admits a positive factorization $t_{x_1}t_{x_2}\cdots 
	t_{x_k}t_{r(y_{\ell})}\cdots t_{r(y_2)}t_{r(y_1)} = W_1(W_2^{-1})^{r}$ in $\Mod(\Sigma_g)$.
	
	\label{prop:reverse_LF}
\end{proposition}
\begin{proof}
	Let $p\in D^2$ be a basepoint and $P=\set{p_1,p_2, \dots,  p_k}$ be the critical values of $f_1$. Let $F_1$ be a fiber over $p$, identified with $\Sigma_g$ via the orientation-preserving diffeomorphism $\alpha_1\colon \Sigma_g \to F_1$. Let $\gamma_{1},\gamma_{2},\dots, \gamma_{k}$ be loops in $\pi_1(D^2-P)$, so that the right-to-left concatenation $\gamma_{1}\ast \gamma_{2}\ast \cdots \ast \gamma_{k}$ is isotopic to $\partial D^2$. Similarly, let $q\in D^2$ be a basepoint and $Q=\{q_1,q_2,\cdots q_\ell\}$ be the critical values of $f_2$. Let $F_2$ be the fiber over $q$, with the identification $\alpha_2\colon \Sigma_g \to F_2$. The loops $\delta_i \in \pi_1(D^2-Q)$ are arranged so that $\delta_1 \ast \delta_2 \ast \cdots \ast \delta_{\ell}$ is isotopic to $\partial D^2$. Figure \ref{fig:reverse_LF1} shows these critical values and path systems for $f_1$ and $f_2$. The monodromy of $f_1$ is given by 
	$$\mu_1:=\alpha_1^* \circ \phi_1 \colon  \pi_1(D^2-P,p)\rightarrow \Mod(F_1) \rightarrow \Mod(\Sigma_g),$$ where $\alpha_1^*([\varphi])=[\alpha^{-1}_1\circ\varphi\circ\alpha_1]$ and $\mu_1$ maps $\gamma_i$ to $t_{x_i}$ \cite[8.1]{GompfStip}. Because we adopt the convention that both function composition and path concatenation are read {\it right--to--left}, the monodromy is an honest homomorphism, so $\mu_1(\lambda_1\ast \lambda_2)=\mu_1(\lambda_1)\circ \mu_1(\lambda_2)$ for based loops $\lambda_1,\lambda_2$. The monodromy homomorphism for $f_2$ is $\mu_2:=\alpha_2^* \circ \phi_2 \colon  \pi_1(D^2-P)\rightarrow \Mod(F_2) \rightarrow \Mod(\Sigma_g),$ which is defined analogously.
	
	Let $\eta$ be a semi-circle arc on the right half of $\partial D^2$, colored green in Figure \ref{fig:reverse_LF1}. Then $f_1^{-1}(\eta)\cong \eta \times \Sigma_g$. We glue two copies of $D^2$ together by gluing $\eta\mapsto a(\eta)$, where $a$ is the antipodal map on $\partial D^2$. Let $\mathcal D= D^2\cup_{\eta = a(\eta)} D^2$. Then after identifying $f_1^{-1}(\eta)$ and $f_2^{-1}(\eta)$ with $\eta \times \Sigma_g$, we can glue $M_1$ and $M_2$ via the fiber-preserving map $(x,y)\mapsto (a(x),r(y))$ for $x\in \eta$ and $y\in \Sigma_g$. After this gluing, we have a well-defined map $f_1 \cup_{a\times r} f_2\colon  M_1\cup_{a\times r} M_2 \to \mathcal D$ which restricts to $f_i$ on $M_i$ (we abuse notation slightly by denoting the boundary gluing map by $a \times r$).
	
	Next we give a positive factorization in $\Mod(\Sigma_g)$ by constructing a path system on $\mathcal D$ with critical points $P\cup Q$. We first choose a basepoint, which will be the basepoint $p$ of the first component of $D^2\cup_{\eta = a(\eta)} D^2$. To define loops for the critical values in the second component $D^2$, we need to ``pull" the basepoint from $q$ to $p$ using a path $\beta\colon [0,1]\to \mathcal D$ from $p$ to $q$ as in Figure \ref{fig:reverse_LF1}. From the choice of gluing when forming $\mathcal D$, we see that after identifying $F_1 = (f_1\cup_{a\times r} f_2)^{-1}(p)$ and $F_2 = (f\cup_{a\times r} f)^{-1}(q)$ with $\Sigma_g$ by $\alpha_1^{-1}$ and $\alpha_2^{-1}$, the identification of $F_1$ with $F_2$ through a parallel transport along $\beta$ is isotopic to $r$. With this established, we define the path system on $\mathcal{D}$ with the following loops based at $p$ by

	\begin{figure}
	
	\begin{tikzpicture}

		
		\draw[thick] (0,0) circle (2.5);
		
		\fill (0,-1) circle (2pt) node[below right] {$p$}; 
		\draw[shift={(-1.05,0.8)}, thick] (-0.1,-0.1) -- (0.1,0.1) (-0.1,0.1) -- (0.1,-0.1) node[below, scale = 0.9] {$p_{1}$} node[above=10] {$\gamma_{1}$};
		\draw[shift={(-0.05,0.8)}, thick] (-0.1,-0.1) -- (0.1,0.1) (-0.1,0.1) -- (0.1,-0.1) node[below, scale = 0.9] {$p_{2}$} node[above=15] {$\gamma_{2}$};
		\draw[shift={(1.65,0.8)}, thick] (-0.1,-0.1) -- (0.1,0.1) (-0.1,0.1) -- (0.1,-0.1) node[below, scale = 0.9] {$p_{k}$} node[above=8] {$\gamma_{k}$};			
		\fill (0.5,0.8) circle (1.4pt); 
		\fill (0.8,0.8) circle (1.4pt); 
		\fill (1.1,0.8) circle (1.4pt); 
		
		
		\draw[thick, blue, rounded corners=10]
		(0,-1) .. controls  (2.2,-0.2) and (2.1,0.5).. (1.9,1.2) .. controls (1.2,0.8) and (1.4,0.5) ..(1.4,0)-- (0,-1);
		\draw[->,thick,blue] (1,-0.57) -- ++(28:0.01);
		\draw[thick, blue, rounded corners=15] (0,-1) --(0.4,0.5) -- (0, 1.5) -- (-0.4,0.5) -- (0,-1);
		\draw[->,thick,blue] (0.21,-0.2) -- ++ (75:0.01);
		\draw[thick, blue, rounded corners=10]
		(0,-1) .. controls  (-1.5,-0.2) and (-1.4,0.5).. (-1.2,1.2) .. controls (-0.5,0.8) and (-0.7,0.5) ..(-0.7,0)-- (0,-1);
		\draw[->,thick,blue] (-0.45,-0.35) -- ++(120:0.01);
		
		\draw[thick, green] (0,2.5) arc[start angle=90, end angle=-90, radius=2.5] node[right=1cm] {$\eta$}; 
		
		\fill (0,2.5) circle (2pt) node[above] {$A$}; 
		\fill (0,-2.5) circle (2pt) node[below] {$B$}; 

		\draw[thick, cyan] (0,-1) -- (1,-1)-- (2.28,-1);
		\draw[thick, cyan,->] (0,-1) -- (1.5,-1) node[above] {$\beta$};

		\begin{scope}[shift={(6,0)}]
			
			\draw[thick] (0,0) circle (2.5);
			
			\fill (0,-1) circle (2pt) node[below right] {$q$}; 
			\draw[shift={(-1.05,0.8)}, thick] (-0.1,-0.1) -- (0.1,0.1) (-0.1,0.1) -- (0.1,-0.1) node[below, scale = 0.9] {$q_1$} node[above=10] {$\delta_1$};
			\draw[shift={(-0.05,0.8)}, thick] (-0.1,-0.1) -- (0.1,0.1) (-0.1,0.1) -- (0.1,-0.1) node[below, scale = 0.9] {$q_2$} node[above=15] {$\delta_2$};
			\draw[shift={(1.65,0.8)}, thick] (-0.1,-0.1) -- (0.1,0.1) (-0.1,0.1) -- (0.1,-0.1) node[below, scale = 0.9] {$q_{\ell}$} node[above=8] {$\delta_{\ell}$};			
			
			\draw[thick, blue, rounded corners=10]
			(0,-1) .. controls  (2.2,-0.2) and (2.1,0.5).. (1.9,1.2) .. controls (1.2,0.8) and (1.4,0.5) ..(1.4,0)-- (0,-1);
			\draw[->,thick,blue] (1,-0.57) -- ++(33:0.01);
			\draw[thick, blue, rounded corners=15] (0,-1) --(0.4,0.5) -- (0, 1.5) -- (-0.4,0.5) -- (0,-1);
			\draw[->,thick,blue] (0.21,-0.2) -- ++ (75:0.01);
			\draw[thick, blue, rounded corners=10]
			(0,-1) .. controls  (-1.5,-0.2) and (-1.4,0.5).. (-1.2,1.2) .. controls (-0.5,0.8) and (-0.7,0.5) ..(-0.7,0)-- (0,-1);
			\draw[->,thick,blue] (-0.45,-0.35) -- ++(120:0.01);
			
			\fill (0.5,0.8) circle (1.4pt); 
			\fill (0.8,0.8) circle (1.4pt); 
			\fill (1.1,0.8) circle (1.4pt); 
			
			\draw[thick, green] (0,-2.5) arc[start angle=270, end angle=90, radius=2.5] node[left=1cm] {$a(\eta)$}; 
			
			\fill (0,2.5) circle (2pt) node[above] {$B$}; 
			\fill (0,-2.5) circle (2pt) node[below] {$A$}; 
			
			\draw[thick, cyan] (-2.28,1) .. controls (-1.5,-0.5) ..  (0,-1);
			\draw[->,thick,cyan] (-1.5,-0.3) -- ++(-45:0.01) node[below] {$\beta$};
		\end{scope}

	\end{tikzpicture}
	\caption{Based paths for the doubled fibration}
	\label{fig:reverse_LF1}
\end{figure}
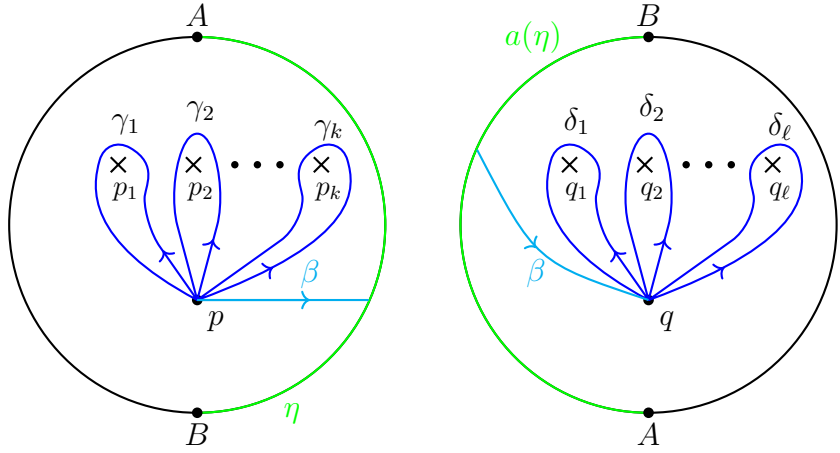
	
	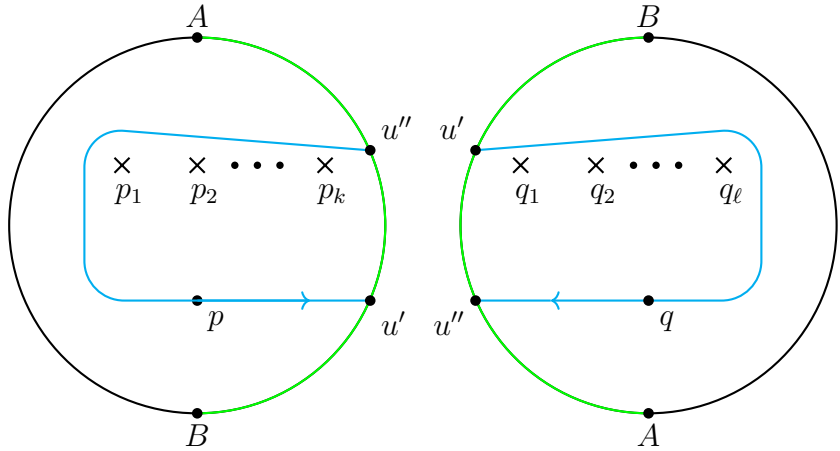
\begin{figure}
		
		\begin{tikzpicture}
			
			
			\draw[thick] (0,0) circle (2.5);
			
			\fill (0,-1) circle (2pt) node[below right] {$p$}; 
			\draw[shift={(-1,0.8)}, thick] (-0.1,-0.1) -- (0.1,0.1) (-0.1,0.1) -- (0.1,-0.1) node[below] {$p_{1}$};
			\draw[shift={(0,0.8)}, thick] (-0.1,-0.1) -- (0.1,0.1) (-0.1,0.1) -- (0.1,-0.1) node[below] {$p_{2}$};
			\draw[shift={(1.7,0.8)}, thick] (-0.1,-0.1) -- (0.1,0.1) (-0.1,0.1) -- (0.1,-0.1) node[below] {$p_{k}$};			
			
			\fill (0.5,0.8) circle (1.4pt); 
			\fill (0.8,0.8) circle (1.4pt); 
			\fill (1.1,0.8) circle (1.4pt);

			\draw[thick, green] (0,2.5) arc[start angle=90, end angle=-90, radius=2.5]; 
			
			\fill (0,2.5) circle (2pt) node[above] {$A$}; 
			\fill (0,-2.5) circle (2pt) node[below] {$B$}; 

			\draw[thick, cyan] (0,-1) --(2.28,-1);
			\draw[thick, cyan,->] (0,-1) -- (1.5,-1) ;
			\draw[thick, cyan,rounded corners=15, ->] (3.7,1) -- (7.5,1.3) -- (7.5,-1) -- (6,-1) -- (4.7,-1);
			\draw[thick,cyan] (4.7,-1)--(3.7,-1);
			\draw[thick, cyan, rounded corners=15] (2.28,1) -- (-1.5,1.3) -- (-1.5,-1) -- (0,-1);

			\begin{scope}[shift={(6,0)}]
				
				\draw[thick] (0,0) circle (2.5);
				
				\fill (0,-1) circle (2pt) node[below right] {$q$}; 
				\draw[shift={(-1.7,0.8)}, thick] (-0.1,-0.1) -- (0.1,0.1) (-0.1,0.1) -- (0.1,-0.1) node[below] {$q_1$};
				\draw[shift={(-0.7,0.8)}, thick] (-0.1,-0.1) -- (0.1,0.1) (-0.1,0.1) -- (0.1,-0.1) node[below] {$q_2$};
				\draw[shift={(1,0.8)}, thick] (-0.1,-0.1) -- (0.1,0.1) (-0.1,0.1) -- (0.1,-0.1) node[below] {$q_{\ell}$};			
				
				\fill (-0.2,0.8) circle (1.4pt); 
				\fill (0.1,0.8) circle (1.4pt); 
				\fill (0.4,0.8) circle (1.4pt);

				\draw[thick, green] (0,-2.5) arc[start angle=270, end angle=90, radius=2.5]; 
				
				\fill (0,2.5) circle (2pt) node[above] {$B$}; 
				\fill (0,-2.5) circle (2pt) node[below] {$A$}; 
				
				\begin{scope}[shift = {(-6,0)}]
					\fill (2.3,-1) circle (2pt) node[below right] {$u'$};
					\fill (3.7,-1) circle (2pt) node[below left] {$u''$}; 
					\fill (3.7,1) circle (2pt) node[above left] {$u'$}; 
					\fill (2.3,1) circle (2pt) node[above right] {$u''$}; 
				\end{scope}
			\end{scope}
			
		\end{tikzpicture}
	\caption{Based path for the total monodromy}
	\label{fig:reverse_LF2}
	\end{figure}
	
	\begin{align*}
		u_i = 
		\begin{cases} 
			\gamma_i & \text{for } 1 \leq i \leq k, \\
			\beta^{-1} \ast \delta_{\ell+k+1-i}^{-1} \ast \beta & \text{for } k+1 \leq i \leq k+\ell.
		\end{cases}
	\end{align*}
	where path concatenation is right--to--left. One can observe from Figure \ref{fig:reverse_LF2} that $u_1\ast u_2\ast \cdots \ast u_{k+\ell}$ is isotopic to $\partial \mathcal{D}$. We can compute the images of these loops under the new monodromy homomorphism $\mu\colon \pi_1(\mathcal D- (P\cup Q),p)\to \Mod(\Sigma_g)$. The loops $u_i$ for $1\leq i\leq k$ have precisely the same monodromy as $\gamma_i$. Hence $\mu(u_i)= t_{x_i}$. On the other hand when $k+1\leq i \leq k+\ell$, $u_i$ is mapped to the conjugate of $\mu_2(\delta^{-1}_{\ell+k+1-i})=t_{y_{\ell+k+1-i}}^{-1}$ by the automorphism induced by $\beta$. Hence, $\mu(u_i) = [r] \circ t_{y_{\ell+k+1-i}}^{-1} \circ [r]^{-1}$, where $[r]$ is the mapping class group represented by $r$.
	
	As a final step, we claim that $[r]t_y^{-1}[r]=t_{r(y)}$. Since both classes have representatives supported in $\nu(r(z))$, this can be verified directly by computing the image of an arc between the two boundary components of $\nu(r(z))$ under both maps. Hence, we have a positive factorization $t_{x_1}t_{x_2}\cdots 
	t_{x_k}t_{r(y_{\ell})}\cdots t_{r(y_2)}t_{r(y_1)}$, which defines a Lefschetz fibration over $\mathcal D$.
	
\end{proof} 

 Whenever $W_1(W_2^{-1})^{r}=1 \in \Mod(\Sigma_g)$, the Lefschetz fibration extends over $S^2$, which will be the main use for our purposes. Note that it suffices for $W_i$ to be a positive factorization of identity, but this is not required. In addition, despite $f_i$ being Lefschetz fibration over a disk in the proposition, we take liberty to glue in the same manner fibrations defined over spheres by removing a regular fiber from each $f_i$.

\begin{definition}
	Whenever $f_1=f=f_2$ in Proposition \ref{prop:reverse_LF}, we call the Lefschetz fibration $f\cup_{a\times r} f \colon X\cup_{a\times r} X\to S^2$ and its total space the {\it fiber-reversing double} of $f$ and $X$, respectively.
	
	\label{def:fiber_reversing_double}
\end{definition}

The strength of the fiber-reversing construction is that it produces a new Lefschetz fibration over a sphere whenever you have a Lefschetz fibration over a disk interacting nicely with the gluing map $r$.

\begin{corollary}
	Let $f\colon X\to D^2$ be a genus $g$ Lefschetz fibration with a positive factorization $V$. Let $r$ be an orientation-reversing diffeomorphism of $\Sigma_g$. If $V$ and $[r]$ commute in $\Mod(\Sigma_g)$, then there exists a genus $g$ Lefschetz fibration $\tilde f\colon \widetilde X \to S^2$ with positive factorization $V (V^{-1})^{r}$, where $\tilde X$ is two copies of $X$ glued along their boundaries.
	
	\label{cor:nontrivial_reverse_LF}
	
\end{corollary}

In addition, if the diffeomorphism $r$ and a representative of $V$ commute on the nose, then the fiber-reversing double comes with a natural free involution.

\begin{corollary}
	Let $\tilde f\colon\tilde X \to S^2$ be a Lefschetz fibration as in Corollary \ref{cor:nontrivial_reverse_LF}. If $r$ is an involution and there exists a diffeomorphism $V_0$ representing $V$ such that $r$ and $V_0$ commute in $\text{Diff}(\Sigma_g)$, then $\tilde X$ admits an orientation-preserving free involution which interchanges two copies of $X$.
	
	\label{cor:reverse_LF_involution}
	
\end{corollary}

\begin{proof}
	From inspecting the construction of $\tilde f$ given above, it's clear that $\widetilde X$ is two copies of $f\colon X\to D^2$ glued together. What is less clear is how the gluing works along $\partial X$, since the boundary is not necessarily a product space, but is diffeomorphic to the mapping torus $MT(\Sigma_g,V_0)=\R\times \Sigma_g / (t,p)\sim (t-2\pi,V_0(p))$. Identifying $\partial X$ with this mapping torus, let $\psi\colon \partial X \to \partial X$ be defined $\psi(t,p)=(t+\pi,r(p)).$ First we verify that $\psi$ is well-defined. Indeed,\[
		\psi(t-2\pi,V_0(p)) = (t-\pi,r\circ V_0(p)) = (t-\pi,V_0\circ r(p)) \sim (t+\pi,r(p)) = \psi(t,p).
	\]
	Moreover, $\psi$ induces the antipodal map on $S^1$ and $r$ on the $\Sigma_g$ fiber. That is, the diagram
	
	\begin{center}
		\begin{tikzcd}
			\Sigma_g \arrow[r,hook]\arrow[d,"r"] &\partial X \arrow[r,two heads]\arrow[d,"\psi"] & S^1\arrow[d,"a"]\\
			\Sigma_g \arrow[r,hook] &\partial X \arrow[r,two heads] & S^1\\ 
		\end{tikzcd}
	\end{center}
	commutes, where the left horizontal maps are inclusions of the fibers. This verifies that $f\cup_\psi f$ is a Lefschetz fibration with a positive factorization of $V (V^{-1})^{r}$, so $X\cup_\psi X\cong \widetilde X$.
	
	Now consider the involution $i\colon \widetilde X\to \widetilde X$ which identically interchanges the two copies of $\text{Int}(X)$ and restricts to $\psi$ along $\partial X$. The fact that $r^2=1$ ensures that $i|_{\partial X}$ is an involution, and the fact that $\psi$ descends to the antipodal map on $S^1$ ensures that $i|_{\partial X}$ is free. Hence, $\widetilde X$ admits the desired involution.
\end{proof}

\subsection{Product of surfaces minus many tori.}

In many of our constructions, we use Luttinger surgery to reduce a $4$--manifold's fundamental group to be either trivial or of order two. For each Luttinger surgery that we perform on a $4$--manifold $M$, the resulting manifold's fundamental group is a quotient of $\pi_1(M-\nu(T^2))$. If we perform multiple Luttinger surgeries simultaneously, we need to understand the fundamental group of $M-\nu(N)$, where $N$ is a collection of disjoint tori. This section provides normal generators for a product $\Sigma_k\times \Sigma_2$ after certain tori are excised, which will be used when we perform simultaneous Luttinger surgeries in Section \ref{sec:sigma_3}. Although there are very similar $\pi_1$--computations in the literature, e.g. \cite{baykur2024smooth}, \cite{BK07}, \cite{ABBKP}, \cite{FPS}, and \cite{akhmedov2009exotic}, we wanted to present normal generators for {\it this particular set of tori}, which we will use later.

To begin, we introduce some notation. Let $F$ be an orientable genus $k$ surface. For now, assume $G$ is a punctured torus. Let $\{x_1,y_1,\dots x_k,y_k\}$ be the standard symplectic generators of $H_1(F;\Z)$ and let $\{a,b\}$ be the symplectic generators of $H_1(G;\Z)$. Set $p_i=x_i\cap y_i$ and $g:=a\cap b$. Let $X_i$, $Y_i$, $A$, and $A'$ be parallel copies of $x_i$, $y_i$, and $a$, so that $X_i\times A$ and $Y_i\times A'$ form a collection of $2k$ disjoint Lagrangian tori in $F\times G$ equipped with the product symplectic structure. Now let $\beta_i$ be an arc from a selected base point $f\in F$ to each $p_i$ so that $\tilde x_i:=\beta_i^{-1} x_i\beta_i$ and $\tilde y_i:=\beta_i^{-1} y_i\beta_i$ are loops based at $f$ disjoint from each $X_j\times A$ and $Y_j\times A'$. We slightly abuse notation and give the loops $\tilde x_i \times \{g\}$ and $\tilde y_i\times \{g\}$ in $F\times G$ the same names. Let $\tilde a = \{f\}\times a$ and $\tilde b = \{f\}\times b$. Figure \ref{fig:surface_times_torus} shows this collection of loops.

\begin{figure}
	\centering
	\begin{tikzpicture}
		\begin{scope}[scale=.95]
		
		\begin{scope}
			
			
			
			\draw[thick,color = blue, name path = x1] (-2,0) circle [radius =1.1cm] (-1,-.447) node [anchor = west,scale=.8] {$x_1$};
			\draw[thick,color = blue, name path = xk] (0,2) circle [radius = 1.1cm](-1.05,2) node [anchor = east,scale=.8] {$x_k$};
			\draw[thick, color = blue, name path = x2] (0,-2) circle [radius = 1.1cm] (-1,-2+.447) node [anchor = east,scale=.8] {$x_2$};

			\draw[thick,color=blue, name path = y1] (-2.4,0) .. controls +(0,-.15) and +(0,-.15)  ..  (-4,0) (-4,0) node[scale=.8,anchor = east]{$y_1$};
			\draw[thick,color=blue,dashed] (-2.4,0) .. controls +(0,.15) and +(0,.15)  ..  (-4,0);
			
			\draw[thick,color=blue, name path = yk] (0,2.4) .. controls +(-.15,0) and +(-.15,0)  ..  (0,4) (0,4) node[scale=.8,anchor = south]{$y_k$};;
			\draw[thick,color=blue,dashed] (0,2.4) .. controls +(.15,0) and +(.15,0)  ..  (0,4);
			
			\draw[thick,color=blue, name path = y2] (0,-2.4) .. controls +(-.15,0) and +(-.15,0)  ..  (0,-4) (0,-4) node[scale=.8,anchor=north] {$y_2$};
			\draw[thick,color=blue,dashed] (0,-2.4) .. controls +(.15,0) and +(.15,0)  ..  (0,-4);

			\path [name intersections={of=x1 and y1,by=p1}];
			\path [name intersections={of=x2 and y2,by=p2}];
			\path [name intersections={of=xk and yk,by=pk}];
			
			
			\draw[thick, color = blue] (0,0)  .. controls +(-1,.25) and +(1,0) .. (-2,1.4) .. controls +(-1,0) and +(-.5,.5).. (p1)  (-2,1.35) node[anchor = south,scale=.8] {$\beta_1$};
			\draw[thick, color = blue] (0,0)  .. controls +(.25,1) and +(0,-1) .. (1.4,2) .. controls +(0,1) and +(.5,.5).. (pk)  (1.35,2) node[scale=.8, anchor = west] {$\beta_k$};
			\draw[thick, color = blue] (0,0)  .. controls +(.25,-1) and +(0,1) .. (1.4,-2) .. controls +(0,-1) and +(.5,-.5).. (p2)  (1.35,-2) node[anchor = west,scale=.8] {$\beta_2$};
			
			\draw[thick, color = red] (-2,0) circle [radius = .6cm] (-2,.6) node [anchor = south, scale =.8] {$X_1$};
			\draw[thick, color = red] (0,2) circle [radius = .6cm] (0,1.4) node [anchor = north, scale=.8] {$X_k$};
			\draw[thick, color = red] (0,-2) circle [radius = .6cm] (0,-1.4) node [anchor = south,scale=.8] {$X_2$};
			
			\draw[thick,color=red] (-2.2,-.346) .. controls +(.05,-.25) and +(.05,-.25) .. (-3.872,-1) (-3.872,-1)  node[anchor = east] {$Y_1$};
			\draw[thick,color=red,dashed] (-2.2,-.346) .. controls +(-.05,.25) and +(-.05,.25).. (-3.872,-1);
			
			\draw[thick,color=red,rotate=-90] (-2.2,-.346) .. controls +(.05,-.25) and +(.05,-.25) .. (-3.872,-1) (-3.872,-1)  node[anchor = south] {$Y_k$};
			\draw[thick,color=red,dashed,rotate=-90] (-2.2,-.346) .. controls +(-.05,.25) and +(-.05,.25).. (-3.872,-1);
			
			\draw[thick,color=red,rotate=90, yscale=-1] (-2.2,-.346) .. controls +(.05,-.25) and +(.05,-.25) .. (-3.872,-1) (-3.872,-1)  node[anchor = north] {$Y_2$};
			\draw[thick,color=red,dashed,rotate=90, yscale = -1] (-2.2,-.346) .. controls +(-.05,.25) and +(-.05,.25).. (-3.872,-1);
			
			\draw[very thick] (0,0) circle[radius = 4cm];
			\draw[very thick] (-2,0) circle [radius =.4cm];
			\draw[very thick] (0,2) circle [radius = .4cm];
			\draw[very thick] (0,-2) circle [radius = .4cm];
			\draw[fill] (2.5,0) circle [radius = 1.5pt] ;
			\draw[fill] (2.38,-.6) circle [radius = 1.5pt] ;
			\draw[fill] (2.38,.6) circle [radius = 1.5pt] ;
			
			\draw[fill,color = blue] (0,0) circle [radius = 2pt] (0,0) node [anchor = west] {$f$} ;
			\draw[fill, color = blue] (p1) circle [radius=1.5pt] (p1) node [anchor = north east, scale = .8] {$p_1$};
			\draw[fill, color = blue] (p2) circle [radius=1.5pt] (p2) node [anchor = north east, scale = .8] {$p_2$};
			\draw[fill, color = blue] (pk) circle [radius=1.5pt] (pk) node [anchor = south east, scale = .8] {$p_k$};
			
		\end{scope}

		\begin{scope}[xshift=9cm]
			
			\draw[thick] (-4,0) node[scale = 2] {$\times$};
			
			
			
			\draw[thick, color=blue, name path = b] (0,0) circle [radius = 1.25cm] (-1.25,0) node [anchor = east] {$b$};
			\draw[thick, color = blue, name path = a] (0,.75) .. controls +(-.25,0) and +(-.25,0) .. (0,3); 
			\draw[thick, color = blue, dashed] (0,.75) .. controls +(.25,0) and +(.25,0) .. (0,3) (0,3) node[anchor= south] {$a$};
			\path [name intersections={of=a and b,by=k}];
			\draw[fill, color = blue] (k) circle [radius=1.5pt] (k) node [anchor = south east] {$g$};
			
			
			\draw[thick, color = red] (.25, -.707) .. controls +(.25,+.05) and +(.25,0) .. (.75,-3) (.75, -3) node [anchor = north] {$A'$};
			\draw[thick, color = red, dashed] (.25, -.707) .. controls +(-.25,-.05) and +(-.25,0) .. (.75,-3);
			
			\draw[thick, color = red] (-.25, -.707) .. controls +(.25,-.05) and +(.25,0) .. (-.75,-2.93) (-.75,-3) node [anchor = north] {$A$};
			\draw[thick, color = red, dashed] (-.25, -.707) .. controls +(-.25,+.05) and +(-.25,0) .. (-.75,-2.93);

			\draw[very thick] (3,3) -- (0,3) .. controls +(-2,0) and + (0,2) .. (-3,0).. controls +(0,-2) and +(-2,0) .. (0,-3) -- (3,-3);
			\draw[very thick] (3,0) ellipse (.75 and 3);
			\draw[very thick] (0,0) circle [radius = .75cm];
			
		\end{scope}
		
		\end{scope}
		
	\end{tikzpicture}
	\caption{The curves on $F\times G\cong \Sigma_k\times (T^2-D^2)$. }
	\label{fig:surface_times_torus}
\end{figure}
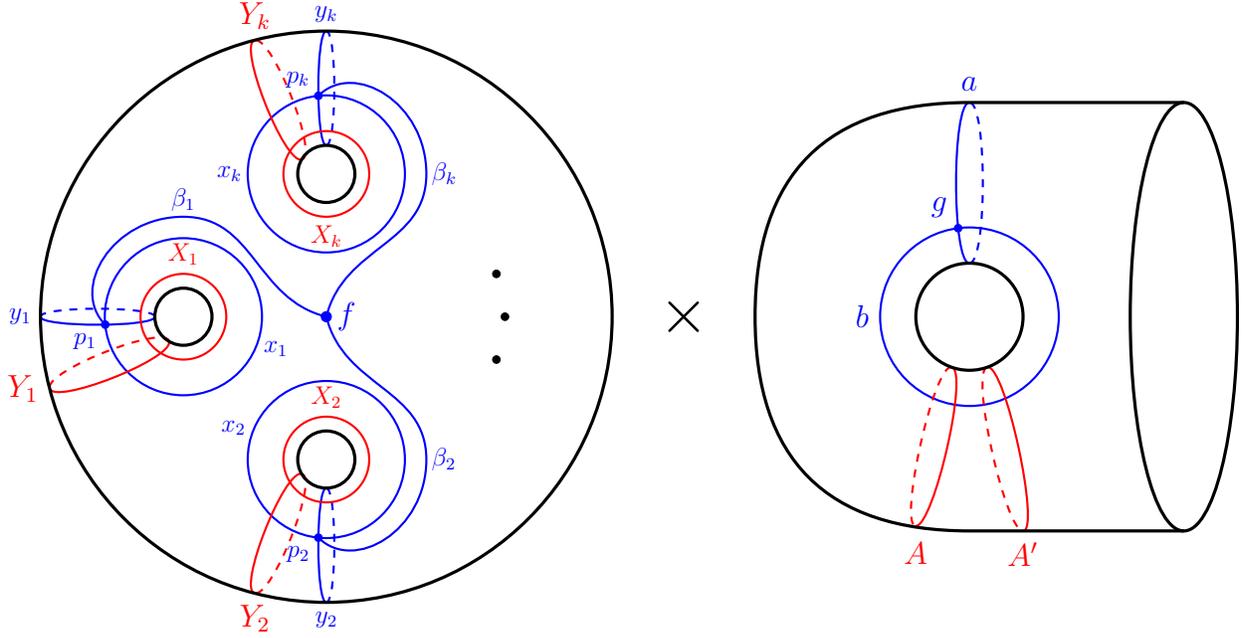

\begin{lemma}
	Let  $F$ be a closed surface with a positive genus $k$ and $G$ a punctured torus. Let $N$ be a union of neighborhoods of the following disjoint Lagrangian tori:\begin{align*}
		T_i:=&X_i\times A\\
		L_i:=&Y_i\times A'.
	\end{align*} Then $\pi_1(F\times G-N,(f,g))$ is normally generated by $\{\tilde x_1,\tilde y_1,\cdots , \tilde x_k,\tilde y_k, \tilde a,\tilde b\}$. Moreover, the curves generating the fundamental group of the boundary of the neighborhood of each Lagrangian torus in $N$ are as follows:\begin{align}
		\label{bd4} \mu_{T_i} &= [\tilde b^{-1},\tilde y_i^{-1}] && m_{T_i} = \tilde x_i &&  \ell_{T_i} =\tilde a\\
		\label{bd5} \mu_{L_i} &= [\tilde x_i^{-1},\tilde b] && m_{L_i} = \tilde y_i && \ell_{L_i} =\tilde b\tilde a\tilde b^{-1}.
	\end{align}
	\label{lem:sigma_k_times_punctured_torus}
\end{lemma}

\begin{proof}
	We proceed by induction. If the genus of $F$ is 1, then the lemma follows immediately from \cite[Theorem 2]{BK07}. Next, assume $F$ has genus $k$, and decompose it as $F=\Sigma_{k-1}^1 \cup \Sigma_1^1.$  Assume $f$ lies along the $S^1$ boundary where these two surfaces are glued together. Then \[F\times G- N = \left(\Sigma_{k-1}^1\times G - N_{k-1} \right) \cup \left( \Sigma_1^1\times G - (T_k\cup L_k) \right),\] 
	where $N_{k-1}=\cup_{i=1}^{k-1}(T_i\cup L_i).$ By the inductive hypothesis, $\set{x_i,y_i,a,b}_{1\leq i \leq k-1}$ normally generate $\pi_1(\Sigma_{k-1}\times G - N_{k-1}, (f,g))$. Since the meridian $\mu_G=\Pi_{i=1}^{k-1} [x_i,y_i]$, this collection also normally generates $\pi_1(\Sigma_{k-1}^1\times G - N_{k-1},(f,g))$. On the other hand, again by \cite[Theorem 2]{BK07} $\set{x_k,y_k,a,b}$ normally generate $\pi_1(\Sigma_1^1\times G - (T_k\cup L_k),(f,g))$. Then by Seifert-Van Kampen $\set{x_i,y_i,a,b}$ normally generate $F\times G- N$. Note that  \eqref{bd4} and \eqref{bd5} still hold in the $F\times G- N$ because they hold in the subspace $\Sigma_1^1 \times G - (T_k\cup L_k)$ and by the inductive hypothesis.
	\end{proof}

Now let $G$ be a closed genus two surface with symplectic generators $\{a_1,b_1,a_2,b_2\}$. Let $q_i=a_i\cap b_i$ and let $\gamma_i$ be an arc from a chosen base point $g$ to $q_i$. In the following proposition, we need one additional push-off of the curve $y_1$ on $F$ that is disjoint from $Y_1$. We label this new push-off $Y_1'$.  Figure \ref{fig:F_times_G} shows this collection of loops.

\begin{figure} 
	\centering
	\begin{tikzpicture}
		
		\begin{scope}[scale=.95]
		
		\begin{scope}
			
			
			
			\draw[thick,color = blue, name path = x1] (-2,0) circle [radius =1.1cm] (-1,-.447) node [anchor = west,scale=.8] {$x_1$};
			\draw[thick,color = blue, name path = xk] (0,2) circle [radius = 1.1cm](-1.05,2) node [anchor = east,scale=.8] {$x_k$};
			\draw[thick, color = blue, name path = x2] (0,-2) circle [radius = 1.1cm] (-1,-2+.447) node [anchor = east,scale=.8] {$x_2$};

			\draw[thick,color=blue, name path = y1] (-2.4,0) .. controls +(0,-.15) and +(0,-.15)  ..  (-4,0) (-4,0) node[scale=.8,anchor = east]{$y_1$};
			\draw[thick,color=blue,dashed] (-2.4,0) .. controls +(0,.15) and +(0,.15)  ..  (-4,0);
			
			\draw[thick,color=blue, name path = yk] (0,2.4) .. controls +(-.15,0) and +(-.15,0)  ..  (0,4) (0,4) node[scale=.8,anchor = south]{$y_k$};;
			\draw[thick,color=blue,dashed] (0,2.4) .. controls +(.15,0) and +(.15,0)  ..  (0,4);
			
			\draw[thick,color=blue, name path = y2] (0,-2.4) .. controls +(-.15,0) and +(-.15,0)  ..  (0,-4) (0,-4) node[scale=.8,anchor=north] {$y_2$};
			\draw[thick,color=blue,dashed] (0,-2.4) .. controls +(.15,0) and +(.15,0)  ..  (0,-4);

			\path [name intersections={of=x1 and y1,by=p1}];
			\path [name intersections={of=x2 and y2,by=p2}];
			\path [name intersections={of=xk and yk,by=pk}];
			
			
			\draw[thick, color = blue] (0,0)  .. controls +(-1,.25) and +(1,0) .. (-2,1.4) .. controls +(-1,0) and +(-.5,.5).. (p1)  (-2,1.35) node[anchor = south,scale=.8] {$\beta_1$};
			\draw[thick, color = blue] (0,0)  .. controls +(.25,1) and +(0,-1) .. (1.4,2) .. controls +(0,1) and +(.5,.5).. (pk)  (1.35,2) node[scale=.8, anchor = west] {$\beta_k$};
			\draw[thick, color = blue] (0,0)  .. controls +(.25,-1) and +(0,1) .. (1.4,-2) .. controls +(0,-1) and +(.5,-.5).. (p2)  (1.35,-2) node[anchor = west,scale=.8] {$\beta_2$};
			
			\draw[thick, color = red] (-2,0) circle [radius = .6cm] (-2,.6) node [anchor = south, scale =.8] {$X_1$};
			\draw[thick, color = red] (0,2) circle [radius = .6cm] (0,1.4) node [anchor = north, scale=.8] {$X_k$};
			\draw[thick, color = red] (0,-2) circle [radius = .6cm] (0,-1.4) node [anchor = south,scale=.8] {$X_2$};
			
			\draw[thick,color=red] (-2.2,-.346) .. controls +(.05,-.25) and +(.05,-.25) .. (-3.872,-1) (-3.872,-1)  node[anchor = east] {$Y_1$};
			\draw[thick,color=red,dashed] (-2.2,-.346) .. controls +(-.05,.25) and +(-.05,.25).. (-3.872,-1);
			
			\draw[thick,color=red,rotate=-90] (-2.2,-.346) .. controls +(.05,-.25) and +(.05,-.25) .. (-3.872,-1) (-3.872,-1)  node[anchor = south] {$Y_k$};
			\draw[thick,color=red,dashed,rotate=-90] (-2.2,-.346) .. controls +(-.05,.25) and +(-.05,.25).. (-3.872,-1);
			
			\draw[thick,color=red,rotate=90, yscale=-1] (-2.2,-.346) .. controls +(.05,-.25) and +(.05,-.25) .. (-3.872,-1) (-3.872,-1)  node[anchor = north] {$Y_2$};
			\draw[thick,color=red,dashed,rotate=90, yscale = -1] (-2.2,-.346) .. controls +(-.05,.25) and +(-.05,.25).. (-3.872,-1);
			
			\draw[very thick] (0,0) circle[radius = 4cm];
			\draw[very thick] (-2,0) circle [radius =.4cm];
			\draw[very thick] (0,2) circle [radius = .4cm];
			\draw[very thick] (0,-2) circle [radius = .4cm];
			\draw[fill] (2.5,0) circle [radius = 1.5pt] ;
			\draw[fill] (2.38,-.6) circle [radius = 1.5pt] ;
			\draw[fill] (2.38,.6) circle [radius = 1.5pt] ;
			
			\draw[fill,color = blue] (0,0) circle [radius = 2pt] (0,0) node [anchor = west] {$f$} ;
			\draw[fill, color = blue] (p1) circle [radius=1.5pt] (p1) node [anchor = north east, scale = .8] {$p_1$};
			\draw[fill, color = blue] (p2) circle [radius=1.5pt] (p2) node [anchor = north east, scale = .8] {$p_2$};
			\draw[fill, color = blue] (pk) circle [radius=1.5pt] (pk) node [anchor = south east, scale = .8] {$p_k$};
			
		\end{scope}

		\begin{scope}[xshift=9cm]
			
			\draw[thick] (-4,0) node[scale = 2] {$\times$};
			
			\draw[thick, color=blue, name path = b1] (-1.25,0) circle [radius = 1cm](-2.2,0) node[anchor=east]{$b_1$};
			\draw[thick, color=blue, name path = b2] (1.25,0) circle [radius = 1cm] (2.2,0) node[anchor=west] {$b_2$};
			
			\draw[thick,color=blue, name path = a1] (-1.25,.5) .. controls +(.25,0) and +(.25,0) .. (-1.25,2)(-1.25,2) node [anchor = south] {$a_1$};
			\draw[thick,color=blue, dashed] (-1.25,.5) .. controls +(-.25,0) and +(-.25,0) .. (-1.25,2);
			
			\draw[thick,color=blue, name path = a2] (1.25,.5) .. controls +(.25,0) and +(.25,0) .. (1.25,2)(1.25,2) node [anchor = south] {$a_2$};
			\draw[thick,color=blue, dashed] (1.25,.5) .. controls +(-.25,0) and +(-.25,0) .. (1.25,2);
			
			\path [name intersections={of=a1 and b1,by=k1}];
			\path [name intersections={of=a2 and b2,by=k2}];
			\draw[fill, color = blue] (k1) circle [radius=1.5pt] (k1) node[scale = .8, anchor=north west] {$q_1$};
			\draw[fill, color = blue] (k2) circle [radius=1.5pt] (k2) node[scale = .8, anchor=north west] {$q_2$};
			
			\draw[fill, color = blue] (0,1.5) circle [radius=1.8pt] (0,1.5) node[anchor=north] {$g$};
			\draw[thick, color = blue] (k1) .. controls +(0,0) and +(-.5,0) .. (0,1.5) .. controls +(.5,0) and +(0,0).. (k2) (-.6,1.55) node[scale=.8] {$\gamma_1$} (.6,1.55) node[scale=.8] {$\gamma_2$};
			
			\draw[thick, color=red] (1.25,-.5) .. controls +(.25,0) and +(.25,0) .. (1.25,-2) (1.25,-2) node [anchor = north] {$A_2$};
			\draw[thick, color=red, dashed] (1.25,-.5) .. controls +(-.25, 0) and +(-.25,0) .. (1.25,-2);
			
			\draw[thick, color = red] (-1,-.433) .. controls +(.25,.05) and + (.25,0) .. (-.75,-2) (-.75,-2) node[anchor = north] {$A_1'$};
			\draw[thick, color = red, dashed] (-1,-.433) .. controls +(-.25,-.05) and + (-.25,0) .. (-.75,-2);
			
			\draw[thick, color = red] (-1.5,-.433) .. controls +(.25,-.05) and + (.25,0) .. (-1.75,-1.8) (-1.75,-1.8) node[anchor = north] {$A_1$};
			\draw[thick, color = red, dashed] (-1.5,-.433) .. controls +(-.25,.05) and + (-.25,0) .. (-1.75,-1.8);
			
			\draw[thick, color=red] (1.25,0) circle [radius = .7cm] (.75,0);
			\draw [->, color = red] (2,-1) -- (1.7,-.6); 
			\draw[thick, color = red] (1.8,-.95) node[anchor=north west, scale = .8]{$B_2$};
			
			\draw[very thick] (1,2) -- (-1,2) .. controls +(-1,0) and + (0,1) .. (-3,0).. controls +(0,-1) and +(-1,0) .. (-1,-2) -- (1,-2) .. controls +(1,0) and +(0,-1) .. (3,0) .. controls +(0,1) and +(1,0) .. (1,2);
			
			\draw[very thick] (-1.25,0) circle[radius = .5cm];
			\draw[very thick] (1.25	,0) circle[radius = .5cm];
			
		\end{scope}
		
		\end{scope}
		
	\end{tikzpicture}
	\caption{The curves on $F\times G\cong \Sigma_k \times \Sigma_2$.}
	\label{fig:F_times_G}
\end{figure}
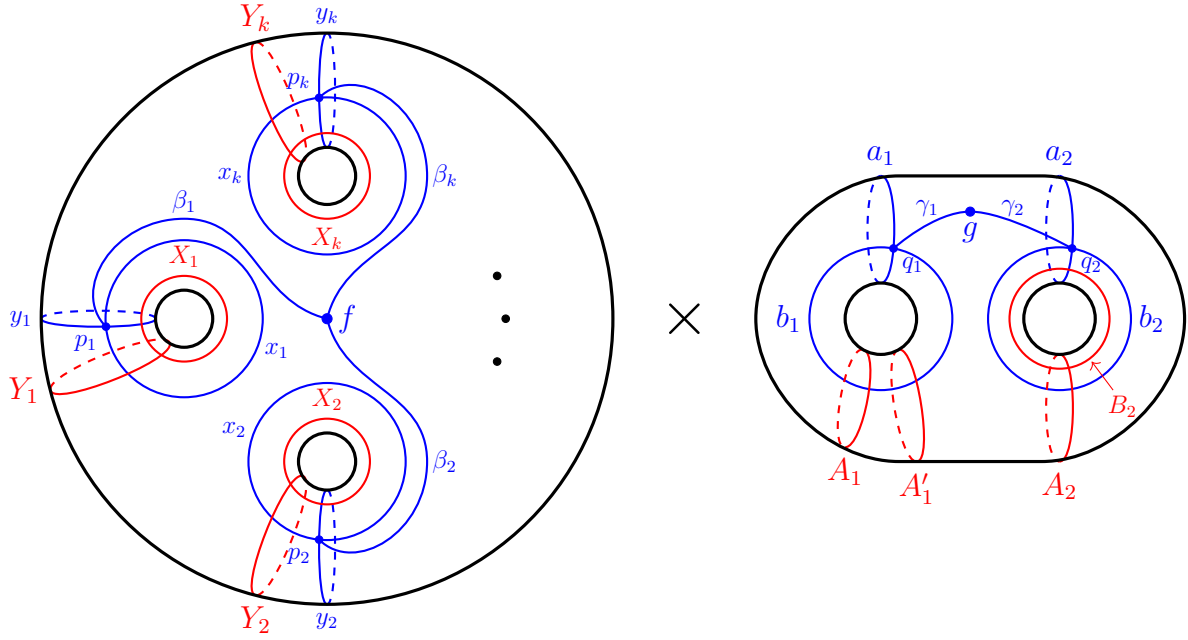

\begin{proposition}
	Let $F$ be a closed genus $k$ surface and $G$ a closed genus two surface. Let $N$ be a union of neighborhoods of the following disjoint Lagrangian tori:\begin{align*}
		T_i &:=  X_i\times A_1   \qquad J_1 :=  Y_1 \times A_2 \\
		L_i &:=  Y_i \times A'_1  \qquad \; J_2 :=  Y'_1\times B_2.\\
	\end{align*} Then $\pi_1(F\times G-N,\,(f,g))$ is normally generated by $\{\tilde x_1,\tilde y_1,\cdots , \tilde x_k,\tilde y_k, \tilde a_1,\tilde b_1,\tilde a_2,\tilde b_2\}$. The curves generating $\pi_1$ for the neighborhood boundaries of each torus are:\begin{align}
		\label{bd1} \mu_{T_i} &= [\tilde b_1^{-1},\tilde y_i^{-1}] && m_{T_i} = \tilde x_i && \ell_{T_i} =\tilde a_1\\
		\mu_{L_i} &= [\tilde x_i^{-1},\tilde b_1] && m_{L_i} = \tilde y_i &&  \ell_{L_i} =\tilde b_1\tilde a_1\tilde b_1^{-1}\\
		\label{bd3} \mu_{J_1} &= [\tilde x_1^{-1},\tilde b_2^{-1}] && m_{J_1} = \tilde a_2 &&  \ell_{J_1} =\tilde y_1\\
		\label{bd2} \mu_{J_2} &= [\tilde a_2^{-1},\tilde x_1] && m_{J_2} = \tilde b_2 &&  \ell_{J_2} =\tilde x_1\tilde y_1\tilde x_1^{-1}.
	\end{align} The relation $[\tilde b_2, \tilde y_i]=1$ holds for all $i$.
	\label{prop:sigmak_times_sigma2}
\end{proposition}

\begin{proof}
	By Seifert-Van Kampen, $\pi_1(F\times G - N, (f,g))$ will be a quotient of \[
	\pi_1(\Sigma_k\times \Sigma_1^1 - N_1,(f,g)) * \pi_1(\Sigma_k\times \Sigma_1^1 - N_2,(f,g)),
	\] where $N_1 = \cup_i(T_i\cup L_i)$ and $N_2 = J_1\cup J_2$. Lemma \ref{lem:sigma_k_times_punctured_torus} provides the normal generators of $\pi_1(\Sigma_k\times \Sigma_1^1 - N_1,(f,g))$ after the conjugation by $\gamma_1$ to re-base from $(f,q_1)$ to $(f,g)$. To compute $\pi_1(\Sigma_k\times \Sigma_1^1 - N_2,(f,g))$, we decompose $\Sigma_k\times \Sigma_1^1 -N_2$ as \[
	(\Sigma_1^1 \times \Sigma_1^1-N_2) \cup \left(\Sigma_{k-1}^1\times \Sigma_1^1\right).
	\] Once again we apply \cite[Theorem 2]{BK07} to find that $\{\tilde a_2,\tilde b_2,\tilde x_1, \tilde y_1\}$ normally generate $\pi_1(\Sigma_1^1\times \Sigma_1^1-N_2,(f,g))$. The generating curves for the neighborhood boundary of each $J_i$ are as in \eqref{bd3} and \eqref{bd2}. Moreover, by the cited Theorem 2, the following relations hold:\[
	[\tilde a_2,\tilde y_1]=[\tilde b_2,\tilde y_1] = [\tilde b_2,\tilde x_1\tilde y_1\tilde x_1^{-1}]=1.
	\] On the other hand, $\Sigma_{k-1}^1\times \Sigma_1^1$ is a product space whose fundamental group is much easier to compute. Applying Seifert-Van Kampen, $\pi_1(\Sigma_k\times \Sigma_1^1-N_2(f,g))$ is generated by $\{\tilde x_1,\tilde y_1,\cdots, \tilde x_k,\tilde y_k, \tilde a_2,\tilde b_2\}$. The following relations hold:\begin{align}
		[\tilde a_2,\tilde y_1]=[\tilde b_2,\tilde y_1] = [\tilde b_2,\tilde x_1\tilde y_1\tilde x_1^{-1}]&=1  \label{eq1}\\ 
		\label{eq2} [\tilde a_2,\tilde x_i]=[\tilde a_2,\tilde y_i] = [\tilde b_2,\tilde x_i]=[\tilde b_2,\tilde y_i]&=1 \text{ for } i>1.  
	\end{align} We apply one final round of Seifert-Van Kampen to see that $\pi_1(F\times G-N,(f,g))$ is generated by the set given in the lemma statement, and that the neighborhood boundary curves are as in \eqref{bd1}-\eqref{bd2}. Equations \eqref{eq1} and \eqref{eq2} ensure that $[\tilde b_2,\tilde y_i]$ is nullhomotopic in $F\times G-N$ for any $i$.
\end{proof}

\section{Irreducible copies of $R_{a,b}$ when $a$ is even}\label{section:even_costructions}

We give various methods to construct irreducible copies of $R_{a,b}$ for even $a$, covering all but finitely many $(a,b)$--pairs. We realize many $(a,b)$--coordinates by applying the $\Z_2$--construction to a family of manifolds which were first defined in \cite{BK}, then used to fill out a large part of the simply--connected geography plane in \cite{ABBKP}. For completeness, we reintroduce this family. These $4$--manifolds are built from four telescoping triples which serve as building blocks. We start by defining the following ``subblocks,'' from which the telescoping triples are obtained. In what follows, we let $X_{n\text{ L.S.}}$ denote a manifold obtained through $n$ Luttinger surgeries on $X$.

\begin{align*}
	M &= T^2\times \Sigma_2\\
	W_1 &= T^2\times S^2 \# 4\CPbar{2} \\
	W_2 & = T^2\times T^2\# 2\CPbar{2}\\
	Z_g & = (\Sigma_2\times \Sigma_g)_{2g \text{ L.S. }}
\end{align*}

In our description of each subblock, we denote by $\{x_i,y_i\}$, and $\{a_i,b_i\}$, the symplectic basis for the integral first homology of the first surface and the second surface in this product, respectively. We drop the subscript whenever the surface is a torus, and there is just one pair of generators. We define the following surfaces up to homology, which will be used for fiber summing and torus surgery: $M$ contains $ \Sigma_{h'} := \{pt\} \times \Sigma_2$ and $T_1 := x\times a_1, T_2 := y\times a_1,  T_3 := x\times a_2, T_4 := y\times a_2$. Next, $W_1$ contains a genus 2 surface $F_1$ obtained by taking two parallel copies of $T^2\times \set{pt}$ and one $\set{pt} \times S^2$ in $T^2\times S^2$. After resolving the two intersection points, this is an embedded genus two surface with self-intersection four. After blowing up four times, we obtain the embedded genus two surface $F_1$ in $W_1$, which has a trivial normal bundle. Similarly, $W_2$ has a genus two surface $F_2$ which is obtained by resolving and blowing up the surface $\{pt\} \times T^2 \cup T^2 \times \{pt\}$ in $T^4$. In addition, $W_2$ contains the tori $T'_1 = x \times a$ and $T'_2 = y \times a$. Note that the number of blow ups is conveniently arranged so that the self-intersection of each $F_i$ is zero and the complement is relatively minimal. Finally, $Z_g$ is obtained from $\Sigma_2\times\Sigma_g$ by performing Luttinger surgeries on the $2g$ lagrangian tori $y_1\times a_i$ and $y_2\times b_i$. We will restrain from specifying the surgery coefficient and the surgery curve unless it is required for an argument. Now we are ready to construct the building blocks $B, B_g, C$, and $D$.
\begin{align*}
	B & = \left(\Fibersumd{W_2}{M}{F_2=\Sigma_{h'}}\right)_{4 \text{ L.S.}}\\ 
	B_g & = \left(\Fibersumd{B}{Z_g}{F_B = \Sigma_2 \times \set{pt}}\right)\\
	C & = \left(\Fibersumd{W_1}{M}{F_1 = \Sigma_{h'}}\right)_{2 \text{ L.S.}} \\
	D & = \Fibersumd{W_1}{W_2}{F_1=F_2}, 
\end{align*}
where $F_B$ is a parallel copy of $\Sigma_{h'}$ in $M$. The Luttinger surgeries for $B$ are performed on $T'_1,T'_2 \subset W_2-F_2$ and $T_1,T_2 \subset M$. For $C$ the surgeries are performed on $T_1,T_2\subset M-\Sigma_{h'} \subset C$. All three of $B, B_g, C,$ are telescoping triples with respect to the pair $T_3, T_4 \subset M-\Sigma_{h'}$. $D$ is also a telescoping triple, with respect to $T'_1,T'_2 \subset W_2-F_2$. Going forward, we relabel the tori so that for $X\in \{B,B_g,C,D\}$, $T_1$ and $T_2$ are the homologically essential tori, and $(X,T_1,T_2)$ is a telescoping triple.

Recall that two telescoping triples can be symplectically fiber summed to give another telescoping triple. Moreover, one can always obtain a simply--connected manifold through Luttinger surgery on a telescoping triple \cite[Proposition 4]{ABBKP}. Through a combination of symplectic fiber summing and Luttinger surgery, a large family of irreducible simply--connected manifolds were built out of $B$, $B_g$, $C$, and $D$ in \cite{ABBKP}. In what follows, we apply the $\Z_2$-construction to this family.

\subsection{$\mathbf{(b_2^+,b_2^-) = (2\chi, 10\chi-c)}$ for $\mathbf{0< c\leq 8\chi-2}$ (except finitely many points)} \label{section:big_region}

The overarching idea in this subsection is the following: take an irreducible, simply-connected $4$-manifold $Y$ with first chern number $c$ and holomorphic Euler characteristic $\chi$, then perform the $\Z_2$-construction to $Y$ along an embedded genus two surface. The resulting $4$-manifold $X$ will have $(b_2^+,b_2^-)=(2\chi,10\chi-c)$. During this $\Z_2$-construction, $c_1^2$ and $\chi_h$ will transform by $c\mapsto c+4$ and $\chi\mapsto \chi+\frac{1}{2}$. If $X$ has order two fundamental group and both $X$ and its double cover are not spin, $X$ is homeomorphic to $R_{2\chi,10\chi-c}$. Given the constraints $0<c \leq 8\chi -2$, this subsection fills in the geography plane above the line $\sigma=-2$ and strictly below the line $c=4$.

We construct $Y$ out of telescoping triples as in \cite[Theorem 22]{ABBKP}. In that proof, they start with a pair of integers $(c,\chi)$ satisfying $0\leq c\leq 8\chi -2$. The proof is split in two cases depending on the parity of $c$. For us, the case when $c=0$ will be addressed with a separate argument in \ref{sec:small_c}, so here we assume $0<c\leq 8\chi -2$. Before we start the constructions, we make some observations which will come in handy for $\pi_1$ computations.

\begin{rmk}
	For a 4-manifold $M$ with an embedded square zero surface $\Sigma_g$, $\mu_{\Sigma_g}$ is trivial in $M-\Sigma_g$ if and only if there exists an immersed sphere $S\subset M$ which intersects $\Sigma_g$ transversely at one point.
	
	\label{rmk:trivial_meridian}
\end{rmk}

\begin{lemma}\label{lem:meridian_fiber_sum}
	Let $M = \Fibersumd{M_1}{M_2}{\Sigma_g}$ and $\Sigma_g^{\parallel} \subset M$ be a disjoint push-off of $\Sigma_g$ in $M_1-\Sigma_g$. If there are immersed spheres $S_i \subset M_i$ intersecting $\Sigma_g$ transversally at a single point, then $\mu_{\Sigma_g^{\parallel}}$ is trivial in $M-\Sigma_g^{\parallel}$. 
	
\end{lemma}

\begin{proof}
	Suppose $S_i \subset M_i$ are immersed spheres intersecting $\Sigma_g$ geometrically once. Then during the fiber sum operation they get punctured once and can be glued together with an annulus in $M$. Also $S_1$ will intersect $\Sigma_g^{\parallel} \subset M_1$ once, and so the same happens in $M$.
\end{proof}

\subsubsection{$c$ is even}
Let $(m,n)= (\frac{1}{2} c,\chi)$. By \cite[Lemma 21]{ABBKP}, there are non-negative integers $b, c, d, g, k$, with $b>0$ if $g>0$, such that $(m,n) = (3b+2c+d+4g, b+c+d+g+k)$. We define $Z$ as the following symplectic fiber sum, where all gluings are defined so that $Z$ is a telescoping triple.
\[
Z:= \begin{cases}
	 \left(\#_b B\right)\#\left(\#_c C\right)\#\left(\#_d D\right)  & \text{ if } g=0\\
	 B_g \#\left(\#_{b-1} B\right)\#\left(\#_c C\right)\#\left(\#_d D\right) & \text{ if } g\geq 1
\end{cases}
\]

Let $T_1$ and $T_2$ be the tori embedded in $Z$ so that $(Z,T_1,T_2)$ is a telescoping triple. If $k = 0$, let $Y$ be the simply--connected manifold obtained through Luttinger surgeries on $Z$. If $k>0$ and one of $b,c,d$ is positive, let $Y$ be obtained by performing $+1$ Luttinger surgery on $T_2$ in the fiber sum $\Fibersumd{Z}{E(k)}{T_1=\text{Fiber}}$.

By the proof of \cite[Theorem 22]{ABBKP}, $c_1^2(Y)= 2m$, $\chi_h(Y)=n$, and $Y$ is simply-connected, minimal, and symplectic. Moreover, $Y$ is a symplectic sum of at least one copy of $B$, $B_g$, $C$, or $D$. Each of these building blocks has a symplectically embedded genus two surface of square zero. Let $H$ be one such surface embedded in $Y$. To proceed, we apply the $\Z_2$-construction to $Y$ along $H$. We claim that the resulting manifold has an order two fundamental group. This follows from Proposition \ref{prop:z2_pi1_odd} as long as $\mu_H$ is trivial in $\pi_1(Y-H)$. To verify the triviality of $\mu_H$, we consider each telescoping triple $(X,T_1,T_2)$ for $X\in\{B,C,D,B_g\}$, relabeling each embedded genus two surface in $B, C, D, B_g$ as $H$. If $\mu_H$ is trivial in $X\setminus \nu(H \cup T_1\cup T_2)$, then $\mu_H$ will remain trivial after performing Luttinger surgeries and symplectic fiber sums along $T_1$ and $T_2$. After constructing $Y$ out of $X$ and taking its double along $H$, the resulting manifold will be simply--connected. Therefore, to show that the $\Z_2$-construction of $Y$ along $H$ has an order two fundamental group, we proceed to show that $\mu_H$ is trivial in $X\setminus \nu(H \cup T_1\cup T_2)$ for each $X\in \{B,C,D,B_g\}$. For the case when $X=B$, this has already been shown.

\begin{proposition}\cite[Theorem 7]{ABBKP}
	The meridian $\mu_H$ is trivial in $B-(H\cup T_1\cup T_2)$.
	\label{prop:B_meridian}
\end{proposition}

We proceed by verifying the same for $C$, $D$, and $B_g$.

\begin{proposition}
	The meridian $\mu_H$ is trivial in $C-(H\cup T_1\cup T_2)$.
\end{proposition}

\begin{proof}
	$C$ is defined in \cite[Theorem 11]{ABBKP} as a manifold from the proof of \cite[Theorem 10]{BK} without applying surgery on $T_3$, and $T_4$, assuming the notation of \cite{BK}. From the page 17 of \cite{BK} the meridian is homotopic to $[x,y]$. After $-1$ Luttinger surgery on $T_1$ along $x$, the commutator becomes trivial, and so does meridian.
\end{proof}

\begin{proposition}
	The meridian $\mu_H$ is trivial in $D-(H\cup T_1\cup T_2)$. \label{prop:D_meridian}
\end{proposition}
\begin{proof}
	This is immediate from Lemma \ref{lem:meridian_fiber_sum} because $D$ is a fiber sum of manifolds both having exceptional spheres intersecting the glued surfaces.
\end{proof}

For $B_g$, recall that this is the fiber sum of $B$ and $Z_g$ along $F_B\subset B$ and $\Sigma\subset Z_g$. Let $F'$ be a push-off of the embedded genus two surface $F_B$ in $B$, so that $F'$ remains an embedded surface in $B_g$. Set $N=\nu(T_1\cup T_2\cup F')$. In order to make the $\Z_2$-construction on $B_g$ work, we need $\mu_{F'}$ to be trivial in $B_g\setminus N$.

\begin{proposition}
	The meridian of $F'$ is trivial in $\pi_1(B_g\setminus N)$.
\end{proposition}

\begin{proof}
	Let $\phi\colon\Sigma\to F$ be the diffeomorphism described in \cite[Corollary 9]{ABBKP}, so that $\tilde \phi:= c_{S^1}\times \phi$ is the gluing map for the symplectic fiber sum forming $B_g$, where $c_{S^1}$ is the complex conjugation map. Then \[
	B_g\setminus N =  B\setminus \nu(F_B\cup F'\cup T_1\cup T_2)\cup_{\tilde \phi} Z_g\setminus \nu(\Sigma). \]
    From Seifert-Van Kampen, we see that $\pi_1(B_g\setminus N)$ is a quotient of $\pi_1(B\setminus \nu(F_B\cup F'\cup T_1\cup T_2))*\pi_1( Z_g\setminus \nu(\Sigma))$. Since $\mu_{F_B}$ is trivial in $\pi_1(B\setminus \nu(F_B\cup T_1\cup T_2))$ by Proposition \ref{prop:B_meridian}, there exists an immersed sphere $S$ in $B\setminus \nu(T_1\cup T_2)$ intersecting $F_B$ in a point. Since $F_B$ and $F'$ are homologous, $|S\cdot F'|=1$ as well. After removing neighborhoods of $F_B$ and $F'$, $S$ becomes an annulus, showing that $\mu_{F_B}=\mu_{F'}$ in $\pi_1(B_g\setminus \nu(F_B\cup F'\cup T_1\cup T_2)).$
	
	 Since $\tilde \phi$ is complex conjugation on the $S^1$-factor, $\mu_{F_B}$ is identified with $\mu_\Sigma^{-1}$ in the quotient space $\pi_1(B_g\setminus N)$, hence so is $\mu_{F'}$. Let $x_1,y_1,x_2,y_2$ denote the standard generators of $\pi_1(\Sigma)$ for $\Sigma \subset Z_g$ and let $\{a_i,b_i\}_{i=1}^g$ generate the fundamental group of $\{pt\}\times \Sigma_g\subset Z_g$. Note that the Luttinger surgery in $Z_g$ gives relations between these generators of $\pi_1(Z_g)$, see \cite[Corollary 9]{ABBKP}. Then $\mu_\Sigma = [a_1,b_1]\cdots [a_g,b_g]\in \pi_1(Z_g\setminus \nu(\Sigma))$, so $\mu_{F'} = \left([a_1,b_1]\cdots [a_g,b_g]\right)^{-1}\in \pi_1(B_g\setminus N)$. Let $\tilde x_1, \tilde y_1, \tilde x_2, \tilde y_2$ be the symplectic generators of $\pi_1(F')$. By \cite[Theorem 7]{ABBKP}, $\tilde x_1 = \tilde y_1 = \tilde x_2 = 1\in \pi_1(B\setminus \nu(F_B\cup T_1\cup T_2))$. Then there is an immersed disk in $B\setminus \nu(F_B\cup T_1\cup T_2)$ bounded by $\tilde x_1$. Since $F'$ is a push-off of $F_B$, this disk can be assumed to be disjoint from $F'$ as well, so $\tilde x_1= 1 \in \pi_1(B\setminus \nu(F_B\cup F'\cup T_1\cup T_2))$. The same logic applies for $\tilde x_2$. Since $\tilde \phi$ maps $x_i\mapsto \tilde x_i$, $x_1$ and $x_2$ are trivial in $\pi_1(B_g\setminus N)$. Then from the Luttinger surgeries in $Z_g$, we have the relations $[x_1,b_i]=a_i$ and $[x_2,a_i]=b_i$. Since $x_1=x_2=1$ in $\pi_1(B_g\setminus N)$, these relations reduce to $a_i=b_i=1$. Therefore $\mu_{F'}=\left([a_1,b_1]\cdots [a_g,b_g]\right)^{-1}=1\in \pi_1(B_g\setminus N)$.
	 
\end{proof}

\begin{rmk}
	One can also work out that although $\mu_{F'}=1$ in $\pi_1(B_g\setminus N)$, $\mu_{F'}$ is {\it nontrivial} in $\pi_1(B \setminus \nu(F_B\cup F'\cup T_1\cup T_2)).$ This is typically the case when two parallel copies of an embedded surface are removed.
\end{rmk}


This concludes our argument that $Y$ has a symplectically embedded genus two surface $H$ such that $\mu_H$ is trivial in $Y-H$. Thus after applying the $\Z_2$-construction to $Y$ along $H$,  the resulting manifold has an order two fundamental group. The last thing to verify is that the $\Z_2$-construction of $Y$ along $H$ will have an odd intersection form.

\begin{proposition} \label{prop:odd_int_form}
	$B$, $C$, $D$, and $B_g$ all have embedded surfaces of odd self-intersection disjoint from the embedded $H$ and the telescoping tori.
\end{proposition}

\begin{proof}
	Since all $B$, $C$, $D$, $B_g$ contain either $W_1-H$ or $W_2-H$, we will show that there is a torus $K$ in $W_i$ of self-intersection $-1$. Assume that one of the blow ups on $W_1$ is made on $\{pt\} \times \{pt'\} \subset \{pt\} \times S^2$ far from the resolution points. Then consider the torus $K$ given by the proper transform of $T^2 \times \{pt'\}$. Its self-intersection is $$[K]^2=[T^2\times \{pt'\}]^2 -1 = -1.$$
	It is disjoint from $H$ because of the blow up, and disjoint from the telescoping tori as long as $\{pt'\}$ does not lie on the curves $a_i$, which we can assume by perturbing $a_i$ if needed. The argument for $W_2$ is the same.
	
\end{proof}

It now follows from Proposition \ref{prop:Rab_topological_type} that the $\Z_2$-construction of $Y$ along $H$ is homeomorphic to $R_{2\chi, 10\chi-c}$. Moreover, Usher's theorem implies that this copy of $R_{2\chi, 10\chi-c}$ is minimal. This realizes all coordinates of the form $(2\chi, 10\chi-c)$ for $0<c\leq 8\chi -2$ and even  $c$. We next turn to the case when $c$ is odd.

\subsubsection{$c$ is odd \label{sec:c_odd}}

We make use of the even case and some new blocks. Note that $c = c_1^2(Y) = 3\sigma(Y)+2e(Y)$, so $\sigma(Y)$ is always odd. Taking the double of $Y$ and then its quotient guarantees that signature of the double is not divisible by 16 and the signature of the quotient is not divisible by 8. Hence, we do not need to search for odd surfaces as in the even case, and the non-spinness of both manifolds is automatic.

\underline{$1\leq c \leq 8\chi -17$.} Let $(c',\chi') = (c-1,\chi-2)$. If $c'\neq 0$, use the previous section to construct the telescoping triple $(Z, T_1,T_2)$ corresponding to the pair $(c',\chi')$. If $k=0$, do $+1$ surgery on $T_1$ and perturb $T_2$ to make it symplectic, then take a fiber sum along $T_2$ with Gompf's manifold $S_{1,1}$ along the embedded torus $F_1\subset S_{1,1}$ \cite[Lemma 5.5]{GompfNew}. If $k\geq 1$, then instead of performing Luttinger surgery on $T_1$, take the symplectic sum along $T_1$ with $E(k)$. The new manifold has invariants $(c,\chi)$. If $c'=0$, then take $Z=E'(k)$ from \cite{GompfNuclei}. Take the symplectic fiber sum of $E'(k)$ and $S_{1,1}$ along embedded tori with trivial normal bundles. The resulting manifold has invariants $(c,\chi)$. By \cite[Lemma 5.5]{GompfNew}, $S_{1,1}$ has an embedded genus two surface $F_2$ of square zero such that $S_{1,1}-(F_1\cup F_2)$ is simply connected. Therefore in all cases described, there is a genus two surface which we can perform the $\Z_2$-construction along to obtain a manifold with order two fundamental group.

\underline{$7 \leq c \leq 8\chi - 11$.} Let $(c',\chi') = (c-7,\chi-2)$. If $c'\neq 0$, obtain $(Z,T_1,T_2)$ corresponding to $(c',\chi')$. Repeat the argument above replacing $S_{1,1}$ with $X_{3,12}$ from \cite[Section 7]{ABBKP}. The new manifold has invariants $(c,\chi)$. Any integral lattice points with $c'=0$ were covered in the previous case, because all pairs $(7,\chi)\in\Z^2$ satisfying $7\leq 8\chi - 11$ also satisfy $7\leq 8\chi -17$. 

\underline{$21 \leq c \leq 8 \chi -5$.} If $c\neq 21$, repeat arguments by replacing $X_{3,12}$ with $P_{5,8}$ from \cite[Remark 1]{ABBKP}. Similar to the previous case, we do not have to worry about when $c=21$, since all integral lattice points satisfying $c=21\leq 8\chi -5$ are covered by the argument for $7 \leq c \leq 8\chi - 11$. The last region to cover is the line $21\leq c=8\chi -3$, which is realized by manifolds with $\sigma =-3$. We were unable to reach this part of the geography plane using $B$, $C$, $D$, and $B_g$, so we introduce a new construction.

\noindent {\bf Current progress of lattice points.} Before moving on to the construction for $\sigma=-3$, we pause to list the $b_2^{\pm}$ of manifolds {\it not} reached by our construction so far. The above argument populates the $\Z_2$ geography plane strictly above the $\sigma=-1$ line, given by $b_2^- = b_2^+ + 1$, and strictly below the $c_1^2 = 4$ line, given by $b_2^- = 5b_2^+$. (Again remember that we are only dealing with the coordinates with even $b_2^+$ in this section). The points strictly between the lines $\sigma=-1$ and $c_1^2 = 4$ which we miss are (1) the line $\sigma = -3$, and (2) the points that \cite[Theorem 22]{ABBKP} misses, with the $\Z_2$-construction applied to them. The proof of \cite[Theorem 22]{ABBKP} misses 13 coordinates. The $\Z_2$-construction applied to each of them gives the following list of coordinates, which our proof so far omits: \begin{align*}
	&(2,9),(2,7),(2,5)\\
	&(4,19),(4,17),(4,15),(4,13),(4,11),(4,9),(4,7)\\
	& (6,15),(6,13),(6,11)
\end{align*}
In summary, we have populated the $\Z_2$ geography plane with $b_2^+$ even, $\sigma \leq 0$, and $c_1^2\geq 0$ {\it except} for the thirteen $(b_2^+,b_2^-)$ coordinates listed above, as well as the coordinates with $c_1^2\leq 4$ and $\sigma = -3$. Our paper is not concerned with the lines $\sigma =0$ and $\sigma = -1$, which were handled in \cite{baykur2024smooth}. We next tackle the line $\sigma = -3$. 

\subsubsection{Irreducible $R_{2+2k,5+2k}$, where $k\geq 1$, i.e. $\sigma = -3$}·\label{sec:sigma_3} We use a genus-2 Lefschetz fibration over $S^2$ from 
\cite[Theorem 7]{Bay_Kork} with the total space $S^2 \times T^2 \# 3\CPbar{}$. We will call it $N_0$. Let $N_k$ be a fiber sum of $N_0$ with a trivial genus-2 fibration over $\Sigma_k$, which admits a genus two Lefschetz fibration over $\Sigma_k$. Since $N_k$ is relatively minimal and $k>0$, $N_k$ is minimal by \cite{Stip_chern}. Its algebraic invariants are $e(N_k)=e(N_0) + e(\Sigma_2 \times \Sigma_k)+4=3+4k$ and $\sigma(N_k)= \sigma(N_0)+\sigma(\Sigma_2 \times T^2)= -3$, which are the same as of $\CPsum{(2k-1)}{(2k+2)}$. Although $\pi_1(N_k)$ is nontrivial, we can obtain a simply--connected manifold by performing Luttinger surgery on $\Sigma_2 \times \left(\Sigma_k-D^2\right)\subset N_k$, where $D^2$ contains all critical values of $f_k$. We will call the resulting manifold $N_k'$. In what follows, we give a precise definition of all the Luttinger surgeries involved, then argue these surgeries kill the fundamental group. From this and Proposition \ref{prop:Luttinger_minimal}, it follows that $N_k'$ is an irreducible copy of $\CPsum{(2k-1)}{(2k+2)}$.

Let $x_i,y_i$ and $a_i,b_i$ be standard generators of the base and fiber of $f_k$, respectively. Let $X_i$, $Y_i$, $A_i$, and $B_i$ be parallel copies of $x_i$, $y_i$, $a_i$, and $b_i$. Let $Y'_1$ and $A'_1$ be additional push-offs that are disjoint from $Y_1$ and $A_1$, respectively. Perform the following Luttinger surgeries:
\begin{align}
	(A_2\times Y_1,a_2, 1) \label{eq:sigma=3_1}\\
	(B_2\times Y'_1,b_2,1)\label{eq:sigma=3_2} \\
	(A_1\times X_i,x_i,1) \label{eq:sigma=3_3}\\
	(A'_1\times Y_i,y_i,1). \label{eq:sigma=3_4}
\end{align} In the tuples above, the first coordinate is the torus where the surgery is performed, the second coefficient is the surgery curve along which $D^2\times T^2$ is re-glued, and the third coordinate is the surgery coefficient. All framings are assumed to be Lagrangian. All of these tori are disjoint from a generic regular fiber of $f_k$, and we denote a surviving regular fiber by $F$. Instead of showing that $N_k'$ is simply-connected, we will show that $N_k'-F$ is simply-connected, so that we guarantee the simply-connectedness of the double when applying the $\Z_2$-construction to $N_k'$.

To compute $\pi_1(N'_k-F)$, first we argue what happens when we remove the Lagrangian tori from $N_k$. Let $\mathcal N$ be a union of neighborhoods of the tori in (\ref{eq:sigma=3_1})-(\ref{eq:sigma=3_4}), so that $\pi_1(N_k')$ is a quotient of $\pi_1(N_k - \mathcal N)$. Note that $\mathcal N$ is the same as $N$ in Proposition \ref{prop:sigmak_times_sigma2}, which we will rely on in the computations below. We can think of $N_k$ as derived from the product space $\Sigma_2\times \Sigma_k$ with added $2$-handles coming from the monodromy of $f_k$. $\pi_1(N_k)$ is therefore generated by $\{a_1,b_1,a_2,b_2,x_1,y_1,\cdots x_k,y_k\}$. By \cite[Theorem 2]{BK07} and its extension in Proposition \ref{prop:sigmak_times_sigma2}, $\pi_1(N_k-\mathcal N)$ is normally generated by this same collection of loops. The relations $b_1b_2=1$ and $a_1a_2^2b_2^4=1$ come from Lefschetz vanishing cycles for $f_k$ \cite[3.2]{Bay_Kork}, and these survive the removal of $\mathcal N$ and so remain relations in $\pi_1(N_k-\mathcal N)$.

Proposition \ref{prop:sigmak_times_sigma2} shows that $[b_2,y_1]=1$ in $\pi_1(N_k-\mathcal N)$, and since $\pi_1(N_k')$ is a quotient of $\pi_1(N_k-\mathcal N)$, this relation holds in $\pi_1(N_k')$ as well. When we remove the fiber $F$, we are left with a Lefschetz fibration over a punctured $\Sigma_k$. This does not affect the normal generating set, and the above relations $b_1b_2=a_1a_2^2b_2^4=[b_2,y_1]=1$ still hold in $\pi_1(N_k'-F)$. When we glue copies of $D^2\times T^2$ to $N_k-\mathcal N$ to complete the Luttinger surgeries, we add $\pi_1$ relations $\eta_T\mu_{T}=1$, where $\eta_T$ is the surgery curve on $T$. By Proposition \ref{prop:sigmak_times_sigma2}, the meridian of the torus $A_1\times X_1$ in (\ref{eq:sigma=3_3}) is $[b_1^{-1},y_1^{-1}]$. As mentioned before, $b_1b_2=1$ and $[b_2,y_1]=1$, therefore $[b_1,y_1]=1$. So the surgery curve $x_1$ of (\ref{eq:sigma=3_3}) is trivial. This will imply that all the meridians of (\ref{eq:sigma=3_1}) and (\ref{eq:sigma=3_2}) are trivial, because they contain $x_1$ or $x_1^{-1}$. Then this implies that corresponding surgery curves $a_2$ and $b_2$ are trivial. Since $b_1b_2=1$, $b_1$ is also trivial. But $b_1$ takes place in the commutators corresponding to the meridians of (\ref{eq:sigma=3_3}) and (\ref{eq:sigma=3_4}), so the surgery curves $x_i$ and $y_i$ are also trivial. This shows that all the normal generators $x_i,y_i,a_i,b_i$ are trivial. Therefore $\pi_1(N_k'-F)=1$. Then if we apply the $\Z_2$-construction on $N'_k$ along $F$, we get an irreducible copy of $R_{2+2k,5+2k}$.

\subsection{Small $\mathbf{c \in \set{0,1,2,3,4}}$  \label{sec:small_c}}

\subsubsection{$R_{2k, 10k+4}$ for $k\geq 1$ (i.e. $c=0$)}

There exists a genus $g\geq 2$ Lefschetz fibration $\pi$ on $\CPsum{}{(4g+5)}$ whose monodromy is the hyperelliptic involution on $\Sigma_g$, see chapter 3 of \cite{BKS}. We perform the $\Z_2$-construction to $\CPsum{}{(4g+5)}$ along a regular fiber $F$ of $\pi$ to obtain $X_g$, whose algebraic invariants are $e(X_g)=6g+6$, $\sigma(X_g)=-4g-4$, and $c_1^2(X_g)=0$. The $2$--handles coming from the monodromy of $\pi$ ensure that $\CPsum{}{(4g+5)}-F$ is simply-connected, so $X_g$ has an order two fundamental group. Now, an {\it untwisted} fiber sum of $\CPsum{}{(4g+5)}$ with itself along $F$ is diffeomorphic to $E(g+1)$, which is irreducible. (By contrast, the fiber-sum performed during the $\Z_2$-construction is twisted.) This shows that the pair $(\CPsum{}{(4g+5)},F)$ is irreducible, so $X_g$ is also irreducible. When $g$ is even, the signature of $X_g$ ensures that both $X_g$ and its universal cover have odd intersection forms. It follows that for $g=2k$, $X_g$ is an irreducible copy of $R_{2k,10k+4}$.

\subsubsection{$R_{2k,10k}$ for $k\geq 6$ (i.e. $c = 4$) \label{sec:2k_10k}} 

We begin by establishing some notation. Let $n\geq 2$ and $\pi_n\colon E(n)\to S^2$ denote the elliptic fibration on $E(n)$ with $12n$ Lefschetz singularities. Let $S_n$ denote the section of $\pi_n$ with self-intersection $-n$. To realize the $c_1^2=4$ line, we take a symplectic fiber sum of two elliptic surfaces in a {\it non-fiber-preserving} way.

First observe that the elliptic surface $E(4)$ has a symplectic embedded genus two surface $\Sigma$ with a trivial normal bundle, obtained by symplectically resolving the union of the section $S_4$ and two regular fibers of $\pi_4$. We claim that there exists a homologically essential Lagrangian torus in $E(4)$ disjoint from $\Sigma$. To see this, let $\gamma$ be a curve on the base diagram for $E(4)$ encircling twelve critical values of $\pi_4$, such that the monodromy associated to $\gamma$ is $(t_at_b)^6$, where $a$ and $b$ are the symplectic generators for $H_1(T^2;\Z)$. Then $\pi_4^{-1}(\gamma)$ is an embedded copy of $S^1\times F$, with $F$ a regular $T^2$ fiber. Set $\Lambda = \gamma \times a \subset \pi_4^{-1}(\gamma)$, where $a$ is the symplectic $H_1$ generator on $F$. This is a Lagrangian torus, and we claim that it is disjoint from $\Sigma$; the fact that $\Lambda$ can be made disjoint from regular fibers is clear immediately. To see that it is disjoint from $S_4$, observe that we may regard $\gamma$ as the curve where a fiber sum is taken between $E(1)$ and $E(3)$ to form $E(4)$. Let $F_1$ be the regular fiber of $E(1)$ where the fiber sum is taken. We may perturb the $(-1)$--section of $E(1)$ so that it is trivial, i.e. it restricts to $S^1\times \set{pt}$ on $\partial\nu(F_1)=\pi_4^{-1}(\gamma)\cong S^1\times F$. This ensures that the section will not intersect $\Lambda$ when the fiber sum is taken, as long as $pt \notin a$, which we can always assume to be true. Therefore $\Lambda$ is disjoint from $S_4$ and regular fibers of $\pi_4$, so is disjoint from $\Sigma$.

To see that $\Lambda$ is homologically essential, observe that $\pi_4$ restricted to the interior of $\gamma$ is a Lefschetz fibration over $D^2$ with $a$ and $b$ as vanishing cycles. The curve $b$ on $F\subset \pi_4^{-1}(\gamma)$ is a vanishing cycle, and so bounds a disk which will intersect $\Lambda$ transversely at a point. The same is true for $\pi_4$ restricted to the exterior of $\gamma$. The union of these two disks is a $(-2)$--sphere intersecting $\Lambda$ transversely at a point, showing that it is homologically essential. Moreover, this $(-2)$--sphere is disjoint from the regular fibers and the $(-4)$--section of $\pi_4$ \cite[Lemma 3.1.10]{GompfStip}, and so it is disjoint from $\Sigma$. With this established, we may perturb the symplectic form on $E(4)$ to assume that $\Lambda$ is symplectic.

Next let $E(n)_2$ be obtained from the elliptic surface $E(n)$ by performing a $2/1$--torus surgery along a regular fiber. $E(n)_2$ is simply-connected because $E(n)$ has a section. In \cite[Section 2]{GompfNuclei}, Gompf shows that the the former is minimal and has an odd intersection form. Let $T$ be a regular fiber of $E(n)_2$ and let $H$ be the surface of odd self-intersection in $E(n)_2$. Set $Y$ to be the symplectic fiber sum of $E(4)$ and $E(n)_2$ along $\Lambda$ and $T$, where after identifying $\nu(\Lambda)$ and $\nu(T)$ with $D^2\times T^2$, the gluing is given by the identity. The genus two surface $\Sigma$ survives the fiber sum, so we consider it a submanifold of $Y$. We claim that $Y$ has odd intersection form: as discussed in the previous paragraph, there is a $-2$-sphere intersecting $\Lambda$ transversely. For any intersection point between $H$ and $T$, we tube a copy of this $-2$ sphere with $H$ in the fiber sum to obtain a surface $\hat H$ with odd self-intersection in $Y$. Note that $\hat H$ is disjoint from $\Sigma$ because the $-2$-spheres are disjoint from $\Sigma$ in $E(4)$. Next we claim that $Y$ is simply-connected. Recall that $\Lambda$ intersects a sphere, so $\pi_1(E(4)-\Lambda)$ is trivial. Now, since $E(n)_2$ is simply-connected, $\pi_1(E(n)_2-T)$ is normally generated by $\mu_T$. By Seifert-Van Kampen, this meridian is null-homotopic in the fiber sum, confirming that $Y$ is simply--connected. Moreover, $Y-\Sigma$ is also simply-connected, which can be seen by describing an immersed sphere $S$ in $E(4)$ intersecting $\Sigma$ geometrically once, and far away from $\Lambda$. For example, construct $S$ as a generic fiber with a meridian and a longitude pinched using vanishing cycles. This sphere $S$ would intersect the section $S_4$ that forms part of $\Sigma$.

Now, if we perform the $\Z_2$-construction to $Y$ along $\Sigma$, we obtain a 4-manifold $X$ with $b_2^+=2n+8$, $b_2^-=10n+40$, and an order two fundamental group. Usher's theorem implies that $Y$ is minimal, hence so is $X$. Since the odd surface $H$ is disjoint from $\Sigma$, $X$ also has an odd intersection form. It follows that $X$ is an irreducible  copy of $R_{2n+8,10n+40}.$

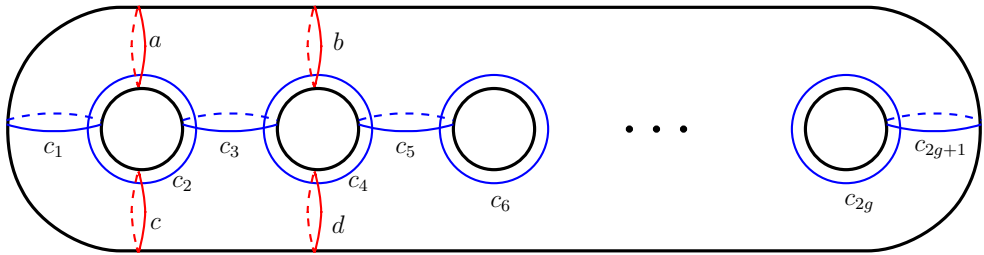
\begin{figure}[h!]
	\begin{tikzpicture}[scale=1.5]
		
		\begin{scope} [xshift=0cm, yshift=0cm, scale=0.6] 
			\draw[very thick,rounded corners=25pt] (6.5,1.8) --(-6.5,1.8) --(-7.45, 0)--(-6.5,-1.8) -- (6.5,-1.8)--(7.45,0) --cycle;
			\draw[very thick, xshift=-5.2cm] (0,0) circle [radius=0.6cm];
			\draw[very thick, xshift=-2.6cm] (0,0) circle [radius=0.6cm];
			\draw[very thick, xshift=0cm] (0,0) circle [radius=0.6cm];
			\draw[very thick, xshift=5.2cm] (0,0) circle [radius=0.6cm];
			\filldraw[very thick, xshift=2cm] (0,0) circle [radius=0.03cm];
			\filldraw[very thick, xshift=2.4cm] (0,0) circle [radius=0.03cm];
			\filldraw[very thick, xshift=2.8cm] (0,0) circle [radius=0.03cm];
			\draw[thick, blue, xshift=-5.2cm] (0,0) circle [radius=0.8cm];
			\draw[thick, blue, xshift=-2.6cm] (0,0) circle [radius=0.8cm];
			\draw[thick, blue, xshift=0cm] (0,0) circle [radius=0.8cm];
			\draw[thick, blue, xshift=5.2cm] (0,0) circle [radius=0.8cm];
			\draw[thick, blue, rounded corners=8pt, xshift=-7.8cm, yshift=0.1cm]             (0.6,-0.03) -- (0.9,-0.13)--(1.7,-0.13) --(2,-0.03);
			\draw[thick,  blue, dashed, rounded corners=8pt, xshift=-7.8cm, yshift=0.1cm] (0.6,0.03) -- (0.9,0.13)--(1.7,0.13) --(2,0.03);
			\draw[thick, blue, rounded corners=8pt, xshift=-5.2cm, yshift=0.1cm]             (0.6,-0.03) -- (0.9,-0.13)--(1.7,-0.13) --(2,-0.03);
			\draw[thick,  blue, dashed, rounded corners=8pt, xshift=-5.2cm, yshift=0.1cm] (0.6,0.03) -- (0.9,0.13)--(1.7,0.13) --(2,0.03);
			\draw[thick, blue, rounded corners=8pt, xshift=-2.6cm, yshift=0.1cm]             (0.6,-0.03) -- (0.9,-0.13)--(1.7,-0.13) --(2,-0.03);
			\draw[thick,  blue, dashed, rounded corners=8pt, xshift=-2.6cm, yshift=0.1cm] (0.6,0.03) -- (0.9,0.13)--(1.7,0.13) --(2,0.03);
			\draw[thick, blue, rounded corners=8pt, xshift=5.2cm, yshift=0.1cm]             (0.6,-0.03) -- (0.9,-0.13)--(1.7,-0.13) --(2,-0.03);
			\draw[thick,  blue, dashed, rounded corners=8pt, xshift=5.2cm, yshift=0.1cm] (0.6,0.03) -- (0.9,0.13)--(1.7,0.13) --(2,0.03);
			\draw[thick, red, rounded corners=8pt, xshift=-5.25cm, yshift=0.25cm, rotate = 90,scale = 0.75]             (0.6,-0.03) -- (0.9,-0.13)--(1.7,-0.13) --(2,-0.03);
			\draw[thick, red, dashed, rounded corners=8pt, xshift=-5.3cm, yshift=0.25cm, rotate = 90,scale = 0.75]             (0.6,-0.03) -- (0.9,0.07)--(1.7,0.07) --(2,-0.03);
			\draw[thick, red, rounded corners=8pt, xshift=-2.65cm, yshift=0.25cm, rotate = 90,scale = 0.75]             (0.6,-0.03) -- (0.9,-0.13)--(1.7,-0.13) --(2,-0.03);
			\draw[thick, red, dashed, rounded corners=8pt, xshift=-2.7cm, yshift=0.25cm, rotate = 90,scale = 0.75]             (0.6,-0.03) -- (0.9,0.07)--(1.7,0.07) --(2,-0.03);
			
			\draw[thick, red, rounded corners=8pt, xshift=-5.25cm, yshift=-2.2cm, rotate = 90,scale = 0.75]             (0.6,-0.03) -- (0.9,-0.13)--(1.7,-0.13) --(2,-0.03);
			\draw[thick, red, dashed, rounded corners=8pt, xshift=-5.3cm, yshift=-2.2cm, rotate = 90,scale = 0.75]             (0.6,-0.03) -- (0.9,0.07)--(1.7,0.07) --(2,-0.03);
			
			\draw[thick, red, rounded corners=8pt, xshift=-2.65cm, yshift=-2.2cm, rotate = 90,scale = 0.75]             (0.6,-0.03) -- (0.9,-0.13)--(1.7,-0.13) --(2,-0.03);
			\draw[thick, red, dashed, rounded corners=8pt, xshift=-2.7cm, yshift=-2.2cm, rotate = 90,scale = 0.75]             (0.6,-0.03) -- (0.9,0.07)--(1.7,0.07) --(2,-0.03);
			

			\node[scale=0.8] at (-6.5,-0.3) {$c_1$};
			\node[scale=0.8] at (-3.9,-0.3) {$c_3$};
			\node[scale=0.8] at (-1.3,-0.3) {$c_5$};
			\node[scale=0.8] at (6.6,-0.3) {$c_{2g+1}$};
			\node[scale=0.8] at (-4.6,-0.8) {$c_2$};
			\node[scale=0.8] at (-2,-0.8) {$c_4$};
			\node[scale=0.8] at (0.1,-1.1) {$c_6$};
			\node[scale=0.8] at (5.4,-1.1) {$c_{2g}$};
			\node[scale=0.8] at (-5,1.3) {$a$};
			\node[scale=0.8] at (-2.3,1.3) {$b$};
			\node[scale=0.8] at (-5,-1.4) {$c$};
			\node[scale=0.8] at (-2.3,-1.4) {$d$};

		\end{scope}
	\end{tikzpicture}
	\caption{Curves on $\Sigma_g$. } \label{fig:chain_of_curves}
\end{figure}

\subsubsection{$R_{g,5g+4-k}$ for all $g\geq 2$, $k\in \{1,2,3\}$ (i.e. $c\in \set{1,2,3}$)} \label{sec:c_1_2_3}

In this section we will take a fiber-reversing double of a genus $g\geq 2$ Lefschetz fibration with hyperelliptic positive factorization, apply several lantern relations, and then quotient by an involution. See \ref{section:fiber_reversing_double} for a reminder of the fiber-reversing construction. Before beginning, we clarify some notation: we let $t_1,t_2,\cdots t_{2g+1}$ denote the $2g+1$ mapping classes in $\Mod{(\Sigma_g)}$ which are Dehn twists around the maximal $2g+1$ chain of curves $c_i$ on $\Sigma_g$ as depicted in Figure \ref{fig:chain_of_curves}. When we write $t_i^{j}$, where $i$ and $j$ are integers, we mean the Dehn twist $t_i$ raised to the $j$--power. On the other hand, when we write $t_i^{t_j}$, we mean $t_jt_it_j^{-1}$.

In \cite[Section 3]{BKS}, the authors describe a genus $g$ Lefschetz fibration on $X_g=\CPsum{}{(4g+5)}$ with a hyperelliptic positive factorization $$(t_1t_2\cdots t_{2g}t_{2g+1}^2t_{2g}\cdots t_1).$$
They show that this factorization is Hurwitz equivalent to $W_g=A_gt_1^{2g+2}t_3^{2g+2},$ where $A_g = t_1^{t_2}t_2^{t_3}t_3^{t_4}\cdots t_{2g}^{t_{2g+1}} t_{2g+1}^{t_{2g}}\cdots t_4^{t_3} t_3^{t_2t_3^{2g+2}}t_2^{t_1t_3^{2g+2}}$. We conjugate $W_g$ with $\phi$, which is any diffeomorphism taking $(c_1,c_3)$ to $(a,b)$, which are shown in Figure \ref{fig:chain_of_curves}. Then we get positive factorization $A_g^{\phi}(t_at_b)^{2g+2}$, which has the same total space $X_g$. By Proposition \ref{prop:reverse_LF}, the fiber-reversing double $DX_g = X_g\cup X_g$ admits a Lefschetz fibration with positive factorization $$V_g= A_g^{\phi}(t_at_b)^{2g+2}(t_ct_d)^{2g+2}(\overline{A_g^{\phi}})^{r},$$
where $\overline{A_g^{\phi}}$ is the inverse of $A_g^\phi$ and $r$ is the reflection map sending the curves $a,b$ to $c,d$ respectively. Note that $a,b,c,$ and $d$ bound a four-punctured sphere. Since $W_g\in \Mod(\Sigma_g)$ can be represented by $\text{id}_{\Sigma_g} \in \text{Diff}(\Sigma_g)$, which of course commutes with $r$, by Corollary \ref{cor:reverse_LF_involution}, there is an involution $\iota'$ on $DX_g$ which interchanges two copies of $X_g$ and restricts to $a\times r$ on the gluing region, which we identify with $S^1\times \Sigma_g$.

\begin{figure}
	\begin{tikzpicture}
		
		\draw[thick] (0,0) ++(-30:3) arc[start angle=-30, end angle=30, radius=3];
		\draw[thick] (0,0) ++(60:3) arc[start angle=60, end angle=120, radius=3];
		\draw[thick] (0,0) ++(150:3) arc[start angle=150, end angle=210, radius=3];
		\draw[thick] (0,0) ++(240:3) arc[start angle=240, end angle=300, radius=3];
		\draw[thick, rotate=45] (2.9,0) ellipse [x radius=5pt, y radius=22pt];
		\draw[thick, rotate=135] (2.9,0) ellipse [x radius=5pt, y radius=22pt];
		\draw[thick, rotate=225] (2.9,0) ellipse [x radius=5pt, y radius=22pt];
		\draw[thick, rotate=315] (2.9,0) ellipse [x radius=5pt, y radius=22pt];
		
		\draw[thick, red, rotate = 45] (2.8,-1.1) .. controls (2.5,0) .. (2.8,1.1);
		\draw[thick, red, rotate = 45, dashed] (2.8,-1.1) .. controls (2.2,0) .. (2.8,1.1);
		\node[red] at (1.3,3) {$b$};
		
		\draw[thick, red, rotate = 135] (2.8,-1.1) .. controls (2.5,0) .. (2.8,1.1);
		\draw[thick, red, rotate = 135, dashed] (2.8,-1.1) .. controls (2.2,0) .. (2.8,1.1) ;
		\node[red] at (-1.3,3) {$a$};

		\draw[thick, red, rotate = 225] (2.8,-1.1) .. controls (2.5,0) .. (2.8,1.1) ;
		\draw[thick, red, rotate = 225, dashed] (2.8,-1.1) .. controls (2.2,0) .. (2.8,1.1) ;
		\node[red] at (-1.3,-3) {$c$};
		
		\draw[thick, red, rotate = 315] (2.8,-1.1) .. controls (2.5,0) .. (2.8,1.1) ;
		\draw[thick, red, rotate = 315, dashed] (2.8,-1.1) .. controls (2.2,0) .. (2.8,1.1) ;
		\node[red] at (1.3,-3) {$d$};
		
		\draw[thick, magenta] (0,-3) .. controls (0.5,0) .. (0,3);
		\draw[thick, magenta, dashed] (0,-3) .. controls (-0.5,0) .. (0,3);
		\node[magenta] at (0,3.3) {$x$};
		\draw[thick, olive, rotate = 90] (0,-3) .. controls (0.5,0) .. (0,3);
		\draw[thick, olive, dashed, rotate = 90] (0,-3) .. controls (-0.5,0) .. (0,3);
		\node[olive] at (3.3,0) {$y$};
		\draw[thick, cyan, dashed] (-2.8,1) .. controls (-1.2,1.2).. (-1,2.8);
		\draw[thick, cyan, dashed, rotate = 180] (-2.8,1) .. controls (-1.2,1.2).. (-1,2.8);
		\draw[cyan, thick] (1,-2.8)-- (-2.8,1);
		\draw[cyan, thick,rotate = 180] (1,-2.8)-- (-2.8,1);
		\node[cyan] at (1.1,1.1) {$z$};

	\end{tikzpicture}
	\caption{Dehn twist curves of the lantern relation}
	\label{fig:lantern}
\end{figure}

We want to apply the lantern substitution to $t_at_bt_ct_d$ several times, then use Corollary \ref{cor:reverse_LF_involution} to claim that even after these substitutions, the manifold still admits an orientation--preserving free involution. Let $t_at_bt_ct_d =t_xt_yt_z$ as in Figure \ref{fig:lantern}. Then after two lantern substitutions we get a new positive factorization
$$V_{g,2}= A_g^{\phi}(t_at_b)^{2g}(t_xt_yt_zt_xt_yt_z)(t_ct_d)^{2g}(\bar A_g^{\phi})^{r}.$$
Let $DX_{g,2}$ denote the total space of the Lefschetz fibration whose positive factorization is $V_{g,2}$. To obtain an involution on $DX_{g,2}$, first we show that $$t_xt_yt_zt_xt_yt_z=t_xt_yt_z(\overline{t_xt_yt_z})^r,$$
or equivalently 
$$t_xt_yt_zt_xt_yt_z=t_xt_yt_zt_{r(z)}t_{r(y)}t_{r(x)}.$$
Since $r(x)=x$ and $r(y)=y$, this is the same as showing $t_xt_yt_z = t_{r(z)}t_y t_x$, which is equivalent to the following:
\begin{align*}
		t_at_bt_ct_d &= t_{r(z)}t_yt_x\\	
		t_at_bt_ct_d &= t_xt_{r(z)}t_y\\
		t_xt_yt_z &= t_xt_{r(z)}t_y\\
		t_yt_z &= t_{r(z)}t_y\\
		t_yt_zt_y^{-1} &= t_{r(z)}\\
		t_{t_y(z)} &= t_{r(z)} \\
		t_y(z) &\sim r(z) \\
\end{align*} 
and the latter is easy to check. So we can rewrite the positive factorization as \[
	V_{g,2}= A_g^{\phi}(t_at_b)^{2g}(t_xt_yt_z)\left(\overline{A_g^{\phi}(t_at_b)^{2g}(t_xt_yt_z)}\right)^r = \mu_{g,2} \circ r \circ (\mu_{g,2})^{-1}\circ r^{-1},
\] where $\mu_{g,2} = A_g^{\phi}(t_at_b)^{2g}(t_xt_yt_z)$. Moreover, we can find a diffeomorphism representing $\mu_{g,2}$ that commutes with $r$: since $A_g^{\phi}(t_at_b)^{2g+2}=1$,  $\mu_{g,2}=(t_xt_yt_z)(t_at_b)^{-2}=t_b^{-1}t_a^{-1}t_ct_d.$ Since $a$ and $c$ are isotopic, this reduces to $t_b^{-1}t_d$. The diffeomorphism $t_b^{-1}t_d$ represents $\mu_{g,2}$, and we claim it commutes with $r$. This follows from the fact that both $rt_b^{-1}t_dr^{-1}$ and $t_b^{-1}t_d$ are supported in collar neighborhoods of $b$ and $d$, so it is sufficient to check that $rt_b^{-1}t_dr^{-1}=t_b^{-1}t_d$ in $\nu(b\cup d)$ which is straightforward. Hence, we may apply Corollary \ref{cor:reverse_LF_involution} to say that $DX_{g,2}$ admits an orientation--preserving free involution. As an aside, we see from Corollary \ref{cor:nontrivial_reverse_LF} that $DX_{g,2}$ is two copies of a Lefschetz fibrations over $D^2$ glued along their boundaries, which are mapping tori with $\mu_{g,2}$ monodromy.

This argument extends for more than two lanterns, and so shows that there is an involution on $DX_{g,2k}$ for $0\leq k \leq g+1$, where $2k$ corresponds to the number of lantern substitutions. To summarize, we have shown that the manifold $DX_{g,2k}$ admits a Lefschetz fibration with the following positive factorization:\[
V_{g,2k} = A_g^{\phi}(t_xt_yt_z)^k(t_at_bt_ct_d)^{2(g+1-k)}r\left(\overline{A_g^{\phi}(t_xt_yt_z)^{k}}\right)r^{-1}.
\] This is a fiber--reversing double of a Lefschetz fibration on $X_{g,2k}\to D^2$ whose monodromy is \[
A_g^{\phi}(t_xt_yt_z)^k(t_at_b)^{2(g+1-k)}.
\] The fiber--reversing double admits an orientation-preserving free involution, which is precisely the map $\iota$ that we quotient by.

Let $Y_{g,2k}$ be the quotient of $DX_{g,2k}$ by $\iota$. Since the number of Dehn twists in the positive factorization after a lantern substitution reduces by one, so does the Euler characteristic of the corresponding Lefschetz fibration. On the other hand, the signature increases by one \cite[Prop.3.12]{Endo2004SignatureOR}. For the sake of convenience, we list all the algebraic invariants of $Y_{g,2k}$: \begin{align*}
	&&e(Y_{g,2k})=6g+6-k, && \sigma(Y_{g,2k}) = -4-4g+k, && c_1^2(Y_{g,2k}) = k, \\ 
	&& b_2^+(Y_{g,2k}) = g, && b_2^- (Y_{g,2k}) = 5g+4-k, && \chi_h(Y_{g,2k}) = \frac{g+1}{2}.
\end{align*}

Restricting to $k=1,2,3$ gives us the points on the desired lines. We just need to check if the manifold is irreducible, has a fundamental group of order 2, and has the correct $w_2$--type. The fact that double cover is not spin is a simple consequence of Rokhlin's theorem. Indeed, for $k\in \{1,2,3\}$, $\sigma(Y_{g,2k}) \equiv 1,2,3\bmod{4}$. So its universal cover's signature is congruent to $2,4,6 \bmod 8$. But by Rokhlin, spin manifold's signature is divisible by 16. For irreducibility, $DX_g$ is symplectic, and is minimal by Usher's theorem. Since the lantern substitution amounts to a symplectic blowdown, $DX_{g,2k}$ is minimal \cite[Lemma 1.1]{Dorfmeister} and consequently irreducible \cite{Hamilton_2006}. Hence, $Y_{g,2k}$ is also irreducible.

For $\pi_1(Y_{g,2k}·)$, note that if $k \in\{1,2\}$, then all the vanishing cycles of $X_g$ are in $DX_{g,2k}$, which makes the latter simply-connected for any $g$. This also holds when $k=3$ and $g>2$. Hence whenever $(k,g)\neq (3,2)$, $Y_{g,2k}$ has order two fundamental group. When $(k,g)=(3,2)$, $DX_{g,2k}$ is the fiber-reversing double of a Lefschetz fibration whose positive factorization is \[
\left(t_1^{t_2}t_2^{t_3}t_3^{t_4}t_{4}^{t_5}t_{5}^{t_4}t_4^{t_3}t_3^{t_2(t_3)^6}t_2^{t_1(t_3)^6}(t_xt_yt_z)^3\right)^\phi.
\] It is no longer the case that all of the vanishing cycles of $X_2$ are in $DX_{2,6}$, since $X_2$ has $a$ and $b$ as vanishing cycles. Nevertheless, the subfactorization $t_1^{t_2}t_2^{t_3}t_3^{t_4}t_{4}^{t_5}t_{5}^{4}$ from above guarantees simply--connectedness of $DX_{2,6}$: This is because the first four vanishing cycles provide relations $c_2c_1^{-1}=c_3c_2^{-1}= c_4c_3^{-1}=c_5c_4^{-1}=1$ in the fundamental group, so that $\pi_1(DX_{2,6})$ is cyclically generated by $c_1$. Then the last vanishing cycle above gives the relation $c_4c_5=1$, meaning $c_1^2=1$. Since $\pi_1$  is abelian, it is isomorphic to $H_1(DX_{2,6},\Z)$. In $H_1(\Sigma_2;\Z)$, we have the relation $c_1+c_3+c_5=0$, which becomes $3c_1=0$. This together with $2c_1=0$ shows that $DX_{g,2k}$ is simply-connected. Hence $Y_{2,6}$ has order two fundamental group. This concludes our construction for the lines $c_1^2=1,2,3$. In particular, we have found symplectic, minimal copies of $R_{g,5g+4-k}$ for all $g\geq 2$, $k\in \{1,2,3\}$.

\noindent {\bf Current progress of lattice points.} At this point, the following coordinates with $b_2^+$ even remain missing:\begin{itemize}
	\item The aforementioned points missed in \ref{section:big_region}:\[
		(2,9),(2,7),(2,5), (4,19),(4,17),(4,15),(4,13),(4,11),(4,9),(4,7), (6,15),(6,13),(6,11).
	\]
	
	\item The first few points along the line $c_1^2=0$: $(2,10), (4,20), (6,30), (8,40), (10,50)$.
\end{itemize}

We fill in some of these missing points before moving on to tackle the $\Z_2$--geography with odd $b_2^+$.

\subsection{Additional missing points.}

\subsubsection{Irreducible $R_{4,19}$} In \cite[Example 5.4]{GompfNew}, Gompf constructs an irreducible copy of $\CPsum{3}{18}$ and denotes it $S_{1,1}$. This manifold contains an embedded genus two surface $F$ of square zero whose complement is simply-connected. We claim that the $\Z_2$-construction to $S_{1,1}$ along $F$ is an irreducible copy of $R_{4,19}$. The triviality of $\mu_F$ ensures that the $\Z_2$--construction gives a $4$--manifold with order two fundamental group. The signature is not divisible by eight, ensuring that the intersection form is odd. Irreducibility follows from Usher's theorem.

\subsubsection{Irreducible $R_{4,17}$.} In \cite[Example 5.10]{GompfNew}, Gompf constructs $R_{2,1}=\Fibersumd{P_2}{Q_1}{\Sigma_2}$, which is an irreducible copy of $\CPsum{3}{16}$. It contains a genus two surface $F$, the push-off of $\Sigma_2$ with a trivial normal bundle, such that $\pi_1(R_{2,1}\setminus \nu(F))=1$. This follows from the fact that $\Sigma_2$ intersects immersed spheres in both $P_2$ and $Q_1$. It holds for $P_2$ because $P_2-\Sigma_2$ is simply-connected \cite{GompfNew}. And it holds for $Q_1$ because of the exceptional sphere. By performing the $\Z_2$-construction on $R_{2,1}$ along $F$, we get a $4$-manifold with $(b_2^+,b_2^-)=(4,17)$. Since its signature is not divisible by eight, its intersection form is odd. Minimality follows from Usher's theorem. This confirms that the $\Z_2$-construction gives as irreducible copy $R_{4,17}$.

\subsubsection{Irreducible $R_{4,n}$ for $n \in \{ 7,9,11 \}$. \label{sec:b+=4}}

In \ref{section:big_region}, we saw that there are $4$-manifolds $B$, $C$, and $D$ which, after performing two Luttinger surgeries, are irreducible copies of $\CPsum{}{3}$, $\CPsum{}{5}$, and $\CPsum{}{7}$ respectively. Let $B'$, $C'$, and $D'$ denote these simply-connected exotic copies, which are studied extensively in \cite{BK}. For each $X\in \{B',C',D'\}$, $X$ contains a genus two surface $H$ with a trivial normal bundle, such that $X-H$ is simply--connected. In this section, we take the fiber sum of $X$ along $H$ with the a manifold $Z'$, which is defined in \cite[Section 4]{akhmedov2009exotic} and is obtained through one Luttinger surgery on $T^4\#\CPbar{}$. There exists a square zero genus two surface $\bar \Sigma$ in $Z'$. As discussed in \cite[Section 12]{akhmedov2009exotic}, the fiber sum $Y_k = X\#_{H=\bar \Sigma} Z'$ is an irreducible, symplectic copy of $\CPsum{3}{(k+3)}$ for $k\in\{3,5,7\}.$ (Here $k=3,5,7$ correspond to $X=B', C', D'$ respectively.) Let $H'$ be a push-off of $H\subset X$, which survives  as an embedded genus two surface in $Y_k$, and let $W_k$ be the $\Z_2$-construction of $Y_k$ along $H'$. The algebraic invariants of $W_k$ are $b_2^+(W_k)=4$ and $b_2^-(W_k)=k+4$. Irreducibility of $W_k$ follows from Usher's theorem, and since $\sigma(W_k)=k$ is odd, $W_k$ has odd intersection form. Hence if $\pi_1(W_k)\cong \Z_2$, then $W_k$ is an irreducible copy of $R_{4,k+4}$. To confirm this, we proceed by showing that $Y_k-H'$ is simply--connected.

From Seifert-Van Kampen, we see that $Y_k-H'$ is a quotient of the free product $\pi_1(X-(H\cup H'))*\pi_1(Z'-\bar \Sigma)$. Because $\pi_1(X-H)=1$, $\pi_1(X-(H\cup H'))$ is normally generated by the meridian $\mu_{H'}$, which is homotopic to $\mu_{H}$ and hence identified with $\mu_{\bar\Sigma}$ in $\pi_1(Y_k-H')$. On the other hand, the following information about $\pi_1(Z'-\bar \Sigma)$ is provided in \cite{akhmedov2009exotic}.

\begin{proposition}\cite[Corollary 7]{akhmedov2009exotic}
	$\pi_1(Z'-\bar \Sigma)$ is a quotient of \[
	\big\langle \alpha_1,\alpha_2,\alpha_3,\alpha_4, [\alpha_3,\alpha_4]^\phi\,\big|\, \alpha_3=[\alpha_1^{-1}, \alpha_4^{-1}], [\alpha_1,\alpha_3]=[\alpha_2,\alpha_3]=[\alpha_2,\alpha_4]=1 \big\rangle,
	\] where $[\alpha_3,\alpha_4]^\phi$ is any conjugate of $[\alpha_3,\alpha_4]$. If $\{a_1,b_1,a_2,b_2\}$ is the standard generating set for the push-off fundamental group $\pi_1(\bar \Sigma^{||})$, then the homomorphism induced by $\bar \Sigma^{||} \hookrightarrow Z'-\bar \Sigma$ maps $a_1\mapsto \alpha_1$, $b_1\mapsto \alpha_2$, and $b_2\mapsto \alpha_4$.
	
	\label{prop:Z'_pi1}
\end{proposition}

Hence $\pi_1(Y_k-H')$ is normally generated by $\alpha_1,\alpha_2,\alpha_3,$ and $\alpha_4$. The meridian $\mu_{H'}=\mu_{H}=\mu_{\bar\Sigma}$ is identified with a conjugate of $[\alpha_3,\alpha_4]$. We will continue to denote the standard generators of $\pi_1(H')$ by $G=\{a_1,b_1,a_2,b_2\}$, and let $i\colon \pi_1(H')\to \pi_1(X-(H\cup H'))$ be the inclusion-induced homomorphism. (We are abusing notation slightly by conflating $H'$ and its push-off that survives in $X-(H\cup H')$.) We first claim that $i(\pi_1(H'))$ is trivial in $\pi_1(X-(H\cup H'))$. To see this, note that since $X-H$ is simply--connected, the image of $G$ in $\pi_1(X-H)$ is trivial. Since $H'$ is a parallel copy of $H$, the immersed disks in $X-H$ contracting each element of $\pi_1(H')$ generator survive in $X-(H\cup H')$, so each $i(g)$ remains nullhomotopic even after removing $H'$. Hence $i(G)=i(\pi_1(H'))=1$. Then returning to the amalgamated free product forming $\pi_1(Y_k-H')$, since $a_1,b_1,b_2$ are identified with $\alpha_1,\alpha_2,\alpha_4$ in $\pi_1(Z'-\bar\Sigma)$, respectively, we see that $\alpha_1=\alpha_2=\alpha_4=1$ in $\pi_1(Y_k-H')$. With the other relations given in Proposition \ref{prop:Z'_pi1}, it follows easily that $\pi_1(Y_k-H')=1$. This completes the proof showing that $W_k$ is an irreducibly copy of $R_{4,k+4}$.

\subsubsection{$R_{g,5g}$ for $g\in \{4,6,8,10\}$.} We apply the fiber--reversing double, as in \ref{sec:c_1_2_3}, to get a few of the missing points on the $c_1^2=4$ line. Recall that for $0\leq k\leq g+1$, there exists a Lefschetz fibration on a manifold $DX_{g,2k}$ whose positive factorization is \[
V_{g,2k} = A_g^{\phi}(t_xt_yt_z)^k(t_at_bt_ct_d)^{2(g+1-k)}r\left(\overline{A_g^{\phi}(t_xt_yt_z)^{k}}\right)r^{-1}.
\] In \ref{sec:c_1_2_3}, we set $k\in\{1,2,3\}$, but here we set $k=4$. Just as was true when $k\in \{1,2,3\}$, $DX_{g,2k}=DX_{g,8}$ admits an orientation--preserving free involution. By the same reasoning as in \ref{sec:c_1_2_3}, the quotient of $DX_{g,2k}$ by this involution, which we call $Y_{g,8}$, has algebraic invariants given by $b_2^+(Y_{g,8})=g$ and $b_2^-(Y_{g,8})=5g$. In addition, $Y_{g,8}$ is irreducible. To verify that $\pi_1(Y_{g,8})\cong \Z_2$, recall that $DX_{g,8}$ is obtained through eight lantern substitutions on a genus $g$ Lefschetz fibration with positive factorization \[
V_{g} = A_g^{\phi}(t_at_bt_ct_d)^{2g+2}r\left(\overline{A_g^{\phi}}\right)r^{-1}.
\] The total space of this original Lefschetz fibration is simply--connected. Since all of the vanishing cycles in $V_g$ remain vanishing cycles in $V_{g,8}$, $DX_{g,8}$ is also simply--connected. Therefore $Y_{g,2k}$ has order two fundamental group. To verify the $w_2$--type, we use the fact shown in \cite[Section 3]{BKS}, that $DX_{g,2k}$ is non-spin. This completes the proof that $Y_{g,8}$ is an irreducible copy of $R_{g,5g}$.

\subsection{$R_{6,13}$.}

Let $f:X\to S^2$ be the minimal genus two Lefschetz fibration in \cite[Theorem 16]{Bay_Kork} whose total space is homeomorphic to $\CPsum{}{8}.$  Then let $X'$ be a fiber sum of $X$ and $\Sigma_2\times \Sigma_2$ along regular $\Sigma_2$ fibers, which has invariants $e(X')=19$, $\sigma(X')=-7$. Note that $f$ extends to a minimal Lefschetz fibration $\hat f:X'\to B$, where $B$ is a genus two surface. Let $\mathcal D$ be a disk in $B$ containing all of $\text{Crit}(f)$, and let $B^\circ = B-\mathcal D$. Let $a_1,b_1,a_2,b_2$ be standard symplectic generators of $H_1(B^\circ;\Z)$. Let $F$ be a regular fiber over $B^\circ$, and let $c_1,d_1,c_2,d_2$ be the standard symplectic generators of $H_1(F;\Z)$. Let $c'_1$ be a parallel copy of $c_1$. Then $\hat f^{-1}(B^\circ) \cong F\times B^\circ.$ With this identification, we may regard $a_1\times c_1$, $b_1\times c'_1$, $a_2\times c_1$, and $b_2\times c'_1$ as disjoint Lagrangian tori. Let $Y$ be obtained from $X'$ through the following Luttinger surgeries, where all framings are Lagrangian:\[
(a_1\times c_1,a_1,1),\,\, (b_1\times c'_1,b_1,1),\,\, (a_2\times c_1,a_2,1),\,\, (b_2\times c'_1,b_2,1).
\]
We next compute $\pi_1(Y)$. For this, we first observe that $\pi_1(Y)$ is obtained from adding relations to $\pi_1(F\times B^\circ -\mathcal N)$, where $\mathcal N$ is a union of neighborhoods where the Luttinger surgeries were performed. Proposition \ref{prop:sigmak_times_sigma2} then implies that $\pi_1(Y)$ is normally generated by $a_i,b_i,c_i,d_i$ for $i=1,2$. Moreover,  $\pi_1(Y)$ is a quotient of $\pi_1((F\times B) -\mathcal N) /N(\gamma_1,\cdots, \gamma_n)$, where $N(\gamma_1,\cdots, \gamma_n)$ is the subgroup normally generated by the vanishing cycles of $f$. Since $X$ is simply--connected, $\pi_1(F)/N(\gamma_1,\cdots \gamma_n)=1$, so $c_1=d_1=c_2=d_2=1\in \pi_1(Y)$. Next, from the Luttinger surgery, we have the following relations:\[
\{b_1,d_1\}a_1=1,\,\, \{a_1,d_1\}b_1=1,\,\,\{b_2,d_1\}a_2=1,\,\,\{a_2,d_1\}b_2=1,
\] where $\{x,y\}$ denotes a commutator of $x^{\pm 1}$ and $y^{\pm 1}$. Since all generators of $\pi_1(F)$ are trivial, the above relations reduce to $a_1=b_1=a_2=b_2=1$, so $\pi_1(Y)=1$.  Hence $Y$ is an irreducible copy of $\CPsum{5}{12}$. Moreover, if $F$ is a regular fiber far from Luttinger surgeries, $Y-F$ is simply--connected; this is because $\mu_{F}=[a_1,b_1][a_2,b_2]$, and the Luttinger surgeries that kill each of these elements happen away from $F$.  Hence, the $\Z_2$--construction to $Y$ along $F$ is an irreducible copy of $R_{6,13}$.

\subsection{$R_{6,15}$}

This is similar to the last construction for $R_{6,13}$, and in fact a little easier. In \cite[Theorem 20]{Bay_Kork}, they construct a minimal genus two Lefschetz fibration $f:X\to S^2$ whose total space is an irreducible copy of $\CPsum{3}{12}$. Suppose we take the fiber sum of $X$ with a copy of $T^2\times \Sigma_2$ along regular $\Sigma_2$ fibers, to get a minimal Lefschetz fibration $f':X'\to T^2$ with the same positive factorization as $f$. The total space $X'$ has $e(X')=21$, $\sigma(X')=-9$, and $\pi_1(X')\cong \pi_1(T^2)\cong \Z\oplus \Z$. In the same fashion as in \cite[Section 5.3]{Bay_Kork}, one can perform Luttinger surgeries on $X'$ to get a minimal simply--connected manifold $Y$. To this end, let $F$ be a regular fiber, and let $c_1,d_1,c_2,d_2$ be the standard $H_1$ generators on $F$. Let $a,b$ generate $H_1(X';\Z)$, which are geometrically dual curves on the $T^2$ base of $f'$. Let $Y$ be the manifold obtained from the following Luttinger surgeries, where the tori are equipped with Lagrangian framing:\[
(c_1\times a,a,1), (c_2\times b,1)
\]

The relations coming from $\pi_1(X)$ ensure that each curve on $F$ is null-homotopic, hence $c_i=d_i=1\in \pi_1(Y)$. The Luttinger surgeries add the relations $a\mu_1=1$ and $b\mu_2=1$, where $\mu_1$ and $\mu_2$ are commutators of $\{d_1,b\}$ and $\{d_2,a\}$, respectively. Since $c_i=d_i=1$, these relations kill $a$ and $b$ as well, so $Y$ is simply--connected. Hence, $Y$ is an irreducible copy of $\CPsum{5}{14}$. Let $F$ be a regular fiber of $X'$, which is far from the surgered tori. The Luttinger surgery relations ensure that $\mu_{F'}=[a,b]=1\in \pi_1(Y-F')$. Hence, the $\Z_2$-construction of $Y$ along $F$ is an irreducible copy of $R_{6,15}$.

\subsection{$R_{6,11}$.} Let $\hat B$ denote the minimal symplectic copy of $\CPsum{}{3}$ obtained through Luttinger surgery on the building block $B$ from Section \ref{section:big_region}. Let $Z$ be the minimal symplectic copy of $\CPsum{}{4}$ given by \cite[Theorem 24]{Bay_Kork}. Both $Z$ and $\hat{B}$ admit square zero symplectic genus two surfaces with simply--connected complements, denoted $H$ and $\Sigma$, respectively, by Proposition \ref{prop:B_meridian} and the existence of a section. It follows that the fiber sum $Y=Z\#_{\Sigma=H} \hat{B}$ is a symplectic minimal copy of $\CPsum{5}{10}$. Let $H'$ be a push-off of $H$ which survives in $Y$. We claim that the $\Z_2$--construction of $Y$ along $H'$ is an irreducible copy of $R_{6,11}$. The $w_2$--type follows from its signature being odd, and irreducibility follows from Usher's theorem. The only thing that remains to verify is that $Y-H'$ is simply--connected. For this, we use the fact that $\pi_1(Y-H')$ is a quotient of $\pi_1(\hat B-(H\cup H'))\ast \pi_1(Z-\Sigma)\cong \pi_1(\hat B-(H\cup H'))\ast 1$. The gluing map trivializes $\mu_{H}$, so this reduces to $\pi_1(\hat B-H')$, which we've already established to be trivial.

\section{Irreducible copies of $R_{a,b}$ when $a$ is odd} \label{sec:odd_constructions}

Recall the building blocks $B$, $C$, $D$, and $B_g$ from \cite{ABBKP}, which we used to construct telescoping triples $(Z,T_1,T_2)$ in section \ref{section:big_region}. By performing Luttinger surgeries to kill the fundamental groups and applying the $\Z_2$-construction on $Z$ over an embedded genus two surface, we obtained irreducible manifolds with order two fundamental group and even $b_2^+$ coefficients. In most of this section, we use the same telescoping triples $(Z,T_1,T_2)$, but only perform Luttinger surgery rather than a $\Z_2$--construction. This has the effect of changing $\pi_1$, but keeps $b_2^+$ odd. There are instances when Luttinger surgery is not an immediate option, and instead we perform $2/1$ torus surgeries on the elliptic surface $E(n)$, studied extensively in \cite{GompfNuclei}. 

The $\Z_2$--geography of irreducible $4$--manifolds was studied by Torres in \cite{Torres_2014}, who also used Luttinger surgery on minimal symplectic $4$--manifolds. Nonetheless, we include these constructions so that our $\Z_2$--geography is self--contained and complete, and also to suggest alternative constructions for coordinates with $b_2^+$ odd. To begin, we recall a construction thoroughly examined in \cite{GompfNuclei} and \cite{FriedMor}, which will serve as a building block in this section.

\begin{proposition}
	Let $E(n)_{2,2}$ denote the manifold obtained by performing two $2/1$-torus surgeries on the elliptic surface $E(n)$ along regular fibers. Then $E(n)_{2,2}$ is a symplectic manifold with order two fundamental group. It has the same Euler characteristic and signature as $E(n)$. $E(n)_{2,2}$ is minimal when $n\geq 2$.
		
	\label{prop:En_log_transform}
\end{proposition}

\begin{corollary}
	Let $F$ be a regular fiber of $E(1)$. The pair $(E(1)_{2,2},F)$ is relatively minimal.
\end{corollary}

\begin{proof}
	If $(E(1)_{2,2},F)$ were not relatively minimal, than the fiber sum of $E(1)_{2,2}$ and $E(1)$ along regular fibers would be non-minimal. But this fiber sum is precisely $E(2)_{2,2}$, so this would contradict Proposition \ref{prop:En_log_transform}.
\end{proof}

Before starting constructions, we give a useful tool for fundamental group computations.

\begin{proposition} \label{prop: 1/2 surgery}
	Let $(Z,T_1,T_2)$ be a telescoping triple. Then there are curves on $T_1$ and $T_2$ such that performing $1/1$-surgery on $T_1$ and $1/2$-surgery on $T_2$ along these curves results in a manifold with a fundamental group of order two.
\end{proposition}

\begin{proof}
	Let $\{a_i,b_i\}$ be a pair of geometrically dual curves on each $T_i$. It follows from the definition of a telescoping triple that the meridians of $T_i$ are trivial in $Z-(T_1\cup T_2)$, and that we may assume $b_2$ generates $\set{0} \times \Z \subset \Z^2\cong \pi_1(Z)$, $a_1,a_2$ generate $\Z\times \set{0}\subset \Z^2\cong \pi_1(Z)$, and $b_1$ is trivial. In particular, $\pi_1(Z)\cong \pi_1(Z-(T_1\cup T_2)) = \langle a_1,b_2|\,[a_1,b_2]\rangle$. When we perform the torus surgery along $T_1$, we attach $\{pt\}\times \partial D^2\subset T^2\times D^2$ along the curve $\mu_{T_1} a_1\in \pi_1(\partial \nu(T_1))$, which introduces the relation $\mu_{T_1} a_1 = 1$. Hence $a_1 = 1$. Similarly, performing $+1/2$ surgery on $T_2$ along $b_2$ introduces the relation $\mu_{T_2} b_2^2 = 1$, and so $b_2^2 = 1$. Then the fundamental group after both surgeries is $ \langle a_1, b_2 \mid [a_1,b_2], a_1, b_2^2  \rangle \cong \Z_2.$
\end{proof}

\begin{rmk}\label{rmk:other_pi1}
	By changing the surgery coefficient on $T_2$ to $1/n$, this proof generalizes to produce a $4$--manifold with any cyclic fundamental group, including $n=0$.
\end{rmk}

\subsection{{\bf Main constructions for the lattice points with odd $\mathbf{b_2^+}$:}}

We go through the proof of \cite[Theorem 22]{ABBKP}, where they construct manifolds $X_{a,b}$ which are irreducible copies of $\CPsum{a}{b}$. With slight modifications, this same proof can construct a symplectic irreducible manifold $Y_{a,b}$ which has all the same $b_2^{\pm}$ and intersection form as $X_{a,b}$, but has  order two fundamental group. In this section, it will often be more convenient to work with the coordinates $(c_1^2,\chi_h)$ rather than $(b_2^+, b_2^-)$. Written in terms of $c_1^2$ and $\chi_h$, this subsection aims to populate the plane $\{(c_1^2,\chi_h)|\, 0\leq c_1^2\leq 8\chi_h-2\}$. Just like in \cite[Theorem 22]{ABBKP}, we start by finding points with even $c_1^2$.

\underline{\underline{$c_1^2$ even.}} Let $X_{a,b}$ be a $4$-manifold constructed in \cite[Theorem 22]{ABBKP} with $c_1^2$ even. Let $Z$ be a telescoping triple constructed using building blocks $B,B_g,C,D$. Then $X_{a,b}$ is obtained in one of the following three ways: (1) $X_{a,b}$ is a result of two Luttinger surgeries on $Z$, (2) $X_{a,b}$ is a result of a single Luttinger surgery on $Z\#_{T_2=Fiber} E(k)$, or (3) $X_{a,b}$ is a copy of $E(k)_{2,3}$, which is $E(k)$ with two torus surgeries. We address each of these cases in order.\begin{enumerate}[(1)]
	\item To get same invariants as $X_{a,b}$, but with order two fundamental group, we perform Luttinger surgery differently to obtain $Y_{a,b}$, as in Proposition \ref{prop: 1/2 surgery}. This change will still give odd manifolds because the odd surface in $Z-T_1-T_2$ is not affected.
	
	\item To get $Y_{a,b}$ with the invariants covered by the second family, i.e. $Z\#_{T_2=Fiber} E(k)$, we slightly diverge from their construction, and take $Z\#_{T_1=Fiber} E(k)$ instead, followed by $1/2$ Luttinger surgery on $T_2$. Similar to the proof of Proposition \ref{prop: 1/2 surgery}, $\pi_1(Y_{a,b})\cong \Z_2$. To see this note that $\pi_1(Z\#_{T_1=Fiber} E(k))$ is generated freely by $b_2$, and after the surgery we obtain a relation $b_2^2$. The manifold is still odd because again there is an odd surface in $Z-(T_1\cup T_2)$.
	
	\item In this case, $c_1^2(X_{a,b})=0$. We cannot just change the simply-connected manifold $E(k)_{2,3}$ to $E(k)_{2,2}$, because the latter may be even. So getting irreducible copies along the $c_1^2=0$ line involves a bit more work. For this, we perform the $\Z_2$-construction to $E(n)_2$ for $n>1$ as follows: consider the Lagrangian, homologically essential torus $\Lambda\subset E(n)$ from Section \ref{sec:2k_10k}. After a perturbation, we may regard $\Lambda$ as a symplectic submanifold. As Gompf shows in \cite{GompfNuclei}, performing 2-torus surgery on a regular fiber of $E(n)$ results in an irreducible simply--connected manifold with odd intersection form, $E(n)_2$. We can always perform the surgery in a ``nucleus'' of $E(n)$, which is a cusp fiber together with the $(-n)$--section. Since $\Lambda$ is disjoint from the nucleus, it survives as an embedded torus in $E(n)_2$. The odd surface in $E(n)_2$ is in the nucleus where the torus surgery is performed, so is also disjoint from $\Lambda$. If we perform the $\Z_2$-construction to $E(n)_2$ along $\Lambda$, we obtain a minimal symplectic manifold $Y$ with $\pi_1(Y)\cong \Z_2$ and an odd intersection form. From computing the algebraic invariants, we deduce that the resulting manifold is an irreducible copy of $R_{2n-1,10n-1}$ for $n>1$.

\end{enumerate}

The above constructions cover all pairs $\{(c_1^2,\chi_h)|\, 0\leq c_1^2 \leq 8\chi_h -2 \text{ and } c_1^2 \text{ is even}\}$, except the point $(0,1)$. Next we address the case when $c_1^2$ is odd.

\underline{\underline{$c_1^2$ odd.}} As before, we summarize the construction given in \cite[Theorem 22]{ABBKP} which constructs irreducible copies of $\CPsum{a}{b}$ for this part of the geography plane, which we again denote by $X_{a,b}$. We will describe how to modify their construction to obtain an irreducible manifold $Y_{a,b}$ with the same algebraic invariants as $X_{a,b}$, but an order two fundamental group. Note that any $Y_{a,b}$ with $c_1^2$ odd will have odd signature, and so will automatically have an odd intersection form. For this reason, we do not need to make any remarks on the existence of surfaces with odd self--intersection. The following additional building blocks are needed, all of which are described in \cite{ABBKP}.\begin{itemize}
	\item $S_{1,1}$ is a manifold constructed by Gompf, with $c_1^2=1$ and $\chi_h=2$. $S_{1,1}$ is simply--connected and has a square zero torus with a trivial meridian.
	
	\item $X_{3,12}$ is simply--connected with algebraic invariants $c_1^2=7$ and $\chi_h=2$. It also has a square-zero torus whose meridian is trivial.
	
	\item $P_{5,8}$ has invariants $c_1^2=21$ and $\chi_h=3$. In this case, $\pi_1(P_{5,8})\cong \Z$, and there is a square zero torus $T$ with a trivial meridian, such that the induced map (of a push-off) $i_*\colon  \pi_1(T)\to \pi_1(P_{5,8}-T)$ is surjective.
\end{itemize}
All of the new building blocks above are symplectic and minimal. Now let $(c,\chi) \in \mathcal P = \{(c_1^2,\chi_h)|\, 0\leq c_1^2 \leq 8\chi_h -2 \text{ and } c_1^2 \text{ is odd}\}$. We consider various regions of $\mathcal P$, and find a way to realize $(c,\chi)$ in each of these regions.

\underline{$1 < c \leq 8\chi - 17$.} Let $(c',\chi')= (c-1,\chi-2)$. Then let $X_{a,b}$ be the manifold constructed in the $c_1^2$ even argument above corresponding to $(c',\chi')$. Then since $c'>0$, $X_{a,b}$ is constructed in case (1) or (2). If $X_{a,b}$ is built using case (1), then $X_{a,b}$ is a surgered telescoping triple $(Z,T_1,T_2)$. We may proceed by taking a fiber sum of $Z$ with $S_{1,1}$ along $T_1$ and applying $1/2$-surgery along $T_2$ as in Proposition \ref{prop: 1/2 surgery}. The fundamental group will be generated by $b_2$ with order two with the computation being the same as before. Fiber summing with $S_{1,1}$ increases $c_1^2$ and $\chi_h$ by $1$ and $2$ respectively, so the resulting manifold realizes $(c,\chi)$.
	
If $X_{a,b}$ is built using (2), then we fiber sum the corresponding $Z$ with $S_{1,1}$ along $T_2$ and $E(k)_{2,2}$ along $T_1$. Here the fundamental group computation will involve one more step. First, $\pi_1(S_{1,1}\#_{T_2} Z)$ is trivial. But then after fiber summing with $E(k)_{2,2}$, the manifold with fundamental group of order two, we can use Seifert-Van Kampen to check that $\pi_1(S_{1,1}\#_{T_2} Z \#_{T_1} E(k)_{2,2})\cong \Z_2$. This will follow from the fact that $\mu_F \subset E(k)_{2,2}-F$ is homotopic to $x_1x_2$, where $x_i$ are the generators of $\pi_1(E(k))_{2,2}$ given in \cite[p. 489]{GompfNuclei}. Again, the resulting manifold will have $c_1^2 = c'+1=c$ and $\chi_h = \chi'+2=\chi$.

\underline{$7< c \leq 8\chi-11$.} We take $(c',\chi')=(c-7,\chi-2)$. Then the argument repeats verbatim as the previous case, except that $S_{1,1}$ is replaced by $X_{3,12}$.

\underline{$21 < c \leq 8\chi -5$.} We repeat the previous two cases, now with $(c',\chi')=(c-21,\chi-3),$ using $P_{5,8}$ in place of $S_{1,1}$ and $X_{3,12}$. This time, more care is needed because $\pi_1(P_{5,8}) \cong \Z$. Recall that $T \hookrightarrow P_{5,8}$ is surjective on the fundamental group, and the inclusion $P_{5,8}-T^2 \hookrightarrow P_{5,8}$ induces an isomorphism. Let $t_1, t_2$ be generators of $\pi_1{(T)}$ so that $t_1$ generates $\pi_1{(P_{5,8})}$. Then $P_{5,8}\#_{T_1} Z$ will have a single generator if we identify $T^2 \subset P_{5,8}$ with $T_1\subset Z$ so that $t_1 \mapsto b_1$, where $b_1$ is nullhomotopic in $Z-(T_1\cup T_2)$. Then we proceed as in the previous cases by either performing Luttinger surgery on $T_2$ or taking a fiber sum with $E(k)_{2,2}$.

\underline{$c \in \set{1,7,21}$.} The manifolds in \cite[Theorem 22]{ABBKP} with $c_1^2\in \{1,7,21\}$ are constructed by fiber summing $S_{1,1}$ ($X_{3,12}$ or $P_{5,8}$, respectively) with $E(k)_{2,3}$, which is obtained from $E(k)$ by multiplicity $2$ and $3$ torus surgeries. If we mimic their construction but replace $E(k)_{2,3}$ with $E(k)_{2,2}$, we arrive at an irreducible manifold with the same Euler characteristic and signature as $X_{a,b}$, but having an order two fundamental group which is a consequence of Seifert-Van Kampen.

\underline{$21\leq c= 8\chi - 3$.} From the \cite[Remark 1]{ABBKP}, there exists a family of manifolds $P_k$ which are minimal and symplectic, with $(b_2^+,b_2^-)=(1+2k,4+2k)$ for all $k\geq 2$. Moreover, $\pi_1(P_k)\cong \Z$, and for each $k$, there is a homologically essential Lagrangian torus $T_k\subset  P_k$ such that its meridian is trivial and the induced map $i_*\colon \pi_1(T_k)\to \pi_1(P_k-T_k)\cong \pi_1(P_k)$ is surjective. So if $t_k$ generates $\pi_1(P_k-T_k)$, there is some loop $\lambda \in \pi_1(T_k)$ such that $i_*(\lambda)=t_k$. Then if we perform $1/2$ Luttinger surgery on $T_k$, along $\lambda$, the resulting manifold has fundamental group $\langle t_k|\, t_k^2\rangle$ which is  isomorphic to  $\Z_2$. The surgered $P_k$ realizes an irreducible copy of $R_{1+2k,4+2k}$.

\noindent {\bf Current progress of lattice points.} We recall that $\overline{R_{a,b}} \cong R_{b,a}$, which is why we restrict our attention to the region of the geography plane satisfying $b_2^+,b_2^- >0$, $c_1^2 \geq 0$, and $\sigma \leq 0 $. With this in mind, the argument above does not give a copy of $R_{a,b}$ for the following $(a,b)$-coordinates with $a$ odd:

\begin{itemize}
	\item $(1,8),\, (1,6),\, (1,4),$\\
	$(3,18),\, (3,16),\,  (3,14),\, (3,12),\, (3,10),\, (3,8),\, (3,6),$\\
	$(5,14),\, (5,12),\, (5,10).$
	
	\item $R_{2k-1, 2k}$ (signature $-1$ line)
	
	\item $R_{2k-1, 2k-1}$ (signature $0$ line)
	
	\item $R_{1,9}$ (first point on $c_1^2=0$ line)
	
\end{itemize}

The points $(a,b)=(3,16)$ and $(a,b)=(3,18)$ are covered by section \ref{sec:c_1_2_3}, and we leave the $\sigma=0$ and $\sigma=-1$ lines to \cite{baykur2024smooth}. Next we fill in more of these missing points.

\subsection{Reaching finitely many missing points}

Here is a lemma that will be useful for fundamental group computations. The proof is left as an exercise.

\begin{lemma} \label{lemma:meridian_product}
	Let $G_1\ast_H G_2$ be an amalgamated free product of $G_i$ over $\phi_i\colon H \to G_i$. Assume that $\phi_1(h_0) = 1$ and $\phi_2(h_0) = g$ for some $h_0 \in H$. Then $G_1\ast_H G_2 \cong G_1\ast_H (G_2/N(g))$ over $\phi_1$ and $\bar\phi_2$, where $\bar \phi_2$ is a composition of $\phi_2$ with the quotient map $q$, and $N(g)$ is the subgroup normally generated by $g\in G_2$.
\end{lemma}

\subsubsection{$R_{3,k}$ for $k\in\{6,8,10\}$.}

Let $X\in \{B,C,D\}$, and let $H$ be its embedded genus two surface. Then recall the manifold $Z'$ from section \ref{sec:b+=4}, which is obtained through $(a_2\times a_3,a_3,-1)$ Luttinger surgery on $T^4\#\CPbar{}$, where $a_i$ is a standard generator of $H_1(T^4;\Z)$. Then $Z'$ has a square zero genus two surface $\bar \Sigma$. We claim that $X\#_{H=\bar \Sigma} Z'$ is again a telescoping triple. To see this, we first regard $\pi_1(X\#_{H=\bar \Sigma} Z')$ as a quotient of $\pi_1(X-H)*\pi_1(Z'-\bar \Sigma)$, where various curves along $\bar \Sigma$ and $H$ are identified, as are $\mu_H$ and $\mu_{\bar \Sigma}$. As we saw in \ref{section:big_region}, $\mu_H=1\in \pi_1(X-H)$. Therefore, by Lemma \ref{lemma:meridian_product} we get $$\pi_1(X\#_{H=\bar \Sigma} Z')\cong \pi_1(X-H)\ast_G \pi_1(Z'-\bar \Sigma)/N(\mu_{\Sigma})\cong \pi_1(X-H)\ast_G\pi_1(Z'),$$ where $G=\pi_1(S^1\times \Sigma_2)$. Computing the fundamental group of $\pi_1(Z')$ is a simple exercise. Removing a standard $S^1\times S^1 \times \set{pt}^2$ from $T^4$ gives $T^2 \times \Sigma_1^1$. Removing $a_2\times a_3$ torus gives the fundamental group $$\langle a_1,a_2,a_3,a_4 \mid [a_i,a_j] \text{ for } \{i,j\}\neq \{1,4\} \rangle.$$ After gluing $T^2\times D^2$ according to the surgery, the fundamental group becomes 
$$\langle a_1,a_2,a_3,a_4 \mid [a_i,a_j] \text{ for } \{i,j\}\neq \{1,4\}, \medspace [a_1,a_4]a_3^{-1}\rangle,$$
which is isomorphic to 
$$\langle a_1,a_2,a_4 \mid [a_1,a_2], [a_2,a_4]\rangle.$$ Blowing up once does not affect the fundamental group, so this gives a presentation of $\pi_1(Z')$.

Next, we claim that the natural map $\pi_1(X-H)\to \pi_1(X-H)\ast_G\, \pi_1(Z')$ is an isomorphism. First, $\bar \Sigma \hookrightarrow Z'$ induces a surjection on the fundamental groups, so all the generators of $\pi_1(Z')$ are identified with generators of $\pi_1(X-H)$ in $\pi_1(X-H)\ast_G\pi_1(Z')$. So $\pi_1(X-H)\ast_G\pi_1(Z')$ is a quotient of $\pi_1(X-H)$ by the relations coming from $\pi_1(Z')$. But $\pi_1(X-H)\cong \Z\times \Z$ is abelian so commutator relations are automatic in $\pi_1(X-H)$. Therefore $\pi_1(X\#_{H=\bar \Sigma} Z')\cong \pi_1(X-H) \cong \pi_1(X)$. 

Since the tori $T_1$ and $T_2$ are disjoint from $H$, their inclusion-induced maps on $\pi_1$ are not affected, so $(X\#_{H=\bar \Sigma} Z',T_1,T_2)$ is a telescoping triple.  By Proposition \ref{prop: 1/2 surgery}, it is related to a manifold with an order two fundamental group through two Luttinger surgeries. Minimality follows from Usher's theorem and Proposition \ref{prop:Luttinger_minimal}, and the $w_2$-type follows from the parity of the signature. For $X=B,C,D$, this realizes an irreducible copy of $R_{3,6}$, $R_{3,8}$, $R_{3,10}$, respectively.

\subsubsection{$R_{5,k}$ for $k\in \{10,12,14\}$.}

Let $X\in \{B,C,D\}$. Similar to the previous construction, we will take a fiber sum with another manifold, argue that the result is a telescoping triple, and then apply Proposition \ref{prop: 1/2 surgery}. This time, we will not fiber sum with $X$ but rather a blow-up of $X$. Recall that each $X$ contains a square zero genus two surface $H$. We first claim that each $X$ also contains a square $-1$ torus $\Lambda$ which is disjoint from the telescoping tori $T_1$ and $T_2$, and intersects $H$ transversely at a point. Such a torus can be constructed in a variety of ways, but one simple way to find it is to recall from section \ref{section:big_region} that each $X$ is a fiber sum of two manifolds: $X=M_1\#_{\Sigma_2} M_2$. In all cases, each $M_i$ has a square zero torus intersecting $\Sigma_2$ once, and at least one of the $M_i$ has an exceptional sphere intersecting $\Sigma_2$ once. By tubing the square zero torus with the exceptional sphere in the fiber sum, we get $\Lambda$. This torus is always disjoint from the telescoping tori $T_1$ and $T_2$, as argued in \cite[Theorems 7, 11, 12]{ABBKP}. Now, by symplectically resolving $\Lambda \cup H$, we get a square $1$ genus three surface $F_3$ in $X$. After one blow-up, we obtain a relatively minimal pair $(\tilde X,\tilde F_3)$, where $\tilde X=X\#\CPbar{}$ and $\tilde F_3$ is the proper transform of $F_3$.

Next, let $W$ be the product space $T^2\times \Sigma_2$. Let $\Sigma$ be the resolution of $\{pt\}\times \Sigma_2\cup T^2\times \{pt\}$, which is a square two genus three surface. After two blow-ups, we get a relatively minimal pair $(\tilde W, \tilde \Sigma)$, where $\tilde W=W\#2\CPbar{}$ and $\tilde \Sigma$ is the proper transform of $\Sigma$. Now that $\tilde \Sigma\subset \tilde W$ and $\tilde F_3\subset \tilde X$ are both square zero genus three symplectic surfaces, we take the fiber sum $Y_X = \tilde X\#_{\tilde F_3=\tilde \Sigma} \tilde W$. Usher's theorem shows that $Y_X$ is minimal. When $X=B$, $Y_B$ has algebraic invariants $(b_2^+,b_2^-) = (5,10)$. When $X=C$, $(b_2^+,b_2^-) = (5,12)$, and when $X=D$, $(b_2^+,b_2^-) = (5,14)$. We claim that for all $X\in \{B,C,D\}$, $Y_X$ is a telescoping triple. This is similar to the $\pi_1$ computations in the previous subsection. We start with the amalgamated free product $\pi_1(Y_X)=\pi_1(\tilde X-\tilde F_3)\ast_G \pi_1(\tilde W-\tilde \Sigma)$, where $G=\pi_1(S^1\times\Sigma_3)$. Because of the exceptional sphere from the blow-up in $\tilde X$, $\mu_{\tilde F_3}=\mu_{\tilde \Sigma}=1$ in the free product, so the generators and relations coming from $\pi_1(\tilde W-\tilde \Sigma)$ reduce to those in $\pi_1(\tilde W)\cong \pi_1(T^2\times \Sigma_2)$. Each of the six generators of $\pi_1(T^2\times \Sigma_2)$ is the image of a standard generator of $\pi_1(\Sigma_3)$ in the map induced by $\tilde \Sigma\hookrightarrow \tilde W$, so in the fiber sum, each of these generators is identified with an element in $\pi_1(\tilde X-\tilde F_3)\cong \pi_1(\tilde X)$. This only adds commutator relations to $\pi_1(\tilde X-\tilde F_3)$, which does not change the fundamental group since $\pi_1(\tilde X-\tilde F_3)$ is abelian. Therefore $\pi_1(Y_X)\cong \pi_1(\tilde X-\tilde F_3)\cong \pi_1(X)$. Since $T_1$ and $T_2$ are disjoint from $\tilde F_3$, their inclusions induce the same maps on $\pi_1$ as before fiber-summing. Therefore $Y_X$ is a telescoping triple, so by Proposition \ref{prop: 1/2 surgery}, we obtain a manifold $Y''_X$ which has the same algebraic invariants as $Y_X$ but an order two fundamental group. The $w_2$--type of $Y''_X$ is confirmed by its signature being odd. This concludes our argument that when $X=B,C,D$, $Y''_X$ is an irreducible copy of $R_{5,10}, R_{5,12}, R_{5,14}$, respectively.


\section{Summary of progress with a nice graph, plus final constructions}

Let us put all of our constructions on a graph, combining our work for $m$ being both even and odd. In the previous sections, we found irreducible copies of $R_{m,n}$ satisfying $\sigma\leq 0$ and $c_1^2 = 4+5b_2^+-b_2^-\geq 0$. By reversing orientation, we realize an irreducible copy of $R_{n,m}$ for each $R_{m,n}$ that we find, so our geography is symmetric over the line $\sigma = 0$. We are therefore populating the geography plane between the lines $b_2^-\leq 4+5b_2^+$ and $b_2^+\leq 4+5b_2^-$, which are shown in red in Figure \ref{fig:plane}.

\begin{figure}[h]
    \centering
    \begin{tikzpicture}
        \begin{axis}[
            axis lines = middle,
            xlabel = $b_2^+$,
            ylabel = {$b_2^-$},
            xmin=0, xmax=15,
            ymin=0, ymax=15,
            xtick  = {1,2,...,15},
            ytick = {1,2,...,15},
            grid=both,
            minor tick num=0,
            ticklabel style = {font=\tiny}
        ]
    \addplot [
        domain=0:2.2, 
        samples=2,
        color=red,]
    {5*x + 4};
	\addplot [
	domain=0:15, 
	samples=2,
	color=red,]
	{0.2*x -0.8};

    \addplot[
    only marks,
    mark=*,
    mark options={scale=1, fill= white},
    color=black
    ]
    coordinates {
    	(1,4) (1,6) (1,8) (1,9)
    	(2,5) (2,7) (2,9) (2,10)
    	(3,12) (3,14)
    	(4,13) (4,15)
    	
    	(1,1) (2,2) (3,3) (4,4) (5,5) (6,6) (7,7)
    	
    	
    	(4, 1) (6, 1) (8, 1) (9, 1) (5, 2) (7, 2) (9, 2) (10, 2)(12, 3) (14, 3) (13, 4) (15, 4)

    };

        \end{axis}
    \end{tikzpicture}
    \caption{The graph with 31 points missing between the red lines}
    \label{fig:plane}
\end{figure}
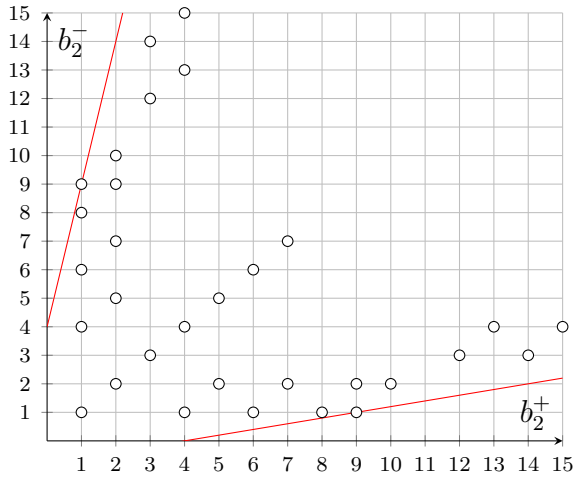

Irreducible copies of $R_{m,n}$ with $\sigma \in \{-1,0,1\}$ were found in \cite{baykur2024smooth} for all cases except the seven points $R_{m,m}$ for $1\leq m\leq 7$. Although the authors don't state irreducibility for their $\sigma = -1$ construction in Section 4 explicitly, this fact follows from their computation of SW-invariants in Lemma 4.8. Our constructions addresses the remaining part of the geography plane. To be precise, let $\mathcal S$ be the region in $\Z_{>0}\times \Z_{>0}$ given by $\mathcal S = \{(m,n): m-n<-1 \text{ and } n\leq 4+5m\}$. Let $\overline{\mathcal S}$ be the region obtained by reflecting $\mathcal S$ over $\sigma=0$. Within $\mathcal S$, there are $12$ points which remain missing. These are the $12$ coordinates in Figure \ref{fig:plane} strictly above the line $\sigma=0$.

\subsection{Remaining points.} We end by obtaining these final missing points in $\mathcal S$ using constructions from \cite{akhmedov2009exotic} and \cite{Bay_Kork}. In this case, we use similar methods to cover points with $b_2^+$ odd and even at the same time, which is why these arguments are an addendum to the work in Sections \ref{section:even_costructions} and \ref{sec:odd_constructions}.

\subsubsection{$R_{1,4}$, $R_{1,6}$.}

In \cite[Section 11]{akhmedov2009exotic}, the authors construct irreducible copies of $\CPsum{}{4}$ and $\CPsum{}{6}$. This construction can be modified to obtain irreducible copies of $R_{1,4}$ and $R_{1,6}$. In their constructions, $M(1/r,1/q)$ is a manifold obtained through two Luttinger surgeries on $T^4\# 2 \CPbar{}$, which contains a symplectically embedded genus two surfaace $\hat\Sigma$ with a trivial normal bundle. Similarly, $Z''(1/q,1/r)$ is obtained through two Luttinger surgeries on $T^4\#\CPbar{}$, and has a symplectically embedded square zero genus two surface $\bar \Sigma$. For a certain gluing map $\psi:\partial \nu(\hat \Sigma)\to \partial \nu(\bar \Sigma)$, the fiber sum $M(1,1)\#_\psi Z''(1,1)$ is an irreducible copy of $\CPsum{}{4}.$ We claim that $X:= M(1,1/2)\#_\psi Z''(1,1)$ is an irreducible copy of $R_{1,4}$. To see this, note that irreducibilty of $X$ follows from Usher's theorem. If $\pi_1(X)\cong \Z_2$, the $w_2$--type of $X$ will follow from its signature being odd, so the fundamental group is all that remains to check. For this, we rely on the following presentations of $\pi_1(M(1,1/2))$, which is provided by \cite[Lemma 12]{akhmedov2009exotic}.
\begin{align*}
	\pi_1(M(1,1/2)-\hat \Sigma) = \langle \beta_1,\beta_2,\beta_3,\beta_4 |& \beta_1=[\beta_2^{-1},\beta_4^{-1}], \beta_2^2 =[\beta_1^{-1},\beta_4],\\
	& [\beta_1,\beta_3]=[\beta_2,\beta_3]=[\beta_1,\beta_2]=[\beta_3,\beta_4]=1 \rangle.
\end{align*} Moreover, $\pi_1(M(1,1/2)-\hat \Sigma)\cong \pi_1(M(1,1/2))$ because $\mu_{\hat \Sigma}$ is trivial in $\pi_1(M(1,1/2)-\hat \Sigma)$. Then by Lemma \ref{lemma:meridian_product}, $\pi_1(X) = \pi_1(M(1,1/2))\ast_G \pi_1(Z''(1,1)-\bar \Sigma)/N(\mu_{\bar\Sigma})=\pi_1(M(1,1/2))\ast_G \pi_1(Z''(1,1))$, where $G=\pi_1(S^1\times \Sigma_2)$. Note that $Z''(1,1)$ is a blow up of $M(1,1)$, and the same surgered tori up to index relabeling, so we again use \cite[Lemma 12]{akhmedov2009exotic}:
\begin{align*}
	\pi_1(Z''(1,1)) \cong \langle \alpha_1,\alpha_2,\alpha_3,\alpha_4 & \mid
	\alpha_3=[\alpha_1^{-1},\alpha_4^{-1}],\alpha_4=[\alpha_1,\alpha_3^{-1}],\\
	&[\alpha_1, \alpha_2] = [\alpha_2,\alpha_3]=[\alpha_2,\alpha_4]=[\alpha_3,\alpha_4]=1 \rangle.
\end{align*}

Next, we review how $M(1,1/2) - \nu(\hat \Sigma)$ and $Z''(1,1)-\nu(\bar{\Sigma})$ are glued under $\psi$. The map $\psi$ sends the standard generators of $\pi_1(\hat \Sigma)$ to the corresponding standard generators of $\pi_1(\bar\Sigma)$ \cite[Section 11]{akhmedov2009exotic}. Then the inclusions $\bar \Sigma\hookrightarrow Z''(1,1)$ and $\hat \Sigma \hookrightarrow M(1,1/2)$ map the standard generators as follows:
\begin{align*}
	\bar{a}_1^{\vert\vert} \mapsto \alpha_1 & \hspace{20 pt}
	\bar{b}_1^{\vert\vert} \mapsto \alpha_2 & \hspace{20 pt}
	\bar{a}_2^{\vert\vert} \mapsto \alpha_3^2 & \hspace{20 pt}
	\bar{b}_2^{\vert\vert} \mapsto \alpha_4 \\
	\hat{a}_1^{\vert\vert} \mapsto \beta_1 & \hspace{20pt}
	\hat{b}_1^{\vert\vert} \mapsto \beta_2 & \hspace{20 pt}
	\hat{a}_2^{\vert\vert} \mapsto \beta_3 & \hspace{20 pt}
	\hat{b}_2^{\vert\vert} \mapsto \beta_4. 
\end{align*} So from $\psi$, we get the identifications $\alpha_i = \beta_i$ for $i\neq 3$ and $\alpha_3^2 = \beta_3$. After applying Seifert-Van Kampen, $\pi_1(X)$ is generated by $\alpha_i$. The relations, in addition to the presentation of $Z''(1,1)$ above, are $\alpha_1=[\alpha_2^{-1},\alpha_4^{-1}], \alpha_2^2 =[\alpha_1^{-1},\alpha_4], [\alpha_1,\alpha^2_3]=[\alpha_2,\alpha^2_3]=[\alpha_1,\alpha_2]=[\alpha^2_3,\alpha_4]=1$. Because $\alpha_2$ and $\alpha_4$ commute, $\alpha_1$ is trivial. But then $\alpha_3= [\alpha_1^{-1},\alpha_4^{-1}]= 1$. Similarly, $\alpha_4=1$, so all commutators vanish. The only generator left is $\alpha_2$, and only relation left is $\alpha_2^2$. So we conclude that $\pi_1(X)\cong \Z_2$, making $X$ an irreducible copy of $R_{1,4}$.\\

The construction of an irreducible $R_{1,6}$ is similar. This time, we take a fiber sum of $Y:=(T^2\times S^2)\#4\CPbar{}$ and $Z''(1/2,1)$ along symplectically--embedded genus two surfaces $\tilde \Sigma\subset Y$ and $\bar \Sigma \subset Z''(1/2,1)$. In this case, $\mu_{\tilde \Sigma}=1\in \pi_1(Y-\tilde \Sigma)$, and so by Lemma \ref{lemma:meridian_product}, \[
\pi_1(Y\#_{\tilde \Sigma=\bar \Sigma} Z''(1/2,1)) = \pi_1(Y) \ast_G \pi_1(Z''(1/2,1)),\]
where $G=\pi_1(S^1\times \Sigma_2)$. This time,
\begin{align*}
	\pi_1(Z''(1/2,1)) \cong \langle \alpha_1,\alpha_2,\alpha_3,\alpha_4 \mid & 
	\alpha_3^2=[\alpha_1^{-1},\alpha_4^{-1}],\alpha_4=[\alpha_1,\alpha_3^{-1}],\\
	&[\alpha_1, \alpha_2] = [\alpha_2,\alpha_3]=[\alpha_2,\alpha_4]=[\alpha_3,\alpha_4]=1 \rangle.
\end{align*}
On the other hand, $\pi_1(Y)$ is a free abelian group with standard generators $x$ and $y$. The embedding of $\tilde{\Sigma}$ in $Y$ induces map on generators mapping on $\set{\tilde{a_i}^{\pm 1},\tilde{b_i}^{\pm 1}}$ to $\set{x,y}$ (The exact signs of the identifications won't make a difference here, see \cite[Section 11]{akhmedov2009exotic} for the precise gluing.) We can assume that this induces $\alpha_1^{\pm 1} = x = \alpha_3$, and $\alpha_2^{\pm 1}=y=\alpha_4$ identifications in the amalgamated product. This results in the following $\pi_1$ presentation for the fiber sum: \begin{align*}
	\langle \alpha_1,\alpha_2,\alpha_3, \alpha_4, x, y \mid & [\alpha_1, \alpha_2] = [\alpha_2,\alpha_3]=[\alpha_2,\alpha_4]=[\alpha_3,\alpha_4]=[x,y]=1, \\
	& \alpha_3^2=[\alpha_1^{-1},\alpha_4^{-1}],\alpha_4=[\alpha_1,\alpha_3^{-1}], x=\alpha_1=\alpha_3^{\pm 1}, y = \alpha_2=\alpha_4^{\pm 1}\rangle ,\\
	&\cong \langle \alpha_1, \alpha_2, \alpha_3, \alpha_4 \mid [\alpha_i,\alpha_j]=1 , \alpha_3^2=1, \alpha_4=1, \alpha_1 = \alpha_3^{\pm 1}, \alpha_2 = \alpha_4^{\pm 1} \rangle \\
	& \cong \langle \alpha_3 \mid \alpha_3^2 \rangle.
\end{align*} Hence the fiber sum is an irreducible copy of $R_{1,6}$.

\subsubsection{$R_{2,5}$, $R_{2,7}$.} In this case, we apply the $\Z_2$--construction to the irreducible copies of $\CPsum{}{4}$ and $\CPsum{}{6}$ along embedded genus two surfaces given in \cite[Section 11]{akhmedov2009exotic}. Let $X=M(1,1)\#_\psi Z''(1,1)$, and let $\bar\Sigma^{||}$ be a parallel copy of $\bar \Sigma$ in $Z''(1,1)$, which survives as an embedded surface in $X$. We claim that the $\Z_2$ construction of $X$ along $\bar\Sigma^{||}$ is an irreducible copy of $R_{2,5}$. For this, the irreducibility, $w_2$--type, and $b_2^{\pm}$ are all immediate. The only thing to check is $\pi_1$, and to do this we compute that $\pi_1(X-\bar\Sigma^{||})=1$. First, write $\pi_1(X-\bar\Sigma^{||})$ as \[
\pi_1(Z''(1,1)-(\bar\Sigma \cup \bar\Sigma^{||}))\ast_G \pi_1(M(1,1)-\hat \Sigma),
\] where $G=\pi_1(S^1\times \Sigma_2)$ Since $\mu_{\hat\Sigma}=1\in \pi_1(M(1,1)-\hat \Sigma)$, by Lemma \ref{lemma:meridian_product} this reduces to \[
\pi_1(Z''(1,1)-\bar\Sigma^{||})\ast_G \pi_1(M(1,1)).
\] The presentation for $\pi_1(M(1,1))$ is: \begin{align*}\langle \beta_1,\beta_2,\beta_3,\beta_4 |& \beta_1=[\beta_2^{-1},\beta_4^{-1}], \beta_2 =[\beta_1^{-1},\beta_4],\\
	& [\beta_1,\beta_3]=[\beta_2,\beta_3]=[\beta_1,\beta_2]=[\beta_3,\beta_4]=1 \rangle.
\end{align*}
The generators for $\pi_1(Z''(1,1)-\bar\Sigma^{||})$ are $\alpha_1,\alpha_2,\alpha_3,\alpha_4,g_1,\cdots g_3$, where each $g_i$ is a conjugate of $[\alpha_3,\alpha_4]$. Some of the relations in $\pi_1(Z''(1,1)-\bar\Sigma^{||})$ are given in \cite[Theorem 6]{akhmedov2009exotic}. On the other hand, $\pi_1(M(1,1))$ is generated by $\beta_1,\beta_2,\beta_3,\beta_4$ with relations given in \cite[Lemma 12]{akhmedov2009exotic}.  From the gluing, the following generators are identified in the amalgamation:\[
\alpha_i=\beta_i \text{ for } i\neq 3,\,\,\, \alpha_3^2=\beta_3.
\]

Now the generators of the presentation in the amalgamation is reduced to the group generated by $\alpha_i$ and $g_j$. Since $\beta_1= [\beta_2^{-1}, \beta_4^{-1}]$ and $[\alpha_2,\alpha_4]=1$ we get that $\alpha_1=\beta_1=1$. The rest quickly reduces to the identity as well. Hence the $\Z_2$--construction of $X$ along $\bar \Sigma^{||}$ will have order two $\pi_1$, and so is an irreducible copy of $R_{2,5}$.\\

The process for getting an irreducible copy of $R_{2,7}$ is similar, and in fact a bit easier. For this, we perform the $\Z_2$--construction to $Y\#_{\tilde \Sigma = \bar \Sigma} Z''(1,1)$ along $\bar\Sigma^{||}\subset Z''(1,1)$. Similar to the previous case, since $\mu_{\tilde{\Sigma}}$ is trivial,
$$\pi_1(Y-\tilde{\Sigma})\ast_G \pi_1(Z''(1,1)-\bar{\Sigma} - \Sigma^{||})\cong \pi_1(Y)\ast_G \pi_1(Z''(1,1)- \Sigma^{||}).$$
This group is generated by $x,y,\alpha_i, g_j$ and the gluing induces the following identifications: $x=\alpha_1=\alpha_3^{\pm 1}, y=\alpha_2=\alpha_4^{\pm 1}$. Since $x$ and $y$ commute, all $[\alpha_i,\alpha_j]$ must be trivial. Then from the presentation of $\pi_1(Z''(1,1)- \Sigma^{||})$, $\alpha_4=\alpha_3 =1$. It immediately follows that all other generators are trivial. So the complement of the meridian in $Y\#_{\tilde \Sigma = \bar \Sigma} Z''(1,1)$ is simply-connected. Hence, the $\Z_2$--construction is an irreducible copy of $R_{2,7}.$

\subsubsection{$R_{1,8}$, $R_{1,9}$, $R_{2,9}$, $R_{2,10}$.}

In \cite[Theorem 16]{Bay_Kork}, the authors construct minimal genus two Lefschetz fibrations whose total spaces are exotic copies of $\CPsum{}{8}$ and $\CPsum{}{9}$. They show in Remark 18 that their construction can also give minimal Lefschetz fibrations on manifolds homeomorphic to $R_{1,8}$ and $R_{1,9}$. Both of these Lefschetz fibrations have separating vanishing cycles, which implies they have odd intersection forms. This confirms that these exotic copies have the correct $\omega_2$--type.

One can also perform the $\Z_2$--construction to their irreducible copies of $\CPsum{}{8}$ and $\CPsum{}{9}$ along regular fibers of the fibration to get irreducible copies of $R_{2,9}$ and $R_{2,10}$. See \cite[Section 5.1]{Bay_Kork}, where $\tilde{W_2}, \tilde{W_3}$ correspond to the positive factorizations of the above Lefschetz fibrations on exotic $\CPsum{}{8}$ and $\CPsum{}{9}$. This follows as long as the regular fibers have trivial meridians, which holds because both fibrations have sections. To see that the $\Z_2$-construction over $\CPsum{}{9}$ has the correct $\omega_2$-type, note that its double cover is a Lefschetz fibration containing a separating vanishing cycle, hence is odd.

\subsubsection{$R_{3,12}$, $R_{3,14}$, $R_{4,13}$, $R_{4,15}$.}

In \cite[Theorem 20]{Bay_Kork}, the authors construct minimal genus two Lefschetz fibrations whose total spaces are exotic copies of $\CPsum{}{12}$ and $\CPsum{}{14}$. By performing the $\Z_2$-construction to each of them along a regular fiber, we obtain irreducible copies of $R_{4,13}$ and $R_{4,15}$. Note that again the meridian of a regular fiber is trivial in the complement because the Leschetz fibrations admit sections, so the fundamental group of the $\Z_2$--construction is indeed $\Z_2$.

We can do something similar to what they do in \cite[Remark 18]{Bay_Kork} to change the irreducible copies of $\CPsum{}{12}$ and $\CPsum{}{14}$ into something with order two $\pi_1$. This involves replacing the conjugating map $\phi$ in the factorizations for the Lefschetz fibrations on exotic copies of  $\CPsum{}{12}$ and $\CPsum{}{14}$ to a new conjugating map, $\phi'=t_{c_1}^{-2}t_{c_4}$. By \cite[Remark 18]{Bay_Kork}, the fundamental group of the Lefschetz fibration corresponding to $W_1^{\phi'}W_1$ is $\Z_2$. Consider Lefschetz fibration $X'$ with positive factorization $\hat{W_3'} = W_1^{\phi'} W_1 W_1$. It has the same vanishing cycles as $W_1^{\phi'}W_1$, so the fundamental group is again $\Z_2$. In addition, it has the same signature and Euler characteristic as $\hat{W_3} =  W_1^{\phi} W_1 W_1$, which according to \cite[Theorem 20]{Bay_Kork} is homeomorphic to $\CPsum{3}{12}$. Then $X'$ must be an irreducible copy of $R_{3,12}$.

For $R_{3,14}$, we will use same remark but one step more extensively. Namely, they claim that the fundamental group of both $\tilde{W'_1}$ and $\tilde{W'_3}$ has an (abelian) presentation $\langle b_2 \mid 2b_2 \rangle$, for some curve $b_2$ on the fiber. Here $\tilde{W_1} = W_1^{\phi'}W_1$ and $\tilde{W'_3}= W_2^{\phi'}W_2$. If we concatenate these two factorizations, we get $W_1^{\phi'}W_1W_2^{\phi'}W_2$ with the same fundamental group $\langle b_2 \mid 2b_2 \rangle$. To see this, just observe that the Lefschetz fibrations for all three of these positive factorizations have the same vanishing cycles and admit sections. Now consider Lefschetz fibration $Y'$ with positive factorization $W_2^{\phi'}W_2W_1$. Its vanishing cycles contain vanishing cycles of $W_2^{\phi'}W_2$ and is contained in the set of vanishing cycles of $W_1^{\phi'}W_1W_2^{\phi'}W_2$. So the fundamental group of $Y'$ gets ``squeezed'' between the other two with the same presentation $\langle b_2 \mid 2b_2 \rangle$. Hence, $\pi_1(Y')\cong \Z_2$. To see that $Y'$ has the same signature and Euler characteristic as $R_{3,14}$, note its similarity to $\hat{W_3}$ which has the correct invariants according to the \cite[Theorem 20]{Bay_Kork}. So $Y'$ is an irreducible copy of $R_{3,14}$.

\subsection{Conclusion}

In total, there are seven points still missing in $\mathcal S$, all with signature zero. This concludes our proof of Theorem \ref{thm:main_thm}.\\

\begin{rmk}
	In many of our constructions, there are opportunities to address the {\it botany problem}, where one constructs an infinite family of irreducible smooth, manifolds that are homeomorphic to $R_{m,n}$, but pairwise non-diffeomorphic. For instance, when our irreducible copy of $R_{m,n}$ contains a null-homologous torus, it's often possible to obtain an infinite family of distinct smooth structures using Luttinger surgery, as in \cite[Theorem 1]{FPS}. In other cases, when there's an embedded torus with a cusp neighborhood, as is often the case when fiber summing with an elliptic surface, one can solve the botany problem using {\it knot surgery} \cite{FS_knots}, \cite{Park_irreducible}. Both constructions produce infinite families of irreducible smooth structures that are homeomorphic with distinct Seiberg--Witten invariants. Moreover, the surgeries can often be done equivariantly, as in \cite{baykur2024smooth}, and so are compatible with the $\Z_2$--construction. This gives a fairly straightforward approach to solving the botany problem for many of the coordinates in our $\Z_2$--geography. That said, careful work needs to be done to verify that a torus meeting the criteria of \cite{FPS} or \cite{FS_knots} exists for each point in the current geography.
\end{rmk}

\begin{rmk} 
There are a few possibilities for extending the irreducible geography to other cyclic fundamental groups. The authors in \cite{baykur2024smooth} introduce a \textit{circle sum} construction to fill in the \textit{reducible} geography of smooth 4-manifolds with cyclic fundamental group of order $4m+2$. Proposition 6.2 in their paper describes a specific case when the circle sum gives an irreducible manifold. Although this idea has potential to give irreducible copies of other lattice points with $\pi_1=4m+2$, showing irreducibility relies on $SW$ computations which do not generalize easily to our constructions.

Another way to extend the irreducible geography to other cyclic fundamental groups is by changing the coefficient of a torus surgery from $1/2$ to $1/n$, or replacing $E(k)_{2,2}$ with $E(k)_{n,n}$. This would work for odd $b_2^+$ coordinates. It can be done for even $b_2^+$ too, by performing the $\Z_2$--construction after changing the coefficient of a torus surgery. That said, one would need to check that the double has the right cyclic fundamental group after making these adjustments, which is not always guaranteed. And even then, some geography regions are missed entirely by these ideas, such as coordinates with $c_1^2\leq 4$.

\end{rmk}

\bibliography{refs} 

\bibliographystyle{alpha} 

\end{document}